\documentclass{article}

\usepackage[utf8]{inputenc}
\usepackage{amsthm}
\usepackage{array}
\usepackage{hyperref}

\usepackage{amsfonts}
\usepackage{times}
\usepackage{verbatim}

\usepackage{mathtools}
\usepackage{amssymb}
\usepackage{multirow}
\usepackage{pifont}
\usepackage{adjustbox}

\usepackage{epsfig}
\usepackage{graphicx}
\usepackage{subcaption}
\usepackage{color}
\usepackage{algorithm}
\usepackage{algpseudocode}
\usepackage{booktabs}

\newcommand{\ignore}[1]{}
\newcommand*\Let[2]{\State #1 $\gets$ #2}
\algrenewcommand\algorithmicrequire{\textbf{Input:}}

\newcommand*\rot{\rotatebox{90}}

\mathtoolsset{showonlyrefs}
\DeclareMathOperator*{\argmin}{arg\,min}
\DeclareMathOperator*{\rank}{rank}
\DeclareMathOperator*{\sparsity}{sp}
\DeclareMathOperator*{\conv}{conv}

\newtheorem{lemma}{Lemma}
\newtheorem{theorem}{Theorem}

\newtheorem{assumption}{Assumption}

\title{Fast Multilevel Algorithms for Compressive Principle Component Pursuit}

\author{Vahan Hovhannisyan\thanks{Department of Computing, Imperial College London, 180 Queen's Gate, SW7 2AZ London, UK
(\href{mailto:vh13@imperial.ac.uk}{vh13@imperial.ac.uk}, \href{mailto:i.panagakis@imperial.ac.uk}{vh13@imperial.ac.uk}, \href{mailto:p.parpas@imperial.ac.uk}{p.parpas@imperial.ac.uk}, \href{mailto:s.zafeiriou@imperial.ac.uk}{s.zafeiriou@imperial.ac.uk}).} \and Yannis Panagakis\footnotemark[2] \and Panos Parpas\footnotemark[2] \and Stefanos Zafeiriou\footnotemark[2]}

\begin{document}

\maketitle

\begin{abstract}
Recovering a low-rank matrix from highly corrupted measurements arises in compressed sensing of structured high-dimensional signals (e.g., videos and hyperspectral images among others). Robust principal component analysis (RPCA), solved via principal component pursuit (PCP), recovers a low-rank matrix from sparse corruptions that are of unknown value and support by decomposing the observation matrix into two terms: a low-rank matrix and a sparse one, accounting for sparse noise and outliers. In the more general setting, where only a fraction of the data matrix has been observed, low-rank matrix recovery is achieved by solving the compressive principle component pursuit (CPCP). Both PCP and CPCP are well-studied convex programs, and numerous iterative algorithms have been proposed for their optimisation. Nevertheless, these algorithms involve singular value decomposition (SVD) at each iteration, which renders their applicability challenging in the case of massive data. In this paper, we propose a multilevel approach for the solution of PCP and CPCP problems. The core principle behind our algorithm is to apply SVD in models of lower-dimensionality than the original one and then lift its solution to the original problem dimension. We show that the proposed algorithms are easy to implement, converge at the same rate but with much lower iteration cost. Numerical experiments on numerous synthetic and real problems indicate that the proposed multilevel algorithms are several times faster than their original counterparts, namely PCP and CPCP.
\end{abstract}

\section{Introduction}

Low-rank matrix recovery is a cornerstone in data analysis and dimensionality reduction, with the principal component analysis (PCA) \cite{hotelling1933analysis} being the most widely employed method for this task. However, PCA is fragile to the presence of gross, non-Gaussian, noise and outliers, and the estimated low-rank subspace may be arbitrarily away from the true one; even when a small fraction of the data is corrupted \cite{huber2011robust}. To alleviate this drawback, robust PCA (RPCA) models have been proposed \cite{candes2011robust}. RPCA aims to recover a low-rank matrix from sparse corruptions that are of unknown value and support by decomposing the observation matrix ($\mathbf{D}$) into two parts, namely $\mathbf{D} = \mathbf{L} + \mathbf{S}$. The first part is a low-rank matrix ($\mathbf{L}$) and the second part is a sparse matrix ($\mathbf{S}$) that accounts for sparse noise and outliers. In case of partially observed data, the RPCA model is extended to consider the following decomposition \cite{wright2013compressive}: $\mathbf{D}\doteq\mathcal{P}_\mathcal{Q}[\mathbf{M}]=\mathcal{P}_\mathcal{Q}[\mathbf{L}+\mathbf{S}]$, where $\mathcal{Q}\subseteq\mathbb{R}^{m\times n}$ is a linear subspace and $\mathcal{P}_\mathcal{Q}$ denotes the projection operator onto that subspace. The aforementioned low-rank matrix recovery models have profound impact in visual data analysis and computer vision applications such as image denoising \cite{candes2011robust}, background subtraction, image alignment \cite{peng2012rasl}, texture recovery \cite{zhang2012tilt}, deformable models \cite{sagonas2014raps}, face frontalization \cite{sagonas2016robust}, and structure from motion \cite{angst2011generalized}, to mention but a few examples.

A natural approach to estimate the low-rank and sparse components in the above mentioned models is to minimise the rank of $\mathbf{L}$ and the number on non-zero entries of $\mathbf{S}$, measured by the $\ell_0$ quasi norm \cite{candes2011robust}. Unfortunately, both rank and $\ell_0$-norm minimisation are NP-hard \cite{vandenberghe1996semidefinite, natarajan1995sparse}. The nuclear- and the $\ell_1$- norms are typically adopted as convex surrogates to rank and $\ell_0$- norm, respectively yielding the convex \textit{principle component pursuit} (PCP) \cite{candes2011robust} and \textit{compressive principle component pursuit} (CPCP) programs \cite{wright2013compressive}.

Common solvers for the convex PCP and CPCP models include: Iterative Thresholding (IT) \cite{daubechies2004iterative}, Accelerated Proximal Gradient (APG)  \cite{parikh2013proximal}, Augmented Lagrange Multipliers (ALM) \cite{lin2010augmented} and Linearized Augmented Lagrangian method \cite{yang2013linearized}. However, all these solvers exhibit significant computational drawbacks. In particular, at each iteration, they require computing several (not necessarily all) singular values and vectors of a large matrix, which is computationally expensive.

There have been several attempts to reduce the computational cost of large nuclear-norm regularised optimisation problems. Concretely, \cite{liu2012active} proposed to reduce the dimensions of the problem by factorising the low-rank matrix 
as the product of two smaller matrices, resulting in a non-convex problem which is solved by employing the augmented Lagrangian alternating direction method. Another very popular approach for reducing the dimensions of large-scale problems is to create smaller sub-problems by applying randomised techniques \cite{allen2016even, drineas2006fast, liu2014linear, musco2015randomized, oh2015fast}. More recently, the Frank-Wolfe (FW) algorithm has regained popularity for solving large-scale problems due to its extremely low iteration complexity. Specifically, the Frank-Wolfe Thresholding (FW-T) method proposed in \cite{mu2016scalable} is arguably the most efficient method for solving large (C)PCP problems. Nevertheless, FW type of methods require significantly more iterations to converge, and hence they can be impractically time consuming.
 
In this paper, motivated by the recent advances in multilevel optimisation algorithms \cite{ho2016multilevel, hovhannisyan2016magma, javaherian2017multi, nash2000multigrid,  campos2018multigrid, parpas2017multilevel}, we propose a simple, yet generic and very effective multilevel approach for significantly reducing computational costs for many problems that require solving nuclear norm based oracles, including RPCA models such as the PCP and CPCP. The core of our proposed methodology is to construct and solve lower dimensional (coarse) models for each optimisation oracle and then lift its solution to the original problem dimension. We show that using appropriately chosen restriction and prolongation operators result in algorithms that converge to an (approximate) solution of the original problem. We apply the proposed multilevel approach on two state-of-the-art algorithms, namely the Inexact Augmented Lagrange Multiplier Method (IALM) for the PCP problem \cite{lin2010augmented}, and the Frank-Wolfe Thresholding method for the more general CPCP model \cite{mu2016scalable}.
In particular, our main contributions are:
\begin{itemize}
\item In section \ref{sec:ml-ialm} we show that the proposed multilevel IALM algorithm converges to an approximate solution, and in sections \ref{sec:video} and \ref{sec:shadow} we show that in practice it is several times faster than the standard IALM. 
\item The first provably convergent variant of IALM with approximate updates (Theorem \ref{th:mlialm-convergence}).
\item In section \ref{sec:ml-fwt} we show that for the FW-T method we prove that its multilevel variant converges in function value with the same worst-case iteration complexity (Theorem \ref{th:mlfw-convergence}). However, in sections \ref{sec:video} and \ref{sec:shadow} we show that in practice each iteration of the multilevel algorithm is up to two times cheaper.
\item Numerical tests in both synthetic and real data indicate that the proposed multilevel variants solve large-scale problems two times faster than their standard counterparts (Section \ref{sec:experiments}). 
\end{itemize}

\textit{Notation.} Throughout the paper, scalars are denoted by lower-case letters, vectors (matrices) are denoted by lowercase (upper-case) boldface letters, e.g. $\mathbf{x}$ ($\mathbf{X}$). $\mathbf{I}$ denotes the identity matrix with appropriate dimension. The $\ell_1$ and $\ell_2$ norms of a vector $\mathbf{x}$ are defined as $\Vert \mathbf{x}\Vert_1=\sum_{i}\vert x_i\vert$, where $\vert\cdot\vert$ denotes the absolute value operator, and $\Vert \mathbf{x}\Vert_2=\sqrt{\sum_i x_i^2}$, respectively. The matrix $\ell_1$ norm is defined as $\Vert \mathbf{X}\Vert_1 = \sum_i\sum_j \vert x_{ij}\vert$. The Frobenius norm is defined as $\Vert\mathbf{X}\Vert_F=\sqrt{\sum_i\sum_j x_{ij}^2}$, and the nuclear norm of $\mathbf{X}$ (i.e., the sum of singular values of a matrix) is denoted by $\Vert \mathbf{X}\Vert_*$. The $l$-th largest singular value of matrix $\mathbf{X}$ is denoted as $\sigma_l(\mathbf{X})$. In algorithm pseudocodes we use $\mathbf{X}^{(k)}$ ($u_k$) to denote the value of matrix $\mathbf{X}$ (scalar $u$) at iteration $k$.

\section{Compressive Principle Component Pursuit and Robust PCA}

In this section, we give a formal presentation of the CPCP \cite{wright2013compressive} model. Let 
$\mathcal{Q}\subseteq\mathbb{R}^{m\times n}$ be a linear subspace spanned by the set of sensing matrices, and $\mathcal{P}_{\mathcal{Q}}$ denote the projection operator onto that subspace. In many applications $\mathcal{Q}$ is the subset of observed values of data $\mathbf{D}$. Then the problem is to find a low rank matrix $\mathbf{L}^\star$ and a sparse matrix $\mathbf{S}^\star$ such that the $\mathcal{P}_{\mathcal{Q}}[\mathbf{L}^\star+\mathbf{S}^\star] = \mathcal{P}_{\mathcal{Q}}[\mathbf{D}]$. The problem can be written as a convex unconstrained minimisation problem as follows:

\begin{equation} \label{eq:cpcp}
    \min_{\mathbf{L},\mathbf{S}} \frac{1}{2} \Vert\mathcal{P}_{\mathcal{Q}}[\mathbf{L}+\mathbf{S}-\mathbf{D}]\Vert_F^2+\lambda_L\Vert\mathbf{L}\Vert_* + \lambda_S\Vert\mathbf{S}\Vert_1,
\end{equation}
where $\lambda_L$ and $\lambda_S$ are positive penalising coefficients.

An important special case of (\ref{eq:cpcp}) is the well known PCP problem robust principle component analysis (RPCA), where $\mathcal{Q}$ is the entire space $\mathbb{R}^{m\times n}$ and $\mathcal{P}_\mathcal{Q}$ is the identity, i.e. all values of $\mathbf{D}$ have been observed. The problem in this case can be formulated to represent the input data matrix $\mathbf{D}\in\mathbb{R}^{m\times n}$ as a sum of a low rank matrix $\mathbf{L}^\star$ and a sparse matrix $\mathbf{S}^\star$. This can be exactly solved via the following convex constrained optimisation problem:
\begin{equation} \label{eq:rpca}
    \min_{\mathbf{L}, \mathbf{S}}  \Vert \mathbf{L}\Vert_* + \lambda \Vert \mathbf{S}\Vert_1 , \quad subject\; to \quad \mathbf{D}=\mathbf{L}+\mathbf{S},
\end{equation}
where $\lambda>0$ is a weighting parameter.

Before proceeding in the presentation of the proposed multilevel algorithms for the PCP and CPCP problems, we will provide an overview of the most widely adopted solvers for these problems

\subsection{Inexact ALM for Robust PCA}

We start with the simpler PCP problem for RPCA. A classical approach for solving (\ref{eq:rpca}) is by minimising its augmented Lagrangian defined as
\begin{equation} \label{eq:aug-lagr}
    \mathcal{L}(\mathbf{L},\mathbf{S},\mathbf{Y},\mu) = \Vert \mathbf{L}\Vert_* + \lambda\Vert \mathbf{S}\Vert_1 + \langle\mathbf{Y},\mathbf{D}-\mathbf{L}-\mathbf{S}\rangle + \frac{\mu}{2} \Vert \mathbf{D}-\mathbf{L}-\mathbf{S}\Vert_F^2,
\end{equation}
where $\mathbf{Y}\in\mathbb{R}^{m\times n}$ is the Lagrangian variable and $\mu>0$ is a penalty parameter. The convex optimization model in (\ref{eq:aug-lagr}) can be solved via alternating directions method. The latter method first solves the problem for each primal variable $\mathbf{L}$ and $\mathbf{S}$ separately for a fixed $\mathbf{Y}$. The dual variable $\mathbf{Y}$ is updated according to a linear rule and $\mu_k$ is chosen as an increasing sequence at each iteration \cite{lin2010augmented}. The method is computationally attractive because each resulting subproblem has a closed form solution. The resulting procedure was dubbed Inexact ALM (IALM) in \cite{lin2010augmented} and is formally given here in Algorithm \ref{alg:ialm}. 

\begin{algorithm}
  \caption{Inexact ALM (IALM) \label{alg:ialm}}
  \begin{algorithmic}[1]
    \Require{$\mathbf{D}, \mathbf{S}^{(0)}, \mathbf{Y}^{(0)} \in\mathbb{R}^{m\times n}$; $\mu_0>0$}
    \For{$k \gets 1 \textrm{ to } ...$}
      \State // Solve $\mathbf{L}^{(k+1)}=\argmin\limits_{\mathbf{L}} \mathcal{L} (\mathbf{L},\mathbf{S}^{(k)},\mathbf{Y}^{(k)},\mu_k)$
      \Let{$\mathbf{M}^{(k)}$}{$\mathbf{D}-\mathbf{S}^{(k)} + \mu_k^{-1}\mathbf{Y}^{(k)}$}
      \Let{$(\mathbf{U}, \mathbf{\Sigma}, \mathbf{V})$}{SVD($\mathbf{M}^{(k)}$)}
      \Let{$\mathbf{L}^{(k+1)}$}{$\mathbf{U}\mathcal{S}_{\mu_k^{-1}}[\mathbf{\Sigma}]\mathbf{V}^\top$}
      \State // Solve $\mathbf{S}^{(k+1)}=\argmin\limits_{\mathbf{S}} \mathcal{L}(\mathbf{L}^{(k+1)},\mathbf{S},\mathbf{Y}^{(k)},\mu_k)$
      \Let{$\mathbf{S}^{(k+1)}$}{$\mathcal{S}_{\lambda\mu_k^{-1}}[\mathbf{D}-\mathbf{L}^{(k+1)}+\mu_k^{-1}\mathbf{Y}^{(k)}]$}
      \State // Update the Lagrangian variable
      \Let{$\mathbf{Y}^{(k+1)}$}{$\mathbf{Y}^{(k)} + \mu_k (\mathbf{D}-\mathbf{L}^{(k+1)}-\mathbf{S}^{(k+1)})$} \State Update $\mu_k \leftarrow \mu_{k+1}$
    \EndFor \\
    \Return{$(\mathbf{L}^{(k+1)}, \mathbf{S}^{(k+1)})$}
  \end{algorithmic}
\end{algorithm}

Minimising (\ref{eq:aug-lagr}) over $\mathbf{L}$ requires computing singular values of a large $m\times n$ matrix. It is well known that computing the $k$ largest singular values has a computational complexity of $\mathcal{O}(kmn)$. Thus for practical efficiency it is important to compute only a few singular values \cite{lin2010augmented}. However, the SVD in step $4$ remains the computational bottleneck of Algorithm \ref{alg:ialm}. Theorem \ref{th:ialm-convergence} \cite{lin2010augmented} gives an asymptotic convergence result for Algorithm \ref{alg:ialm}.

\begin{theorem} \label{th:ialm-convergence}
For Algorithm \ref{alg:ialm}, if $\mu_k$ is non-decreasing and $\sum_{k=1}^{+\infty}\mu_k^{-1}=+\infty$, then $(\mathbf{L}^{(k)}, \mathbf{S}^{(k)})$ asymptotically converges to an optimal solution of the RPCA problem.
\end{theorem}

\subsection{Frank-Wolfe Method}

In this section, we review the more general CPCP problem and the well studied Frank-Wolfe (FR) method \cite{frank1956algorithm}, also known as the conditional gradient method \cite{levitin1966constrained} and its associated convergence result. Since the method minimises any convex smooth function $f$ over a bounded convex set $\mathcal{D}\subseteq\mathcal{H}$, where $\mathcal{H}$ is a Hilbert space endowed with an inner product $\langle\cdot,\cdot\rangle$, we study the more general optimisation problem:
\begin{equation} \label{eq:fw-problem}
\begin{array}{ccc}
    \min f(\mathbf{x}) & \text{s.t} & \mathbf{x}\in\mathcal{D},
\end{array}
\end{equation}
where $f$ has a $L$-Lipschitz continues gradient:
\begin{equation}
\begin{array}{cc}
    \forall \mathbf{x},\mathbf{y}\in\mathcal{D}, & \Vert \nabla f(\mathbf{x})-\nabla f(\mathbf{y})\Vert\leq L\Vert\mathbf{x}-\mathbf{y}\Vert.
\end{array}
\end{equation}
Throughout, we let $D=\max_{\mathbf{x},\mathbf{y}\in\mathcal{D}}\Vert\mathbf{x}-\mathbf{y}\Vert$ denote the diameter of the feasible set $\mathcal{D}$. The Frank-Wolfe method, as well as all its variants and extensions studied in this work will assume that the feasibility set $\mathcal{D}$ is bounded, i.e. $D<\infty$. The classical Frank-Wolfe method has many variants with different update rules, but in the most general form it can be written as in Algorithm \ref{alg:fw} \cite{mu2016scalable}.

\begin{algorithm}
  \caption{Frank-Wolfe (FW) \label{alg:fw}}
  \begin{algorithmic}[1]
    \Require{$\mathbf{x}^{(0)}\in\mathcal{D}$}
    \For{$k \gets 1 \textrm{ to } ...$}
      \State $\mathbf{v}^{(k)}\in \argmin_{\mathbf{v}\in\mathcal{D}}\langle\mathbf{v},\nabla f(\mathbf{x}^{(k)})\rangle;$
      \State $\gamma=\frac{2}{k+2}$
      \State Update $\mathbf{x}^{(k+1)}$ to a point in $\mathcal{D}$ so that $f(\mathbf{x}^{(k+1)})\leq f(\mathbf{x}^{(k)}+\gamma(\mathbf{v}^{(k)}-\mathbf{x}^{(k)}));$
    \EndFor \\
    \Return{$\mathbf{x}^{(k+1)}$}
  \end{algorithmic}
\end{algorithm}

The two most common updating rules for $\mathbf{x}^{(k+1)}$ are the simple

\begin{equation}
    \mathbf{x}^{(k+1)}=\mathbf{x}^{(k)}+\gamma(\mathbf{v}^{(k)}-\mathbf{x}^{(k)}),
\end{equation}

and the following slightly more sophisticated one
\begin{equation} \label{eq:fw-update}
\begin{array}{ccc}
    \mathbf{x}^{(k+1)}\in \argmin_{\mathbf{x}} f(\mathbf{x}) & \text{s.t.} & \mathbf{x}\in\conv\{\mathbf{x}^{(k)},\mathbf{v}^{(v)}\}.
\end{array}
\end{equation}

In this paper we will use the more advanced update rule (\ref{eq:fw-update}) for its better practical performance. Using standard techniques it can be shown that the FW method converges at a rate of $\mathcal{O}(1/k)$ in function values.

\begin{theorem}
Let $\mathbf{x}^\star$ be an optimal solution of (\ref{eq:fw-problem}). For $\{\mathbf{x}^{(k)}\}$ generated by Algorithm \ref{alg:fw}, we have for $k=0,1,2,..$
\begin{equation}
    f(\mathbf{x}^{(k)})-f(\mathbf{x}^\star) \leq \frac{2LD^2}{k+2}.
\end{equation}
\end{theorem}

\begin{proof}
The proof can be found, for example, in \cite{mu2016scalable}.
\end{proof}

\subsection{Frank-Wolfe Thresholding Method for CPCP}
In this section we discuss the application of the Frank-Wolfe method to problem (\ref{eq:cpcp}). Since the FW algorithm can be applied only for smooth and constrained convex optimisation problems with a bounded feasible set, we reformulate (\ref{eq:cpcp}) into a such problem. In \cite{mu2016scalable} the reformulation was done by first performing an epigraph reformulation on (\ref{eq:cpcp}) obtaining,
\begin{equation} \label{eq:cpcp-epi}
\begin{array}{ll}
    \min & f(\mathbf{L},\mathbf{S}, t_L, t_S) := \frac{1}{2}\Vert\mathcal{P}_{\mathcal{Q}}[\mathbf{L}+\mathbf{S}-\mathbf{D}]\Vert_F^2 + \lambda_L t_L + \lambda_S t_S \\
    s.t. & \Vert\mathbf{L}\Vert_* \leq t_L, \Vert\mathbf{S}\Vert_1 \leq t_S.
\end{array}
\end{equation}
Now the objective function has a $2$-Lipschitz gradient $\nabla f$ with partial derivatives given as follows:
\begin{equation} \label{eq:cpcp-grad-LS}
    \nabla_L f(\mathbf{L},\mathbf{S},t_L,t_S) = \nabla_S f(\mathbf{L},\mathbf{S},t_L,t_S) = \mathcal{P}_{\mathcal{Q}} [\mathbf{L}+\mathbf{S}-\mathbf{D}],
\end{equation}
\begin{equation} \label{eq:cpcp-grad-t}
    \nabla_{t_L} f(\mathbf{L},\mathbf{S},t_L,t_S) = \lambda_L, \nabla_{t_S} f(\mathbf{L},\mathbf{S},t_L,t_S) = \lambda_S.
\end{equation}

Then to make the feasible region bounded we introduce upper bounds $U_L$ and $U_S$ for $t_L$ and $t_S$, respectively. It can be shown \cite{mu2016scalable} that we can choose
\begin{equation}
\begin{array}{cc}
    U_L=\frac{1}{2\lambda_L}\Vert\mathcal{P}_{\mathcal{Q}}[\mathbf{D}]\Vert_F^2, & U_S=\frac{1}{2\lambda_S}\Vert\mathcal{P}_{\mathcal{Q}}[\mathbf{D}]\Vert_F^2,
\end{array}
\end{equation}
and the resulting feasible set has a bounded diameter: $D\leq \sqrt{5}\cdot\sqrt{U_L^2+U_S^2}$. 

Now we can apply the FW algorithm on
\begin{equation} \label{eq:cpcp-fw}
\begin{array}{ll}
    \min & f(\mathbf{L},\mathbf{S}, t_L, t_S) := \frac{1}{2}\Vert\mathcal{P}_{\mathcal{Q}}[\mathbf{L}+\mathbf{S}-\mathbf{D}]\Vert_F^2 + \lambda_L t_L + \lambda_S t_S \\
    s.t. & \Vert\mathbf{L}\Vert_* \leq t_L \leq U_L, \Vert\mathbf{S}\Vert_1 \leq t_S \leq U_S.
\end{array}
\end{equation}
Setting $\mathbf{x}=(\mathbf{L},\mathbf{S},\lambda_L,\lambda_S)$ and using the gradient expressions (\ref{eq:cpcp-grad-LS})-(\ref{eq:cpcp-grad-t}) we can derive the linear optimisation oracle in step $2$ of Algorithm \ref{alg:fw} as two independent linear optimisation problems:   

\begin{equation} \label{eq:fw-L-subproblem}
(\mathbf{V}_L^{(k)}, V_{t_L}^{(k)})
 \in \argmin\limits_{\Vert\mathbf{V}_L\Vert_*\leq V_{t_L}\leq U_L}\langle \mathcal{P}_\mathcal{Q} [\mathbf{L}^{(k)}+\mathbf{S}^{(k)}-\mathbf{D}], \mathbf{V}_L\rangle + \lambda_L V_{t_L},
\end{equation}

\begin{equation} \label{eq:fw-S-subproblem}
(\mathbf{V}_S^{(k)}, V_{t_S}^{(k)})
 \in \argmin\limits_{\Vert\mathbf{V}_S\Vert_*\leq V_{t_S}\leq U_S}\langle \mathcal{P}_\mathcal{Q} [\mathbf{L}^{(k)}+\mathbf{S}^{(k)}-\mathbf{D}], \mathbf{V}_S\rangle + \lambda_S V_{t_S}.
\end{equation}

Both (\ref{eq:fw-L-subproblem}) and (\ref{eq:fw-S-subproblem}) are separable and can be solved in closed form using the leading singular values and the largest in magnitude elements of $\mathcal{P}_\mathcal{Q} [\mathbf{L}^{(k)}+\mathbf{S}^{(k)}-\mathbf{D}]$ \cite{mu2016scalable}. This is given in steps $3-16$ in Algorithm \ref{alg:fw-t}. Finally, we use the update rule (\ref{eq:fw-update}) for problem (\ref{eq:cpcp-fw}) resulting to step $17$ of Algorithm \ref{alg:fw-t}.

A major drawback of the FW method is that for the $\ell_1$-norm, at each iteration it projects a linear function on a $\ell_1$ ball, thus updating only one entry of a very large matrix. To resolve this problem \cite{mu2016scalable} suggested to add a thresholding operation to the FW algorithm, calling it Frank-Wolfe Thresholding (FW-T). This is performed in steps $18-19$ of Algorithm \ref{alg:fw-t}. Finally, as suggested in \cite{mu2016scalable}, in steps $20-21$ we update the bounds $U_L$ and $U_S$ tightening the feasibility set.

We present the complete Frank-Wolfe Thresholding method in Algorithm \ref{alg:fw-t}. Both primal and dual convergence of the FW-T algorithm was established in \cite{mu2016scalable}. Since the thresholding (and more generally proximal) operator decreases the function value more than the Frank-Wolfe update, the Frank-Wolfe Thresholding algorithm can be seen as a special case of Algorithm \ref{alg:fw} \cite{mu2016scalable}. Therefore FW-T converges to the solution $(\mathbf{L}^\star, \mathbf{S}^\star)$ in function value with the same $\mathcal{O}(1/k)$ rate. Although the FW-T method requires only computing the largest singular value of a $m\times n$ matrix at each iteration, SVD computations still remain the computational bottleneck.

\begin{algorithm}
  \caption{Frank-Wolfe Thresholing (FW-T) for (\ref{eq:cpcp})} \label{alg:fw-t}
  \begin{algorithmic}[1]
    \Require{$\mathbf{D}\in\mathbb{R}^{m\times n}; \lambda_L, \lambda_S > 0$}
    \State Set $\mathbf{L}^{(0)}=\mathbf{S}^{(0)}=\mathbf{0}; t_L^{(0)}=t_S^{(0)}=0; U_L^{(0)}=f(\mathbf{L}^{(0)},\mathbf{S}^{(0)},t_L^{(0)},t_S^{(0)})/\lambda_L; U_S^{(0)}=f(\mathbf{L}^{(0)},\mathbf{S}^{(0)},t_L^{(0)},t_S^{(0)})/\lambda_S$.
    \For{$k \gets 1 \textrm{ to } ...$}
    \State $\mathbf{M}_L^{(k)}\in\argmin\limits_{\Vert\mathbf{M}_L\Vert_*\leq 1} \langle \mathcal{P}_{\mathcal{Q}}[\mathbf{L}^{(k)}+\mathbf{S}^{(k)}-\mathbf{D}],\mathbf{M}_L\rangle$; \\
      \State $\mathbf{M}_S^{(k)}\in\argmin\limits_{\Vert\mathbf{M}_S\Vert_1\leq 1} \langle \mathcal{P}_{\mathcal{Q}}[\mathbf{L}^{(k)}+\mathbf{S}^{(k)}-\mathbf{D}],\mathbf{M}_S\rangle$; \\
      \If {$\lambda_L \geq -\langle\mathcal{P}_{\mathcal{Q}}[\mathbf{L}^{(k)}+\mathbf{S}^{(k)}-\mathbf{D}], \mathbf{M}_L^{(k)}\rangle$}
        \State $\mathbf{V}_L^{(k)} = \mathbf{0}$; $V_{t_L}^{(k)}=0$
      \Else
        \State $\mathbf{V}_L^{(k)} = U_L\mathbf{M}_L^{(k)}$; $V_{t_L}^{(k)}=U_L$
      \EndIf
      \If {$\lambda_S \geq -\langle\mathcal{P}_{\mathcal{Q}}[\mathbf{L}^{(k)}+\mathbf{S}^{(k)}-\mathbf{D}], \mathbf{M}_S^{(k)}\rangle$}
        \State $\mathbf{V}_S^{(k)} = \mathbf{0}$; $V_{t_S}^{(k)}=0$
      \Else
        \State $\mathbf{V}_S^{(k)} = U_S\mathbf{M}_S^{(k)}$; $V_{t_S}^{(k)}=U_S$
      \EndIf
      \State Compute $(\mathbf{L}^{(k+1)},\mathbf{S}^{(k+\frac{1}{2})},t_L^{(k+1)},t_S^{(k+\frac{1}{2})})$ as a minimiser of
      \begin{equation}
          \begin{array}{ll}
               \min\limits_{\mathbf{L},\mathbf{S},t_L,t_S} & \frac{1}{2}\Vert\mathcal{P}_{\mathcal{Q}}[\mathbf{L}+\mathbf{S}-\mathbf{D}]\Vert_F^2+\lambda_L t_L + \lambda_S t_S \\
               s.t. & \begin{pmatrix}\mathbf{L}\\ t_L\end{pmatrix}\in\conv\begin{Bmatrix}\begin{pmatrix}\mathbf{L}^{(k)}\\t_L^{(k)}\end{pmatrix},\begin{pmatrix}\mathbf{V}_L^{(k)}\\V_{t_L}^{(k)}\end{pmatrix}\end{Bmatrix}, \\
               & \begin{pmatrix}\mathbf{S}\\ t_S\end{pmatrix}\in\conv\begin{Bmatrix}\begin{pmatrix}\mathbf{S}^{(k)}\\t_S^{(k)}\end{pmatrix},\begin{pmatrix}\mathbf{V}_S^{(k)}\\V_{t_S}^{(k)}\end{pmatrix}\end{Bmatrix};
          \end{array}
      \end{equation}
      \Let{$\mathbf{S}^{(k+1)}$}{$\mathcal{S}_{\lambda_S}[\mathbf{S}^{(k+\frac{1}{2})}-\mathcal{P}_{\mathcal{Q}}[\mathbf{L}^{(k+\frac{1}{2})}+\mathbf{S}^{(k+\frac{1}{2})}-\mathbf{D}]]$}
      \Let{$t_S^{(k+1)}$}{$\Vert\mathbf{S}^{(k+1)}\Vert_1$}
      \Let{$U_L^{(k+1)}$}{$g(\mathbf{L}^{(K+1)},\mathbf{S}^{(k+1)},t_L^{(k+1)}),t_S^{(k+1)}/\lambda_L$}
      \Let{$U_S^{(k+1)}$}{$g(\mathbf{L}^{(K+1)},\mathbf{S}^{(k+1)},t_L^{(k+1)}),t_S^{(k+1)}/\lambda_S$}
    \EndFor \\
    \Return{$(\mathbf{L}^{(k+1)}, \mathbf{S}^{(k+1)})$}
  \end{algorithmic}
\end{algorithm}

\section{Multilevel Algorithms}

We propose multilevel variants of the two classical methods discussed above. IALM and FW-T have the same computational bottleneck of computing one or more singular values at each iteration. Our methods will build lower dimensional, so-called \textit{coarse}, models for each problem and use their singular values for iteration updates. We will show that both algorithms converge to an (approximate) solution of the original problem with the same worst case iteration complexity of their standard counterparts. However, the per iteration cost of multilevel methods is much smaller, since SVDs are performed on much smaller coarse models \footnote{In our experiments we observed 2-10 times cheaper per iteration complexities}.

\subsection{The Coarse Model} \label{sec:coarse}

In this section we will introduce a generic lower dimensional model for (\ref{eq:cpcp}) and (\ref{eq:rpca}). It uses the so-called restriction operator $\mathbf{R}\in\mathbb{R}^{n\times n_H}$, for some $n_H\leq n$, where $n_H$ is the dimension of the coarse model. Throughout the paper we will make the following assumption about the restriction operator $\mathbf{R}$.

\begin{assumption} \label{ass:R}
The restriction operator $\mathbf{R}$ has linearly independent columns. Therefore, $\mathbf{R}$ has a left inverse $\mathbf{R}^\dagger\in\mathbb{R}^{n_H\times n}$ so that $\mathbf{R}^\dagger\mathbf{R}=\mathbf{I}_{n_H}$ (in general, $\mathbf{R}\mathbf{R}^\dagger\neq\mathbf{I}_n$ and $\mathbf{R}$ may not have a right inverse).
\end{assumption}

Assumption \ref{ass:R} is a very generic and natural assumption about the restriction operator and it is satisfied for all restriction operators used in this work. Indeed, there is no practical advantage of having redundant columns in $\mathbf{R}$, and we can always remove the redundant columns thus creating lower dimensional coarse models.  We make two generic assumptions about the coarse model.
\begin{assumption} \label{ass:nH}
\begin{equation} \label{eq:nH}
\rank(\mathbf{L}^\star)\leq n_H \leq \frac{m+1}{2}.
\end{equation}
\end{assumption}
The first inequality of (\ref{eq:nH}) holds whenever $n<m$, which is the case in all practical problems we consider in the paper. The second inequality, on the other hand, has to be explicitly enforced using an approximate guess on the rank of $\mathbf{L}^\star$, which is well known for most applications. For example, for the video background extraction problem $\rank(\mathbf{L}^\star)\approx 1$ and for the facial shadow removal problem it is $\approx 9$. In all experiments we use this prior information to set the number of levels so that (\ref{eq:nH}) is satisfied.

\begin{assumption} \label{ass:LH}
The low-rank component $\mathbf{L}^\star$ can be represented as $\mathbf{L}^\star=\mathbf{L}_H^\star\mathbf{R}^\top$ for some coarse $\mathbf{L}_H^\star\in\mathbb{R}^{m\times n_H}$.
\end{assumption}

Assumption \ref{ass:LH} is not restrictive for the problems considered here. We demonstrate this fact by using a simple example. Let $\mathbf{L}^\star$ be a simple white rank one background extracted from a video. This means that each of its columns is a white frame stacked as a column vector. Thus it can be written as a matrix of ones. Also, as it is a standard practice in multigrid literature \cite{briggs2000multigrid}, assume $\mathbf{R}$ is the normalised interpolation operator (\ref{eq:Rx}):

\begin{equation} \label{eq:Rx}
\mathbf{R}_n = \frac{1}{2}\begin{bmatrix}
2 & 0 & 0 & 0 &...& 0 & 0 \\
2 & 0 & 0 & 0 &...& 0 & 0 \\
1 & 1 & 0 & 0 &...& 0 & 0 \\
0 & 2 & 0 & 0 &...& 0 & 0 \\
0 & 1 & 1 & 0 &...& 0 & 0 \\
  &   &   &	  &...&   &   \\
0 & 0 & 0 & 0 &...& 1 & 2 \end{bmatrix}{\in\mathbb{R}^{n\times \frac{n}{2}}}.
\end{equation}

Then setting $\mathbf{L_H}^\star$ as a matrix of ones we get:

\begin{equation}
\mathbf{L}_H^\star \mathbf{R}^\top = \begin{bmatrix}
1 & 1 & 1 \\
 & \cdots &  \\
1 & 1 & 1 
\end{bmatrix} \cdot \frac{1}{2} \begin{bmatrix}
2 & 2 & 1 & 0 & 0 & 0 \\
0 & 0 & 1 & 2 & 1 & 0 \\
0 & 0 & 0 & 0 & 1 & 2 
\end{bmatrix} = \begin{bmatrix}
1 & 1 & 1 & 1 & 1 & 1 \\
  & & & \cdots & & \\
1 & 1 & 1 & 1 & 1 & 1
\end{bmatrix} = \mathbf{L}^\star.
\end{equation}

As can be seen from this example, the interpolation restriction operator is suitable for rank one matrices. Indeed, consider the case where $\mathbf{L}_H^\star$ is constructed by repeating a column vector next to each other (this will obviously give a rank one matrix). In fact, it is easy to notice that multiplying a matrix $\mathbf{L}_H^\star$ by this $\mathbf{R}^\top$ from right, effectively adds linear combinations of columns of $\mathbf{L}_H^\star$ without altering the existing columns. And since $\mathbf{R}$ is normalised, the resulted higher dimensional matrix will have exactly the same columns as $\mathbf{L}^\star$. Having this in mind, given a rank one $\mathbf{L}^\star$ with normalised columns, we can always construct a $\mathbf{L}_H^\star$ simply by removing a certain number of its columns.

Assumption \ref{ass:LH} is one of the key assumptions of this paper. It essentially says that the solution can be represented with varying degrees of fidelity, which is what we observe in many applications, including those studied here. It is important to notice, that while we assume the existence of such $\mathbf{L}_H^\star$, we do \textit{not} need to know its value. Nowhere in our multilevel algorithms we use the value of $\mathbf{L}_H^\star$.

In the multilevel literature the standard choice for restriction operator is the interpolation operator (\ref{eq:Rx}) and we will also use it in our experiments. Often in practice we use more than $2$ levels of coarse models. Specifically, we use a restriction operator $\mathbf{R}=\mathbf{R}_n \cdot\mathbf{R}_{\frac{n}{2}}\cdot \ldots \cdot\mathbf{R}_{n_H}\in\mathbb{R}^{n\times n_H}$, where $\mathbf{R}_k\in\mathbb{R}^{k\times\frac{k}{2}}$ is the interpolation operator of appropriate dimensions. For all experiments we use up to the deepest possible levels, so that $n_H > \max \{\rank(\mathbf{L}^\star), r \}$, where $r$ is the number of singular values required by the overlying algorithm. Clearly, this $\mathbf{R}$ has linearly independent columns and thus is full rank.

\subsection{Multilevel IALM} \label{sec:ml-ialm}
In this section we present a computationally efficient multilevel variant of the Inexact ALM algorithm. First, we make an additional assumption for this subsection:

\begin{assumption} \label{ass:Rialm}
$\mathbf{R}$ is normalised so that
$\Vert\mathbf{R}\Vert_2 \leq \Vert\mathbf{R}\Vert_\star \leq 1$.
\end{assumption}

Before we proceed, we shall present an important inequality about singular values that we will use in the proofs. The proof can be found for instance in \cite{wang1997some}.

\begin{theorem}
Let $\mathbf{A},\mathbf{B}\in\mathbb{R}^{m\times n}$ be matrices and $m\geq n$. Then for any $k\in \{1,\dots ,n\}$
\begin{equation} \label{eq:sing-ineq}
\sum_{i=1}^k \sigma_i (\mathbf{A})\sigma_{n-i+1}(\mathbf{B}) \leq \sum_{i=1}^k \sigma_i (\mathbf{A}\mathbf{B}) \leq \sum_{i=1}^k \sigma_i (\mathbf{A})\sigma_i (\mathbf{B}).
\end{equation}
\end{theorem}

Now we show that the restriction operator approximately preserves the nuclear norm.

\begin{theorem} \label{th:LH-eps}
For any $\mathbf{L}_H\in\mathbb{R}^{m\times n_H}$ and $\mathbf{R}\in\mathbb{R}^{n\times n_H}$ with $n_H\leq n\leq m$, the following inequalities hold:
\begin{enumerate}
    \item \begin{equation}
    \Vert\mathbf{L}_H\mathbf{R}^\top\Vert_*\geq \Vert\mathbf{L}_H\Vert_* -\epsilon,
\end{equation}
with $\epsilon=\sum_{k=1}^{r_H}\sigma_k(\mathbf{L}_H)(1-\sigma_{n_H -k+1}(\mathbf{R}))$, where $r_H=\rank(\mathbf{L}_H)$; and
    \item if $\Vert\mathbf{R}\Vert_2\leq 1$, then also
\begin{equation}
    \Vert\mathbf{L}_H\Vert_* \geq \Vert\mathbf{L}_H\mathbf{R}^\top\Vert_*.
\end{equation}
\end{enumerate}
\end{theorem}

\begin{proof}
Since $n\leq n_H$, then using (\ref{eq:sing-ineq}) we have
\begin{equation}
\begin{split}
    \Vert\mathbf{L}_H\mathbf{R}^\top\Vert_* & \geq \sum_{k=1}^{n_H} \sigma_k(\mathbf{L}_H\mathbf{R}^\top) \\
     & \geq \sum_{k=1}^{n_H} \sigma_k(\mathbf{L}_H)\sigma_{n_H-k+1}(\mathbf{R}^\top) \\
     & = \Vert\mathbf{L}_H\Vert_* -\sum_{k=1}^{r_H}\sigma_k(\mathbf{L}_H)(1-\sigma_{n_H -k+1}(\mathbf{R})).
\end{split}
\end{equation}

The second part can be shown similarly. From Assumption \ref{ass:R} we have 
\begin{equation}
\begin{split}
    \Vert\mathbf{L}_H\Vert_* & = \Vert\mathbf{L}_H\mathbf{R}^\top{\mathbf{R}^\dagger}^\top\Vert_* \\ 
     & \geq \sum_{k=1}^{n_H} \sigma_k(\mathbf{L}_H\mathbf{R}^\top{\mathbf{R}^\dagger}^\top) \\
     & \geq \sum_{k=1}^{n_H} \sigma_k(\mathbf{L}_H\mathbf{R}^\top)\sigma_{n_H-k+1}({\mathbf{R}^\dagger}^\top) \\
     & = \Vert\mathbf{L}_H\mathbf{R}^\top\Vert_* -\sum_{k=1}^{r_H}\sigma_k(\mathbf{L})(1-\sigma_{n_H -k+1}(\mathbf{R}^\dagger)) \\
     & = \Vert\mathbf{L}_H\mathbf{R}^\top\Vert_* -\sum_{k=1}^{r_H}\sigma_k(\mathbf{L})(1-\sigma_{n_H -k+1}^{-1}(\mathbf{R})).
\end{split}
\end{equation}

Finally, if $\sigma_k(\mathbf{R}) \leq 1$ for all $k=n_H-r_H+1,\dots,n_H$, then $\Vert\mathbf{L}_H\mathbf{R}^\top\Vert_* \leq \Vert\mathbf{L}_H\Vert_\star$.

\end{proof}

As it can been seen from Theorem \ref{th:LH-eps}, in general having $\epsilon=0$ requires an orthogonal $\mathbf{R}$, which may not be sparse, thus making each iteration of the algorithm computationally expensive. However, if $\mathbf{L}^\star$ is low rank with quickly decreasing singular values and $\sigma_i(\mathbf{R})$ drop slowly, then $\epsilon$ will be small. This is indeed the case for many computer vision applications such as video background extraction and facial shadow removal discussed in the paper, where $\mathbf{L}^\star$ is not only low rank, but also has quadratically decreasing singular values. Moreover, it is easy to check that the singular values of the interpolation restriction operator (\ref{eq:Rx}) satisfy $\sigma_1 (\mathbf{R}) / \sigma_{n_H} (\mathbf{R}) \leq 2$.

Now we proceed to define a coarse model for the Augmented Lagrangian function (\ref{eq:aug-lagr2}):
\begin{equation} \label{eq:aug-lagr2}
    \mathcal{L}(\mathbf{L},\mathbf{S},\mathbf{Y},\mu) = \Vert \mathbf{L}\Vert_* + \lambda\Vert \mathbf{S}\Vert_1 + \langle\mathbf{Y},\mathbf{D}-\mathbf{L}-\mathbf{S}\rangle + \frac{\mu}{2} \Vert \mathbf{D}-\mathbf{L}-\mathbf{S}\Vert_F^2,
\end{equation}

For any fixed $\mathbf{S}$, $\mathbf{Y}$ and $\mu$, we define the coarse augmented Lagrangian function of (\ref{eq:rpca}) as
\begin{equation} \label{eq:coarse-aug-lagr}
    \mathcal{L}_H(\mathbf{L}) = \Vert \mathbf{L}\Vert_* + \lambda\Vert\mathbf{S}\Vert_1 + \langle\mathbf{Y}, [(\mathbf{D}-\mathbf{S})\mathbf{R}-\mathbf{L}]\mathbf{R}^\top\rangle + \frac{\mu}{2} \Vert (\mathbf{D}-\mathbf{S})\mathbf{R}-\mathbf{L}\Vert_F^2.
\end{equation}
We will use the minimiser of $\mathcal{L}_H$ to approximately minimise the true Augmented Lagrangian function (\ref{eq:aug-lagr2}) over $\mathbf{L}$. The minimiser of $\mathcal{L}_H$ is given by the singular value thresholding operator as
\begin{equation}
    \mathbf{L}_H=\mathbf{U}_H\mathcal{S}_{\mu^{-1}}[\mathbf{\Sigma}_H]\mathbf{V}_H^\top,
\end{equation}
where $\mathbf{U}_H\mathbf{\Sigma}_H\mathbf{V}_H^\top=(\mathbf{D}-\mathbf{S}+\mu^{-1}\mathbf{Y})\mathbf{R} = \mathbf{M}\mathbf{R}$. Then we use the prolongation operator $\mathbf{R}^\top$ to lift the minimiser of (\ref{eq:coarse-aug-lagr}) to the fine dimension. The Multilevel IALM algorithm becomes as stated in Algorithm \ref{alg:ml-ialm}.

\begin{algorithm}
  \caption{Multilevel Inexact ALM (ML-IALM) \label{alg:ml-ialm}}
  \begin{algorithmic}[1]
    \Require{$\mathbf{D}, \mathbf{S}^{(0)}, \mathbf{Y}^{(0)} \in\mathbb{R}^{m\times n}$; $\mu_0>0$}
    \For{$k \gets 1 \textrm{ to } ...$}
      \State// Solve $\mathbf{L}_H^{(k+1)}=\argmin\limits_{\mathbf{L}} \mathcal{L}_H (\mathbf{L},\mathbf{S}^{(k)},\mathbf{Y}^{(k)},\mu_k)$
      \Let{$\mathbf{M}^{(k)}_H$}{$(\mathbf{D}-\mathbf{S}^{(k)} + \mu_k^{-1}\mathbf{Y}^{(k)})\mathbf{R}$}
      \Let{$(\mathbf{U}_H, \mathbf{\Sigma}_H, \mathbf{V}_H)$}{SVD($\mathbf{M}^{(k)}_H$)}
      \Let{$\mathbf{L}_H^{(k+1)}$}{$\mathbf{U}_H \mathcal{S}_{\mu^{-1}}[\mathbf{\Sigma}_H]\mathbf{V}_H^\top$}
      \Let{$\mathbf{L}^{(k+1)}$}{$\mathbf{L}_H^{(k+1)} \mathbf{R}^\top$}
      \State// Continue as in Algorithm \ref{alg:ialm}
      \Let{$\mathbf{S}^{(k+1)}$}{$\mathcal{S}_{\lambda\mu_k^{-1}}[\mathbf{D}-\mathbf{L}^{(k+1)}+\mu_k^{-1}\mathbf{Y}^{(k)}]$}
      \Let{$\mathbf{Y}^{(k+1)}$}{$\mathbf{Y}^{(k)} + \mu_k (\mathbf{D}-\mathbf{L}^{(k+1)}-\mathbf{S}^{(k+1)})$} 
      \State Update $\mu_k \leftarrow \mu_{k+1}$
    \EndFor\\
    \Return{$(\mathbf{L}^{(k+1)}, \mathbf{S}^{(k+1)})$}
  \end{algorithmic}
\end{algorithm}

Now we proceed to study the convergence of Algorithm \ref{alg:ml-ialm}. First, we define the sequence $\widehat{\mathbf{Y}}^{(k)}$, which we will use throughout this section. Let
\begin{equation} \label{eq:Yhat}
\widehat{\mathbf{Y}}^{(k+1)}=\mathbf{Y}^{(k)}+\mu_{k}[\mathbf{D}-\mathbf{S}^{(k)} -\mathbf{L}^{(k+1)}].
\end{equation}

We start the convergence proof of Algorithm \ref{alg:ml-ialm} with the multilevel versions of some lemmas of \cite{lin2010augmented}. We begin with Lemma 1 from \cite{lin2010augmented} and show that Multilevel IALM also produces bounded sequences. 

\begin{lemma} \label{l:bounded-sequences}
Let $\widehat{\mathbf{Y}}^{(k+1)}$ be defined as in (\ref{eq:Yhat}) and $\mathbf{Y}^{(k)}$ be defined as in Algorithm \ref{alg:ml-ialm}. Then  the sequence $\{\mathbf{Y}^{(k)}\}$ is bounded and the following hold:
\begin{enumerate}
    \item \begin{equation} \label{eq:LH-optimality}
\widehat{\mathbf{Y}}^{(k+1)}\mathbf{R} + \mu_k \mathbf{L}_H^{(k+1)}(\mathbf{R}^\top\mathbf{R}-\mathbf{I}) \in \partial\Vert\mathbf{L}_H^{(k+1)}\Vert_*,
\end{equation}
for any $\mathbf{L}_H^{(k+1)}\in\mathbb{R}^{m\times n_H}$.
    \item \begin{equation} \label{eq:L-optimality}
\mathbf{Y}^{(k+1)} \in \partial\Vert\mathbf{S}^{(k+1)}\Vert_*.
\end{equation}
\end{enumerate}
\end{lemma}

\begin{proof}
To show (\ref{eq:LH-optimality}) we use the optimality condition of $\mathcal{L}_H$ and the construction $\mathbf{L}^{(k+1)}=\mathbf{L}_H^{(k+1)}\mathbf{R}^\top$. We have
\begin{equation}
\begin{split}
    \mathbf{0}\in & \partial_{\mathbf{L}_H}\mathcal{L}(\mathbf{L}_H^{(k+1)},\mathbf{S}^{(k)},\mathbf{Y}^{(k)},\mu_k) \\
    = & \partial \Vert\mathbf{L}_H^{(k+1)}\Vert_* - \mathbf{Y}^{(k)}\mathbf{R}-\mu_k[(\mathbf{D}-\mathbf{S}^{(k)})\mathbf{R}-\mathbf{L}_H^{(k+1)}] \\
    = & \partial \Vert\mathbf{L}_H^{(k+1)}\Vert_* - \widehat{\mathbf{Y}}^{(k+1)}\mathbf{R} - \mu_k(\mathbf{L}_H^{(k+1)}\mathbf{R}^\top\mathbf{R}-\mathbf{L}_H^{(k+1)}).
\end{split}
\end{equation}

Similarly, using the optimality condition for updating $\mathbf{S}^{(k+1)}$ we can show (\ref{eq:L-optimality}):
\begin{equation}
\mathbf{Y}^{(k+1)} \in \partial\Vert\mathbf{S}^{(k+1)}\Vert_*.
\end{equation}

\end{proof}

\begin{lemma}[\cite{lin2010augmented}, Lemma 2] \label{l:lemma2}
\begin{equation}
\begin{split}
      & \Vert\mathbf{S}^{(k+1)}-\mathbf{S}^\star\Vert_F^2+\mu_k^{-2}\Vert\mathbf{Y}^{(k+1)}-\mathbf{Y}^\star\Vert_F^2 \\
    = & \Vert\mathbf{S}^{(k)}-\mathbf{S}^\star\Vert_F^2+\mu_k^{-2}\Vert\mathbf{Y}^{(k)}-\mathbf{Y}^\star\Vert_F^2-\Vert\mathbf{S}^{(k+1)}-\mathbf{S}^{(k)}\Vert_F^2-\mu_k^{-2}\Vert\mathbf{Y}^{(k+1)}-\mathbf{Y}^{(k)}\Vert_F^2 \\
      & -2\mu_k^{-1}(\langle\mathbf{Y}^{(k+1)}-\mathbf{Y}^{(k)},\mathbf{S}^{(k+1)}-\mathbf{S}^{(k)}\rangle + \langle\mathbf{L}^{(k+1)}-\mathbf{L}^\star,\widehat{\mathbf{Y}}^{(k+1)}-\mathbf{Y}^\star\rangle + \langle\mathbf{S}^{(k+1)}-\mathbf{S}^\star,\mathbf{Y}^{(k+1)}-\mathbf{Y}^\star\rangle).
\end{split}
\end{equation}
\end{lemma}

\begin{proof}
Since the proof of Lemma 2 in \cite{lin2010augmented} relies only on the optimality of $\mathbf{L}^\star$ and $\mathbf{S}^\star$ and the update formula for $\mathbf{Y}^{(k+1)}$, but not on the update rule for $\mathbf{L}^{(k+1)}$ (the only multilevel update part of ML-IALM), then its proof can be exactly repeated for Algorithm \ref{alg:ml-ialm}.
\end{proof}

\begin{lemma} \label{l:subgrad-inequality-L}
If $\Vert\mathbf{R}\Vert_2\leq 1$, then under Assumption \ref{ass:LH}
\begin{equation} \label{eq:subgrad-inequality-L}
    \langle\mathbf{L}^{(k+1)}-\mathbf{L}^\star,\widehat{\mathbf{Y}}^{(k+1)}-\mathbf{Y}^\star\rangle \geq -\epsilon -\mu_k \Delta_{k+1},
\end{equation}
for $\epsilon$ defined as in Theorem \ref{th:LH-eps} and $\Delta_{k+1} =\langle\mathbf{L}_H^{(k+1)}(\mathbf{R}^\top\mathbf{R}-\mathbf{I}), \mathbf{L}_H^{(k+1)} - \mathbf{L}_H^\star\rangle$.
\end{lemma}

\begin{proof}
Since $\mathbf{Y}^\star\in\partial\Vert\mathbf{L}^\star\Vert_*$, we have 
\begin{equation}
    \Vert\mathbf{L}_H^{(k+1)}\mathbf{R}^\top\Vert_* -\Vert\mathbf{L}_H^\star\mathbf{R}^\top\Vert_* \geq \langle\mathbf{Y}^\star,\mathbf{L}^{(k+1)} - \mathbf{L}^\star\rangle.
\end{equation}
Then applying Theorem \ref{th:LH-eps} with $\mathbf{L}_H = \mathbf{L}_H^\star$ we derive:

\begin{equation} \label{eq:lemma3-eq1}
    \Vert\mathbf{L}_H^{(k+1)}\mathbf{R}^\top\Vert_* -\Vert\mathbf{L}_H^\star\Vert_* \geq \langle\mathbf{Y}^\star, \mathbf{L}^{(k+1)} -\mathbf{L}^\star\rangle - \epsilon.
\end{equation}
On the other hand, from (\ref{eq:LH-optimality}) we have

\begin{equation} \label{eq:lemma3-eq2}
    \Vert\mathbf{L}_H^\star\Vert_* - \Vert\mathbf{L}_H^{(k+1)}\Vert_* \geq \langle \widehat{\mathbf{Y}}^{(k+1)}\mathbf{R} + \mu_k \mathbf{L}_H^{(k+1)}(\mathbf{R}^\top\mathbf{R}-\mathbf{I}), \mathbf{L}_H^\star - \mathbf{L}_H^{(k+1)} \rangle.
\end{equation}
Then adding (\ref{eq:lemma3-eq1}) and (\ref{eq:lemma3-eq2}) and using Assumption \ref{ass:LH} we get

\begin{equation}
    \Vert\mathbf{L}_H^{(k+1)}\mathbf{R}^\top\Vert_* - \Vert\mathbf{L}_H^{(k+1)}\Vert_* \geq \langle\mathbf{Y}^\star - \widehat{\mathbf{Y}}^{(k+1)}, \mathbf{L}^{(k+1)} - \mathbf{L}^\star\rangle - \epsilon -\mu_k \langle\mathbf{L}_H^{(k+1)}(\mathbf{R}^\top\mathbf{R}-\mathbf{I}), \mathbf{L}_H^{(k+1)} - \mathbf{L}_H^\star\rangle.
\end{equation}

We finish the proof applying the construction $\mathbf{L}^{(k+1)}=\mathbf{L}_H^{(k+1)}\mathbf{R}^\top$ and denoting $\Delta_{k+1} := \langle\mathbf{L}_H^{(k+1)}(\mathbf{R}^\top\mathbf{R}-\mathbf{I}), \mathbf{L}_H^{(k+1)} - \mathbf{L}_H^\star\rangle$.

\end{proof}

\begin{lemma} \label{l:lemma4}
Let $\Delta_k$ be defined as in Lemma \ref{l:subgrad-inequality-L} and define
\begin{equation}
\begin{split}
c_k & := \mu_k^{-1} (\langle\mathbf{Y}^{(k+1)}-\mathbf{Y}^{(k)},\mathbf{S}^{(k+1)}-\mathbf{S}^{(k)}\rangle + \langle\mathbf{L}^{(k+1)}-\mathbf{L}^\star,\widehat{\mathbf{Y}}^{(k+1)}-\mathbf{Y}^\star\rangle + \\ & \langle\mathbf{S}^{(k+1)}-\mathbf{S}^\star,\mathbf{Y}^{(k+1)}-\mathbf{Y}^\star\rangle).
\end{split}
\end{equation}
Then if $\mu_k$ is non-decreasing, then
\begin{itemize}
\item $c_k \geq -\epsilon\mu_k^{-1} -\Delta_{k+1}$, and
\item $\sum_{k=1}^{k=+\infty} c_k < +\infty$.
\end{itemize}

\end{lemma}

\begin{proof}
Let $(\mathbf{L}^\star,\mathbf{S}^\star,\mathbf{Y}^\star)$ be a saddle point of the Lagrangian of (\ref{eq:rpca}). So we have
\begin{equation}
    \mathbf{Y}^\star \in\partial\Vert\mathbf{L}^\star\Vert_*,\quad\quad\quad \mathbf{Y}^\star\in\partial\Vert\lambda\mathbf{S}^\star\Vert_1.
\end{equation}
Then from Lemma 3 of \cite{lin2010augmented} and $\mathbf{Y}^{(k+1)}\in\partial\Vert\lambda\mathbf{S}^{(k+1)}\Vert_1$ we have
\begin{equation} \label{eq:subgrad-inequalities-S}
\begin{split}
    \langle\mathbf{S}^{(k+1)}-\mathbf{S}^\star, \mathbf{Y}^{(k+1)}-\mathbf{Y}^\star \rangle \geq 0, \\
    \langle\mathbf{S}^{(k+1)}-\mathbf{S}^{(k)}, \mathbf{Y}^{(k+1)}-\mathbf{Y}^{(k)} \rangle \geq 0.
\end{split}
\end{equation}
Adding (\ref{eq:subgrad-inequality-L}) and the two inequalities of (\ref{eq:subgrad-inequalities-S}) we get
\begin{equation}
\begin{split}
 & -\epsilon -\mu_k \Delta_{k+1} \\
 \leq & \langle\mathbf{Y}^{(k+1)}-\mathbf{Y}^{(k)},\mathbf{S}^{(k+1)}-\mathbf{S}^{(k)}\rangle + \langle\mathbf{L}^{(k+1)}-\mathbf{L}^\star,\widehat{\mathbf{Y}}^{(k+1)}-\mathbf{Y}^\star\rangle + \langle\mathbf{S}^{(k+1)}-\mathbf{S}^\star,\mathbf{Y}^{(k+1)}-\mathbf{Y}^\star\rangle.
\end{split}
\end{equation}
Showing that $c_k \geq -\epsilon\mu_k^{-1} -\Delta_{k+1}$.

Show part 2, we apply Lemma \ref{l:lemma2} and use $\mu_{k+1}\geq\mu_k$ to get
\begin{equation} \label{eq:b_k}
    \Vert\mathbf{S}^{(k+1)}-\mathbf{S}^\star\Vert_F^2+\mu_{k+1}^{-2}\Vert\mathbf{Y}^{(k+1)}-\mathbf{Y}^\star\Vert_F^2 \leq \Vert\mathbf{S}^{(k)}-\mathbf{S}^\star\Vert_F^2+\mu_k^{-2}\Vert\mathbf{Y}^{(k)}-\mathbf{Y}^\star\Vert_F^2 + 2\epsilon\mu_k^{-1} +2\Delta_{k+1}.
\end{equation}
Therefore from Lemma \ref{l:lemma2} we conclude that
\begin{equation*}
\begin{split}
     & 2\mu_k^{-1} (\langle\mathbf{Y}^{(k+1)}-\mathbf{Y}^{(k)},\mathbf{S}^{(k+1)}-\mathbf{S}^{(k)}\rangle + \langle\mathbf{L}^{(k+1)}-\mathbf{L}^\star,\widehat{\mathbf{Y}}^{(k+1)}-\mathbf{Y}^\star\rangle \\
     & + \langle\mathbf{S}^{(k+1)}-\mathbf{S}^\star,\mathbf{Y}^{(k+1)}-\mathbf{Y}^\star\rangle) \\
    \leq & (\Vert\mathbf{S}^{(k)}-\mathbf{S}^\star\Vert_F^2+\mu_k^{-2}\Vert\mathbf{Y}^{(k)}-\mathbf{Y}^\star\Vert_F^2) - (\Vert\mathbf{S}^{(k+1)}-\mathbf{S}^\star\Vert_F^2+\mu_{k+1}^{-2}\Vert\mathbf{Y}^{(k+1)}-\mathbf{Y}^\star\Vert_F^2).
\end{split}
\end{equation*}

\end{proof}

\begin{theorem} [Convergence of Multilevel IALM] \label{th:mlialm-convergence}
For Algorithm \ref{alg:ml-ialm}, if $\{\mu_k\}$ is non-decreasing, $\sum_{k=1}^{+\infty}\mu_k^{-1}=+\infty$ and $\sum_{k=1}^{+\infty}\mu_k^{-2}<+\infty$, then $(\mathbf{L}^{(k)},\mathbf{S}^{(k)})$ asymptotically converges to an approximate solution of (\ref{eq:rpca}).
\end{theorem}

\begin{proof}
Similarly to the proof of Lemma \ref{l:lemma4} we have that
\begin{equation}
\begin{split}
	& \mu_k^{-2} \Vert\mathbf{Y}^{(k+1)}-\mathbf{Y}^{(k)} \Vert_F^2 -2\epsilon\mu_k^{-1} -2[\Delta_{k+1}]_{+} \\
    \leq & \mu_k^{-2} \Vert\mathbf{Y}^{(k+1)}-\mathbf{Y}^{(k)} \Vert_F^2 -2\epsilon\mu_k^{-1} -2\Delta_{k+1} \\
    \leq &  (\Vert\mathbf{S}^{(k)}-\mathbf{S}^\star\Vert_F^2+\mu_k^{-2}\Vert\mathbf{Y}^{(k)}-\mathbf{Y}^\star\Vert_F^2) - (\Vert\mathbf{S}^{(k+1)}-\mathbf{S}^\star\Vert_F^2+\mu_{k+1}^{-2}\Vert\mathbf{Y}^{(k+1)}-\mathbf{Y}^\star\Vert_F^2) \\
    & - 2\epsilon\mu_k^{-1} -2\Delta_{k+1} -2\mu_k^{-1}(\langle\mathbf{Y}^{(k+1)}-\mathbf{Y}^{(k)},\mathbf{S}^{(k+1)}-\mathbf{S}^{(k)}\rangle + \langle\mathbf{L}^{(k+1)}-\mathbf{L}^\star,\widehat{\mathbf{Y}}^{(k+1)}-\mathbf{Y}^\star\rangle \\ & + \langle\mathbf{S}^{(k+1)}-\mathbf{S}^\star,\mathbf{Y}^{(k+1)}-\mathbf{Y}^\star\rangle) \\
    \leq & (\Vert\mathbf{S}^{(k)}-\mathbf{S}^\star\Vert_F^2+\mu_k^{-2}\Vert\mathbf{Y}^{(k)}-\mathbf{Y}^\star\Vert_F^2) - (\Vert\mathbf{S}^{(k+1)}-\mathbf{S}^\star\Vert_F^2+\mu_{k+1}^{-2}\Vert\mathbf{Y}^{(k+1)}-\mathbf{Y}^\star\Vert_F^2).
\end{split}
\end{equation}
Therefore,
\begin{equation}
    \sum_{k=1}^{+\infty} (\mu_k^{-2} \Vert\mathbf{Y}^{(k+1)}-\mathbf{Y}^{(k)} \Vert_F^2 -2\epsilon\mu_k^{-1} -2 [\Delta_{k+1}]_{+}) < +\infty.
\end{equation}
Thus, $\mu_k^{-2}\Vert\mathbf{Y}^{(k+1)}-\mathbf{Y}^{(k)}\Vert_F^2 -2\epsilon\mu_k^{-1} -2[\Delta_{k+1}]_{+}\rightarrow 0$, and since $2\epsilon\mu_k^{-1}\rightarrow 0$ we see that 

\begin{equation}
\Vert\mathbf{D}-\mathbf{L}^{(k)}-\mathbf{S}^{(k)}\Vert_F = \mu_{k}^{-1} \Vert\mathbf{Y}^{(k)}-\mathbf{Y}^{(k-1)}\Vert_F\rightarrow (2[\Delta_{k+1}]_{+})^{1/2},
\end{equation}

Moreover, since we showed in Lemma \ref{l:bounded-sequences} that $\{\mathbf{Y}^{(k+1)}\}$ is a bounded sequence, it follows that so is $[\Delta_k]_{+}$. Therefore, denoting $\max \{\Delta_k\} := \delta$ we show that
any accumulation point of $(\mathbf{L}^{(k)},\mathbf{S}^{(k)})$ is a $\delta$-feasible solution.

On the other hand, denote the optimal objective value of problem (\ref{eq:rpca}) by $f^*$. As $\widehat{\mathbf{Y}}^{(k)}\mathbf{R}\in\partial\Vert\mathbf{L}_H^{(k)}\Vert_*$ and $\mathbf{Y}^{(k)}\in\partial(\lambda\Vert\mathbf{S}^{(k)}\Vert_1)$, under Assumption \ref{ass:Rialm} we have
\begin{equation}
\begin{split}
         & \Vert\mathbf{L}^{(k)}\Vert_* + \lambda\Vert\mathbf{S}^{(k)}\Vert_1 \\
    \leq & \Vert\mathbf{L}_H^{(k)}\Vert_* + \lambda\Vert\mathbf{S}^{(k)}\Vert_1 \\
    \leq & \Vert\mathbf{L}_H^\star\Vert_* + \lambda\Vert\mathbf{S}^\star\Vert_1 - \langle\widehat{\mathbf{Y}}^{(k)}\mathbf{R},\mathbf{L}_H^\star-\mathbf{L}_H^{(k)}\rangle - \langle\mathbf{Y}^{(k)},\mathbf{S}^\star-\mathbf{S}^{(k)}\rangle \\
    \leq & \Vert\mathbf{L}^\star\Vert_* + \epsilon + \lambda\Vert\mathbf{S}^\star\Vert_1 - \langle\widehat{\mathbf{Y}}^{(k)},\mathbf{L}^\star-\mathbf{L}^{(k)}\rangle - \langle\mathbf{Y}^{(k)},\mathbf{S}^\star-\mathbf{S}^{(k)}\rangle \\
    =   & f^\star + \langle\mathbf{Y}^\star-\widehat{\mathbf{Y}}^{(k)},\mathbf{L}^\star-\mathbf{L}^{(k)}\rangle + \langle\mathbf{Y}^\star-\mathbf{Y}^{(k)},\mathbf{S}^\star-\mathbf{S}^{(k)}\rangle - \langle\mathbf{Y}^\star,\mathbf{L}^\star-\mathbf{L}^{(k)}+\mathbf{S}^\star-\mathbf{S}^{(k)}\rangle + \epsilon \\
    =   & f^\star + \langle\mathbf{Y}^\star-\widehat{\mathbf{Y}}^{(k)},\mathbf{L}^\star-\mathbf{L}^{(k)}\rangle + \langle\mathbf{Y}^\star-\mathbf{Y}^{(k)},\mathbf{S}^\star-\mathbf{S}^{(k)}\rangle - \langle\mathbf{Y}^\star,\mathbf{D}-\mathbf{L}^{(k)}-\mathbf{S}^{(k)}\rangle + \epsilon.
\end{split}
\end{equation}

Similarly to Lemma \ref{l:lemma4} we notice that from Lemma \ref{l:lemma2}
\begin{equation}
\begin{split}
    & \mu_k^{-1} (\langle\mathbf{L}^{(k)}-\mathbf{L}^\star,\widehat{\mathbf{Y}}^{(k)}-\mathbf{Y}^\star\rangle + \langle\mathbf{S}^{(k)}-\mathbf{S}^\star,\mathbf{Y}^{(k)}-\mathbf{Y}^\star\rangle) \\
    \leq & (\Vert\mathbf{S}^{(k)}-\mathbf{S}^\star\Vert_F^2+\mu_k^{-2}\Vert\mathbf{Y}^{(k)}-\mathbf{Y}^\star\Vert_F^2) - (\Vert\mathbf{S}^{(k+1)}-\mathbf{S}^\star\Vert_F^2+\mu_{k+1}^{-2}\Vert\mathbf{Y}^{(k+1)}-\mathbf{Y}^\star\Vert_F^2).
\end{split}
\end{equation}
And therefore,
\begin{equation}
    \sum_{k=1}^{k=+\infty} \mu_k^{-1} (\langle\mathbf{L}^{(k)}-\mathbf{L}^\star,\widehat{\mathbf{Y}}^{(k)}-\mathbf{Y}^\star\rangle + \langle\mathbf{S}^{(k)}-\mathbf{S}^\star,\mathbf{Y}^{(k)}-\mathbf{Y}^\star\rangle) < +\infty.
\end{equation}

As $\sum_{k=1}^{+\infty}\mu_k^{-1}=+\infty$, there must exist a subsequence $(\mathbf{L}^{(k_j)},\mathbf{S}^{(k_j)})$ such that
\begin{equation}
    \langle\mathbf{L}^{(k_j)}-\mathbf{L}^\star,\widehat{\mathbf{Y}}^{(k_j)}-\mathbf{Y}^\star\rangle + \langle\mathbf{S}^{(k_j)}-\mathbf{S}^\star,\mathbf{Y}^{(k_j)}-\mathbf{Y}^\star\rangle \rightarrow 0.
\end{equation}
Then since $\Vert\mathbf{D}-\mathbf{L}^{(k)}-\mathbf{S}^{(k)}\Vert_F \leq \delta$, we have that 
\begin{equation}
    \lim_{j\rightarrow +\infty} \Vert\mathbf{L}^{(k_j)}\Vert_* + \lambda\Vert\mathbf{S}^{(k_j)}\Vert_1 \leq f^\star + \epsilon + \delta.
\end{equation}
So $(\mathbf{L}^{(k_j)},\mathbf{S}^{(k_j)})$ approaches to an $(\epsilon+\delta)$-approximate solution of problem (\ref{eq:rpca}).

\end{proof}

Notice that Theorem \ref{th:mlialm-convergence} gives a similar convergence results as Theorem \ref{th:ialm-convergence}, meaning that one should expect a similar number of iterations for IALM and ML-IALM methods. This is indeed the case, as observed from empirical studies. However, since ML-IALM performs SVDs on much smaller dimensional matrices, each iteration is significantly cheaper.Of course, as opposed to the original IALM algorithm, here we only showed an approximate convergence. However, as several numerical experiments will demonstrate, the approximation error is practically negligible.
\subsection{Multilevel Frank-Wolfe}

In this section we use operator notation for the linear restriction operator, i.e. $\mathcal{R}:\mathcal{H}\rightarrow\mathcal{H}_H$ is a linear operator from the original space $\mathcal{H}$ to the \textit{coarse space} $\mathcal{H}_H$. $\mathcal{H}$ and $\mathcal{H}_H$ are both Hilbert spaces endowed with inner products and $\mathcal{H}_H$ has lower dimension. In the next subsection we will see what $\mathcal{H}$, $\mathcal{H}_H$ and $\mathcal{R}$ are for the CPCP model.

First we create a coarse model for the gradient by applying the restriction operator $\mathcal{R}$, then solve the linear optimisation oracle over the coarse gradient, then lift the solution back to the original dimension applying the transpose of the restriction operator. For the algorithm we use a convex set $\mathcal{D}_H\subseteq\mathcal{H}_H$ such that for every $\mathbf{x}_H\in\mathcal{D}_H$ it holds that $\mathcal{R}^\top(\mathbf{x}_H)\in\mathcal{D}$. The method is given in Algorithm \ref{alg:ml-fw}.

\begin{algorithm}
  \caption{Multilevel Frank-Wolfe (ML-FW) \label{alg:ml-fw}}
  \begin{algorithmic}[1]
    \Require{$\mathbf{x}_H^{(0)}\in\mathcal{D}_H$}
    \Let{$\mathbf{x}^{(0)}$}{$\mathcal{R}^\top(\mathbf{x}_H^{(0)})$}
    \For{$k \gets 0,1, \textrm{ to } \dots$}
      \State $\mathbf{v}_H^{(k)} \in\argmin_{\mathbf{v}\in\mathcal{D}_H}\langle\mathbf{v},\mathcal{R}(\nabla f(\mathbf{x}^{(k)}))\rangle$
      \Let{$\mathbf{v}^{(k)}$}{$\mathcal{R}^\top(\mathbf{v}_H^{(k)})$}
      \Let{$\gamma$}{$\frac{2}{k+2}$}
      \State Set $\mathbf{x}^{(k+1)}\in\mathcal{D}$ so that $f(\mathbf{x}^{(k+1)})\leq f(\mathbf{x}^{(k)}+\gamma(\mathbf{v}^{(k)}-\mathbf{x}^{(k)}))$
    \EndFor \\
    \Return{$\mathbf{x}^{(k+1)}$}
  \end{algorithmic}
\end{algorithm}

Using techniques similar to the original proof \cite{mu2016scalable}, it can be shown that the ML-FW method converges to a point obtained from the coarse level at a $\mathcal{O}(1/k)$ rate in function values.

\begin{theorem} \label{th:mlfw-convergence}
For any $\mathbf{x}_H^\star\in\mathcal{D}_H$ and for $\{\mathbf{x}^{(k)}\}$ generated by Algorithm \ref{alg:ml-fw}, we have for any $\mathbf{x}_H^\star\in\mathcal{H}_H$ and $k=0,1,2,..$
\begin{equation}
    f(\mathbf{x}^{(k)})-f(\mathcal{R}^\top(\mathbf{x}_H^\star)) \leq \frac{2LD^2}{k+2}.
\end{equation}
\end{theorem}

\begin{proof}
For $k=0,1,2,\dots$ we have

\begin{equation}
\begin{split}
    f(\mathbf{x}^{(k+1)}) & \leq f(\mathbf{x}^{(k)}+\gamma(\mathbf{v}^{(k)}-\mathbf{x}^{(k)})) \\
     & \leq f(\mathbf{x}^{(k)})+\gamma\langle\nabla f(\mathbf{x}^{(k)}),\mathbf{v}^{(k)}-\mathbf{x}^{(k)}\rangle + \frac{L\gamma^2}{2}\Vert\mathbf{v}^{(k)}-\mathbf{x}^{(k)}\Vert^2 \\
     & \leq f(\mathbf{x}^{(k)})+\gamma\langle\nabla f(\mathbf{x}^{(k)}),\mathcal{R}^\top(\mathbf{v}_H^{(k)})-\mathcal{R}^\top(\mathbf{x}_H^{(k)})\rangle + \gamma\langle\nabla f(\mathbf{x}^{(k)}),\mathcal{R}^\top(\mathbf{x}_H^{(k)})-\mathbf{x}^{(k)}\rangle\\
     & + \frac{\gamma^2LD^2}{2}\\
     & = f(\mathbf{x}^{(k)})+\gamma\langle\mathcal{R}(\nabla f(\mathbf{x}^{(k)})),\mathbf{v}_H^{(k)}-\mathbf{x}_H^{(k)}\rangle + \gamma\langle\nabla f(\mathbf{x}^{(k)}),\mathcal{R}^\top(\mathbf{x}_H^{(k)})-\mathbf{x}^{(k)}\rangle + \frac{\gamma^2LD^2}{2} \\
     & \leq f(\mathbf{x}^{(k)})+\gamma\langle\mathcal{R}(\nabla f(\mathbf{x}^{(k)})),\mathbf{x}_H^\star-\mathbf{x}_H^{(k)}\rangle + \gamma\langle\nabla f(\mathbf{x}^{(k)}),\mathcal{R}^\top(\mathbf{x}_H^{(k)})-\mathbf{x}^{(k)}\rangle + \frac{\gamma^2LD^2}{2} \\
     & = f(\mathbf{x}^{(k)})+\gamma\langle\nabla f(\mathbf{x}^{(k)}),\mathcal{R}^\top(\mathbf{x}_H^\star)-\mathbf{x}^{(k)}\rangle + \frac{\gamma^2LD^2}{2} \\
     & \leq f(\mathbf{x}^{(k)})+\gamma(f(\mathcal{R}^\top(\mathbf{x}_H^\star))-f(\mathbf{x}^{(k)})) + \frac{\gamma^2LD^2}{2}.
\end{split}
\end{equation}
Here for the first line we used the updating rule in Algorithm \ref{alg:ml-fw}; for the second line we used the Lipschitz continuity of $f$; for the third line - the definitions of $\mathbf{v}_H^{(k)}$ and $D$, and we added and subtracted $\mathcal{R}^\top (\mathbf{x}_H^{(k)})$; for the fourth line - the property of inner product; for the fifth line - the optimality of $\mathbf{v}_H^{(k)}$; for the sixth line - the property of inner product and the definition of $\mathbf{x}^{(k)}$; and for the last line - the convexity of $f$. Then rearranging the terms we get
\begin{equation} \label{eq:ml-fw-thm-eq1}
    f(\mathbf{x}^{(k+1)})-f(\mathcal{R}^\top(\mathbf{x}_H^\star)) \leq (1-\gamma)(f(\mathbf{x}^{(k)})-f(\mathcal{R}^\top(\mathbf{x}_H^\star)))+\frac{\gamma^2 LD^2}{2}.
\end{equation}
Therefore, by mathematical induction, it can be verified that
\begin{equation} \label{eq:ml-fw-rate}
\begin{array}{ccc}
    f(\mathbf{x}^{(k)})-f(\mathcal{R}^\top(\mathbf{x}_H^\star)) \leq \frac{2LD^2}{k+2} & \text{for} & k=1,2,3,\dots.
\end{array}
\end{equation}

\end{proof}

Theorem \ref{th:mlfw-convergence} tells us, that if the minimiser $\mathbf{x}^\star$ can be accurately represented in terms of a coarse variable, then ML-FWT is a good and efficient method for that particular problem. As we will see in the next subsection, this is indeed the case for the CPCP model.

\subsection{Multilevel Frank-Wolfe Thresholding for CPCP} \label{sec:ml-fwt}

In this subsection we will modify the Multilevel Frank-Wolfe method similarly to the Frank-Wolfe Thresholding method introducing the Multilevel Frank-Wolfe Thresholding method and apply it for the CPCP problem (\ref{eq:cpcp}). In this case as well we will apply the multilevel update only on nuclear ball projections. 

We begin with defining the fine and coarse spaces, and the restriction operator for the CPCP problem. In this setting our variables becomes $\mathbf{x}=(\mathbf{L},\mathbf{S},t_L,t_S)$ and the space is $\mathcal{H}=\mathbb{R}^{m\times n}\times\mathbb{R}^{m\times n}\times \mathbb{R}\times \mathbb{R}$. 

Since we are applying the multilevel steps only for updating $\mathbf{L}^{(k)}$, we will use the following restriction operator:
\begin{equation} \label{eq:R-fw}
    \mathbf{R}=\begin{pmatrix}
    \mathbf{R}_x &  &  &  & \\
                 & \mathbf{I}_{n\times n} & & \\
                 &                        & 1 & \\
                 &                        &   & 1
    \end{pmatrix},
\end{equation}
so that 
\begin{equation}
\begin{split}
    & \mathcal{R}(\mathbf{L},\mathbf{S},\lambda_L,\lambda_S) = (\mathbf{L}\mathbf{R}_x,\mathbf{S},\lambda_L,\lambda_S) \\
    & \mathcal{R}^\top(\nabla_\mathbf{L} f, \nabla_\mathbf{S} f, \nabla_{\lambda_L} f, \nabla_{\lambda_S} f) = (\nabla_\mathbf{L} f \mathbf{R}_x, \nabla_\mathbf{S} f, \nabla_{\lambda_L} f, \nabla_{\lambda_S} f),
\end{split}
\end{equation}
and thus $\mathcal{R}(\nabla f(\mathbf{x}))$ only affects $\nabla_\mathbf{L} f(\mathbf{L},\mathbf{S},t_L ,t_S )$ and correspondingly, $\mathcal{R}^\top(\mathbf{v}_H)$ only affects $\nabla f_{\mathbf{L}}$. Here $\mathbf{R}_x = \mathbf{R}_n\cdot\ldots\cdot\mathbf{R}_{n_H}$ is the restriction operator as defined in Section \ref{sec:coarse}. Therefore, the coarse space becomes $\mathcal{H}_H=\mathbb{R}^{m\times n_H}\times \mathbb{R}^{m\times n}\times \mathbb{R}\times \mathbb{R}$. Now we can define the coarse feasibility set $\mathcal{D}_H$ for the CPCP problem as follows:
\begin{equation}
    \Vert\mathbf{M}_{L,H}\Vert_* \leq 1/\Vert\mathcal{R}\Vert_*,
\end{equation}
so that for each $k=0,1,\dots$
\begin{equation}
    \Vert\mathbf{M}_L^{(k)}\Vert_* = \Vert\mathcal{R}^T(\mathbf{M}_{L,H}^{(k)})\Vert_* \leq \Vert\mathcal{R}\Vert_*\Vert\mathbf{M}_{L,H}^{(k)}\Vert_* \leq 1,
\end{equation}
is a feasible point of the fine problem, where $\mathbf{M}_{L,H}$ is the coarse variable. We call the new algorithm Multilevel Frank-Wolfe Thresholding (ML-FWT) (Algorithm (\ref{alg:ml-fwt-cpcp})).

\begin{algorithm}
  \caption{Multilevel Frank-Wolfe Thresholing (ML-FWT) \label{alg:ml-fwt-cpcp}}
  \begin{algorithmic}[1]
    \Require{$\mathbf{D}\in\mathbb{R}^{m\times n}; \lambda_L, \lambda_S > 0$}
    \State Initialize as in Algorithm \ref{alg:fw-t}
    \For{$k \gets 1 \textrm{ to } ...$}
      \State $\mathbf{M}_{L,H}^{(k)}\in\argmin\limits_{\Vert\mathbf{M}_{L,H}\Vert_*\leq 1/\Vert\mathcal{R}\Vert_*} \langle \mathcal{R}(\mathcal{P}_{\mathcal{Q}}[\mathbf{L}^{(k)}+\mathbf{S}^{(k)}-\mathbf{D}]),\mathbf{M}_{L,H}\rangle$ \\
      \State \Let{$\mathbf{M}_L^{(k)}$}{$\mathcal{R}^\top(\mathbf{M}_{L,H}^{(k)})$}
      \State // Continue as in Algorithm \ref{alg:fw-t}
    \EndFor \\
    \Return{$(\mathbf{L}^{(k+1)}, \mathbf{S}^{(k+1)})$}
  \end{algorithmic}
\end{algorithm}

Thus the multilevel update step for $\mathbf{M}_L^{(k)}$ requires calculating only the largest singular value with the corresponding singular vectors for much lower dimensional matrices. The next theorem gives convergence guarantees for the ML-FWT method.

\begin{theorem}
Let $f(\mathbf{L},\mathbf{S},t_L,t_S)$ be defined as in (\ref{eq:cpcp-fw}) and $\rank(\mathbf{L}^\star)\leq n_H$. Then under Assumptions \ref{ass:R} and \ref{ass:LH}, for $\mathbf{x}^{(k)}$, $k=1,2,\dots$ defined as in Algorithm \ref{alg:ml-fwt-cpcp} the following holds
\begin{equation}
    f(\mathbf{x}^{(k)}) - f(\mathbf{x}^\star) \leq \frac{2LD^2}{k+2}.
\end{equation}
\end{theorem}

\begin{proof}
From assumption \ref{ass:LH} we derive
\begin{equation}
\begin{split}
    f(\mathcal{R}^\top(\mathbf{L}_H^\star,\mathbf{S}^\star,t_L^\star,t_S^\star)) = & \frac{1}{2}\Vert \mathcal{P}_\mathcal{Q}[\mathbf{D}-\mathbf{L}_H^\star\mathbf{R}^\top -\mathbf{S}^\star]\Vert_F^2 + \lambda_L t_L^\star + \lambda_S t_S^\star \\
    = & \frac{1}{2}\Vert \mathcal{P}_\mathcal{Q}[\mathbf{D}-\mathbf{L}^\star -\mathbf{S}^\star]\Vert_F^2 + \lambda_L t_L^\star + \lambda_S t_S^\star \\
    = & f(\mathbf{L}^\star,\mathbf{S}^\star,t_L^\star,t_S^\star).
\end{split}
\end{equation}

Therefore, the claim follows from Theorem \ref{th:mlfw-convergence}.

\end{proof}

Finally, note that since FWT and ML-FWT have the same convergence rate and multiplying by the sparse matrix $\mathbf{R}$ (and its powers) is much cheaper than computing one singular value (although both have the same $\mathcal{O}(mn)$ worst case complexity), ML-FWT has a lower overall complexity.

\section{Experiments} \label{sec:experiments}

To test the practical efficiency of the proposed methods we compare them with the standard Inexact ALM \cite{lin2010augmented} and Frank-Wolfe Thresholding \cite{netrapalli2014non} algorithms on several synthetically generated problems, as well as real life video background extraction and facial shadow removal problems. For the standard Inexact ALM and Frank-Wolfe Thresholding algorithms we used the provided Matlab code. Then for each multilevel variant we replaced the standard singular value thresholding parts of respective algorithms with corresponding multilevel singular value thresholding code, keeping the rest of the algorithms unchanged. Particularly, we used the same optimality criteria, so that the comparisons are fair. All methods were tested in Matlab R2015a on a standard Ubuntu 16.4 machine with Intel Core i7 processor and 32GB RAM. The code is available online at \href{https://github.com/vahanhov/ml-rpca}{https://github.com/vahanhov/ml-rpca}.

\subsection{Synthetic Data} \label{sec:synthetic}

First we test the multilevel algorithms on synthetically generated data matrix $\mathbf{D}=\hat{\mathbf{L}}+\hat{\mathbf{S}}\in\mathbb{R}^{m\times n}$, where $\hat{\mathbf{L}}$ has a fixed low rank $r$ and $\hat{\mathbf{S}}$ is $\eta$-sparse (i.e. has at most $\eta\cdot m n$ non-zero entries). We generate the synthetic data so that the singular values of the low rank component follow $1/k^2$, where $k$ indicates the $k$-th largest singular value. 

We run two sets of experiments. The first one compares the results achieved after running IALM and ML-IALM on RPCA problems for a fixed time. We run two pairs of experiments: with smaller and larger data, each with lower and higher rank of the low-rank component. The experiments are described in Table \ref{t:exp-synth-ialm}, with each row corresponding to one experimental setting. The first three columns describe the particular setting and the number of seconds dedicated to solve the problem. Then each triplet of columns gives the results achieved by IALM and ML-IALM algorithms correspondingly. The reported results are $f^\star$ - achieved objective value, error$(\mathbf{S}^\star)$ and error$(\mathbf{L}^\star)$ - achieved relative errors from corresponding ground truths. It is evident that ML-IALM accurately solves all four problems, while both objective values and  relative errors from IALM are several times larger than those of ML-IALM. This effectively means that ML-IALM can solve large problems in reasonable time that may require impractically long times for IALM.

Then we synthetically generate similar data, but this time with partial observations. This setting is modelled as a CPCP problem and is then solved using FWT and ML-FWT methods. The experimental settings are given in Table \ref{t:exp-synth-fwt}. Here as well we have two pairs of problems: larger and smaller with larger and smaller ranks of the low-rank component. Here we run both problems until $10^{-3}$ tolerance as suggested in \cite{mu2016scalable}. Here both algorithms achieve relatively small objective values and relative errors from the ground truth, with ML-FWT being slightly better, however ML-FWT takes significantly less time to do so. In fact, it is more than twice faster for the smaller rank settings.

\begin{table}
\center
\begin{tabular}{|c|c|c|c|c|c|c|c|c|}
	\hline 
    \multicolumn{3}{|c|}{problem} & \multicolumn{3}{|c|}{IALM} & \multicolumn{3}{|c|}{ML-IALM} \\
	\hline 
	dimensions & rank & sec & $f^\star$ & error$(\mathbf{L}^\star)$ & error$(\mathbf{S}^\star)$ & $f^\star$ & error$(\mathbf{L}^\star)$ & error$(\mathbf{S}^\star)$ \\
	\hline 
	$5000\times 100$ & $2$ & $5$ & $19$ & $7$ & $0.1$ & $10$ & $1$ & $0.02$ \\
	\hline 
	$5000\times 100$ & $5$ & $5$ & $18$ & $6$ & $0.1$ & $7.5$ & $1$ & $0.02$ \\
	\hline 
	$5000\times 1000$ & $2$ & $10$ & $64$ & $42$ & $0.8$ & $7$ & $1$ & $0.01$ \\
	\hline 
	$5000\times 1000$ & $5$ & $10$ & $64$ & $43$ & $0.8$ & $8$ & $1$ & $0.01$ \\
	\hline 
\end{tabular} 
\caption{Achieved objective function values and relative errors from ground truth after running IALM and ML-IALM on RPCA problems with synthetic data for a fixed time.}
\label{t:exp-synth-ialm}
\end{table}

\begin{table}
\center
\begin{tabular}{|c|c|c|c|c|c|c|c|c|c|}
	\hline 
    \multicolumn{2}{|c|}{problem} & \multicolumn{4}{|c|}{FWT} & \multicolumn{4}{|c|}{ML-FWT} \\
	\hline 
	dimensions & rank & sec & $f^\star$ & err$(\mathbf{L}^\star)$ & err$(\mathbf{S}^\star)$ & sec & $f^\star$ & err$(\mathbf{L}^\star)$ & err$(\mathbf{S}^\star)$ \\
	\hline 
	$10000\times 500$ & $2$ & $14.5$ & $8\cdot 10^{-4}$ & $1.3$ & $88$ & $9$ & $2\cdot 10^{-4}$ & $1.1$ & $89.4$ \\
	\hline 
	$10000\times 5000$ & $2$ & $209$ & $0.02$ & $1.3$ & $125$ & $91$ & $7\cdot 10^{-4}$ & $1$ & $125$ \\
	\hline 
	$10000\times 500$ & $5$ & $14.4$ & $8\cdot 10^{-4}$ & $1.3$ & $89$ & $12.5$ & $2\cdot 10^{-4}$ & $1.1$ & $89$ \\
	\hline 
	$10000\times 5000$ & $5$ & $203$ & $0.002$ & $1.3$ & $125$ & $90$ & $7\cdot 10^{-4}$ & $1.1$ & $125$ \\
	\hline 
\end{tabular} 
\caption{CPU times (in seconds), achieved objective function values and relative errors from ground truth after running FWT and ML-FWT on CPCP problems with synthetic data until $10^{-3}$ convergence error.}
\label{t:exp-synth-fwt}
\end{table}

\subsection{Video Background Extraction} \label{sec:video}
Now we test the algorithms on real surveillance videos. Assume we are given a surveillance video from a fixed camera and the task is to separate the constant background from moving objects. This problem can be modelled as a RPCA problem \cite{bouwmans2016decomposition}. We first stack each frame of the video as a column vector creating a data matrix $\mathbf{D}$. Then, since the fixed background remains (approximately) constant in each frame and the moving objects take a relatively small portion of each frame, they can respectively represent the low rank and sparse components of the RPCA decomposition. We tested all algorithms on several surveillance videos described below.

\begin{itemize}
    \item \textbf{highway}: $ 48\times 64\times 400$; run $5$ seconds
    \item \textbf{copy machine}: $ 48\times 72\times 3400$; run $50$ seconds  \footnote{\url{http://wordpress-jodoin.dmi.usherb.ca/dataset2012/}}
    \item \textbf{walk:} $ 240\times 320\times 794$; run $30$ seconds \cite{vacavant2012benchmark}
    \item \textbf{gates:} $ 240\times 320\times 1895$; run $200$ seconds
    \cite{vacavant2012benchmark}
\end{itemize}

First we test the IALM and ML-IALM methods. Here we run both methods for a fixed amount of time until a reasonably small error from ground truth has been achieved. The running times for each problem are indicated above. We then compare the results, which are reported in Figure \ref{fig:exp-video-ialm}. Each row represents a tested video. The first column contains sample frames from each corresponding video, then each of the following column triplets contains corresponding low rank and sparse components as returned from IALM and ML-IALM algorithms. Below each frame we also report the corresponding achieved rank and the feasibility gap (FG) i.e. $\Vert \mathbf{D}-\mathbf{L}^{\star}-\mathbf{S}^{\star}\Vert_F / \Vert \mathbf{D}\Vert_F$.

As the results indicate, both algorithms produce similar results for all videos, except the larger \textbf{copymachine} and \textbf{gates} examples, for which ML-IALM produces significantly clearer separation of background than IALM.

\begin{figure*}
\center
\setlength\tabcolsep{6pt}
\begin{tabular}{|c|c|c|c|c|c|}\hline
  \multirow{2}{*}{Original} & \multicolumn{2}{|c|}{Low Rank} & \multicolumn{2}{|c|}{Sparse} \\ 
  \cline{2-5}
   & IALM & ML-IALM & IALM & ML-IALM \\ \hline
  \includegraphics[scale=0.12]{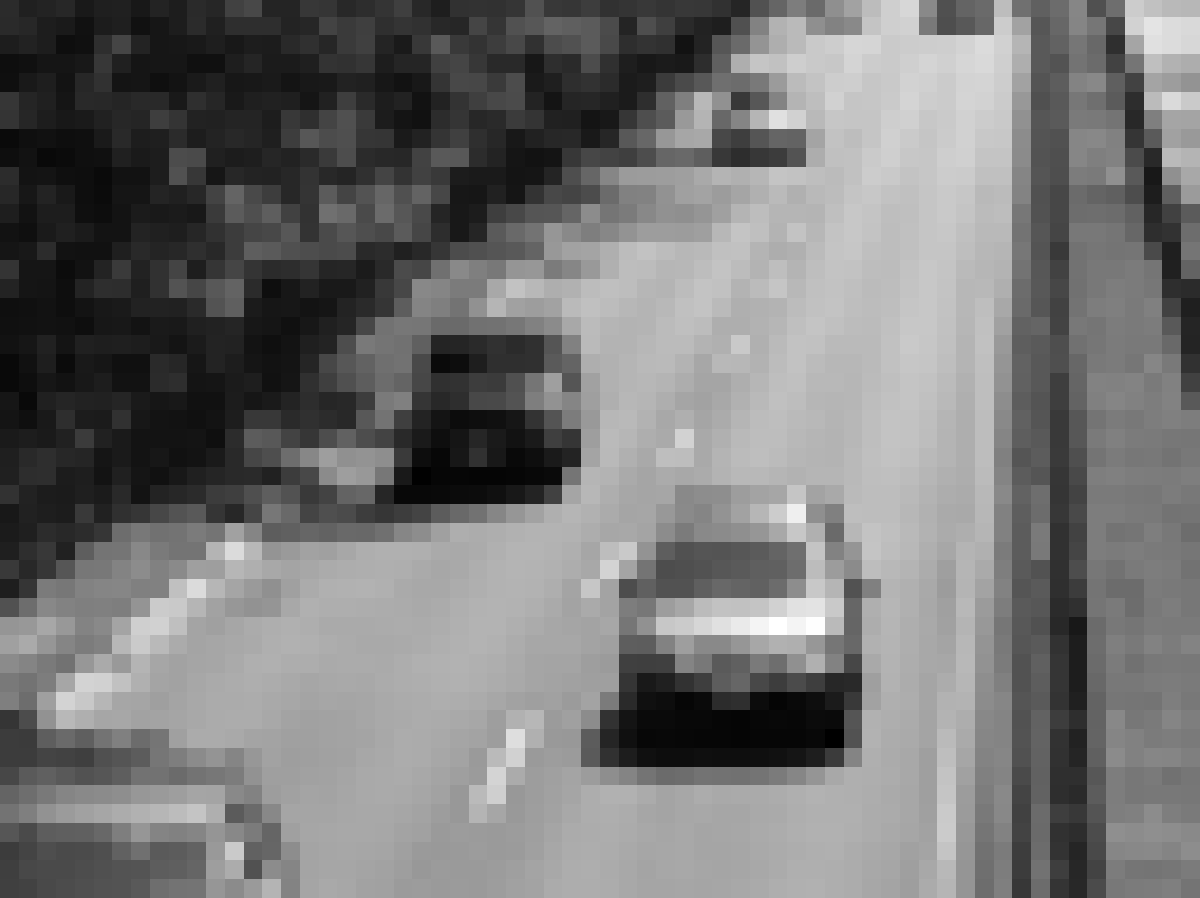} & \includegraphics[scale=0.12]{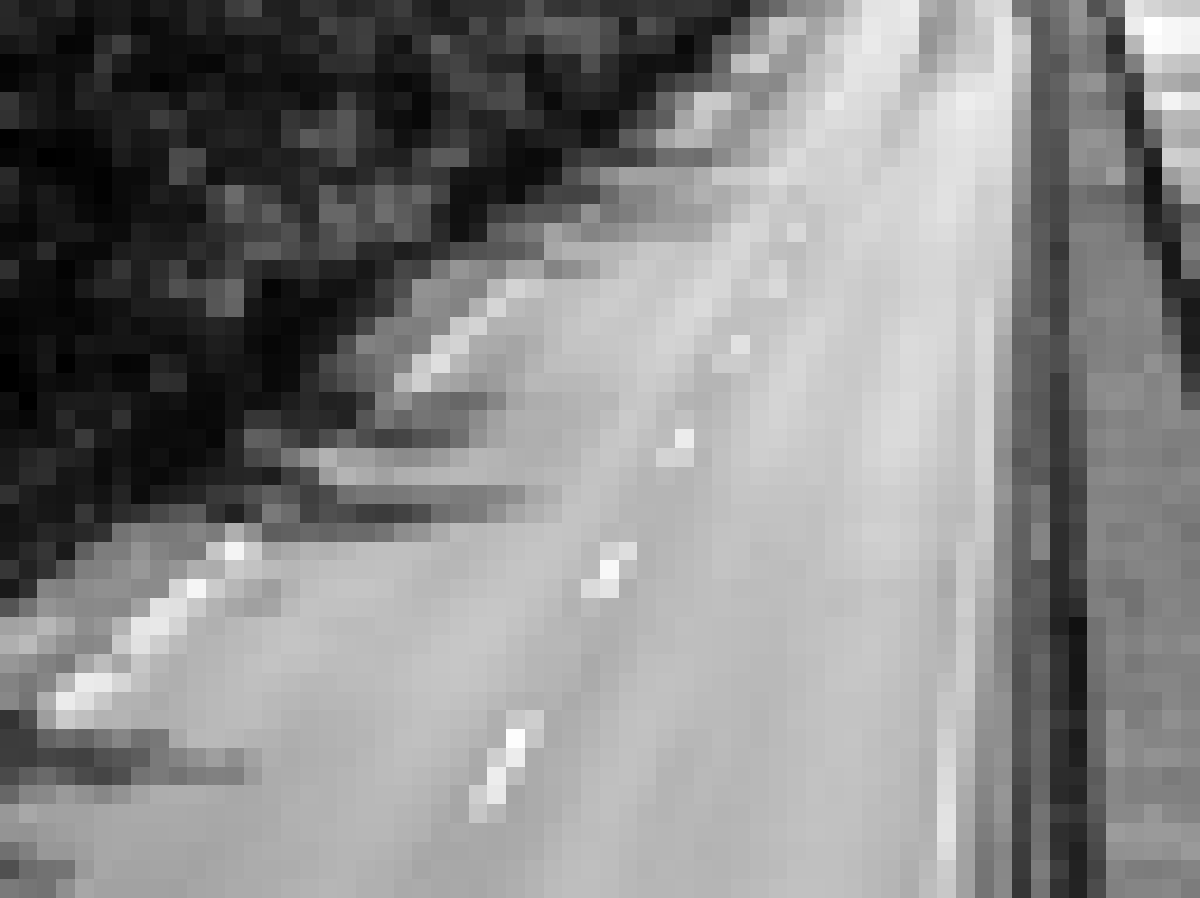} & \includegraphics[scale=0.12]{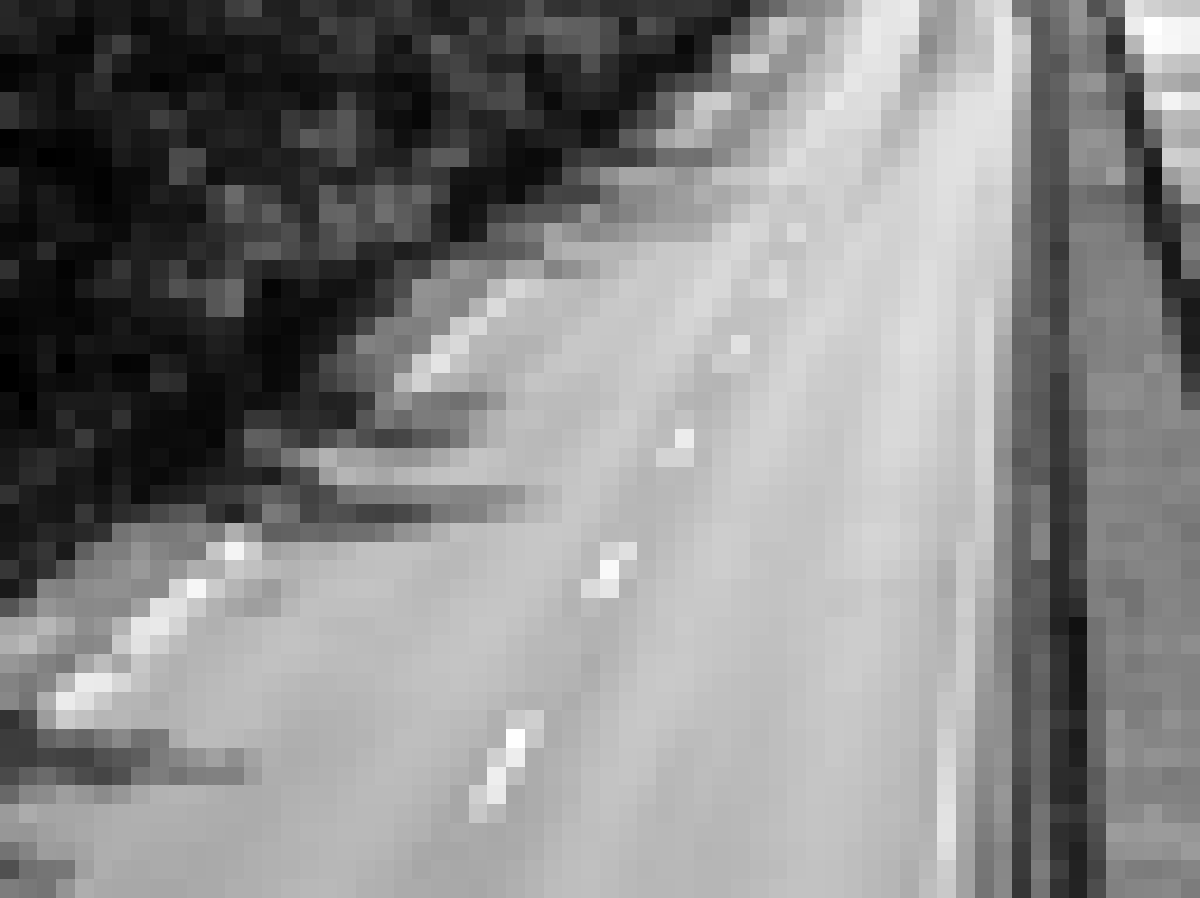} & \includegraphics[scale=0.12]{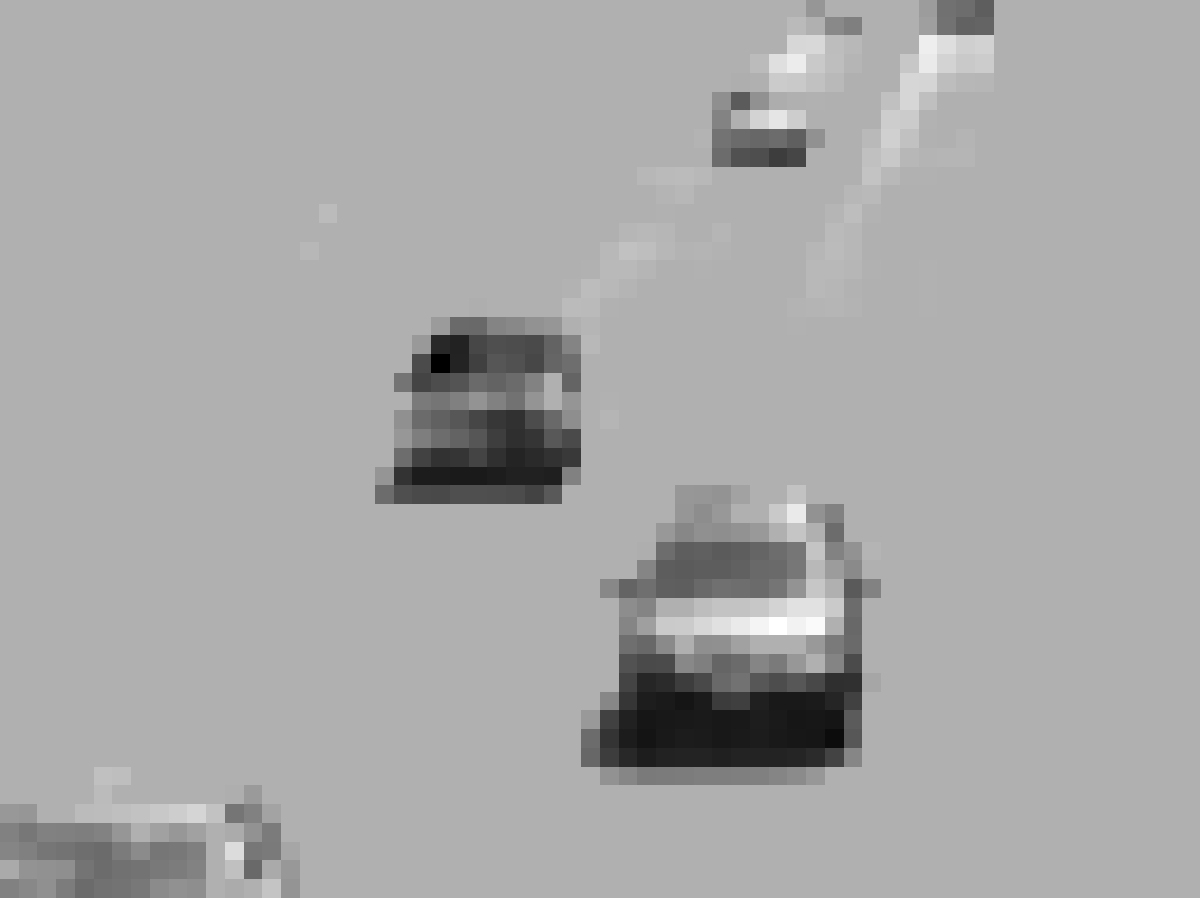} & \includegraphics[scale=0.12]{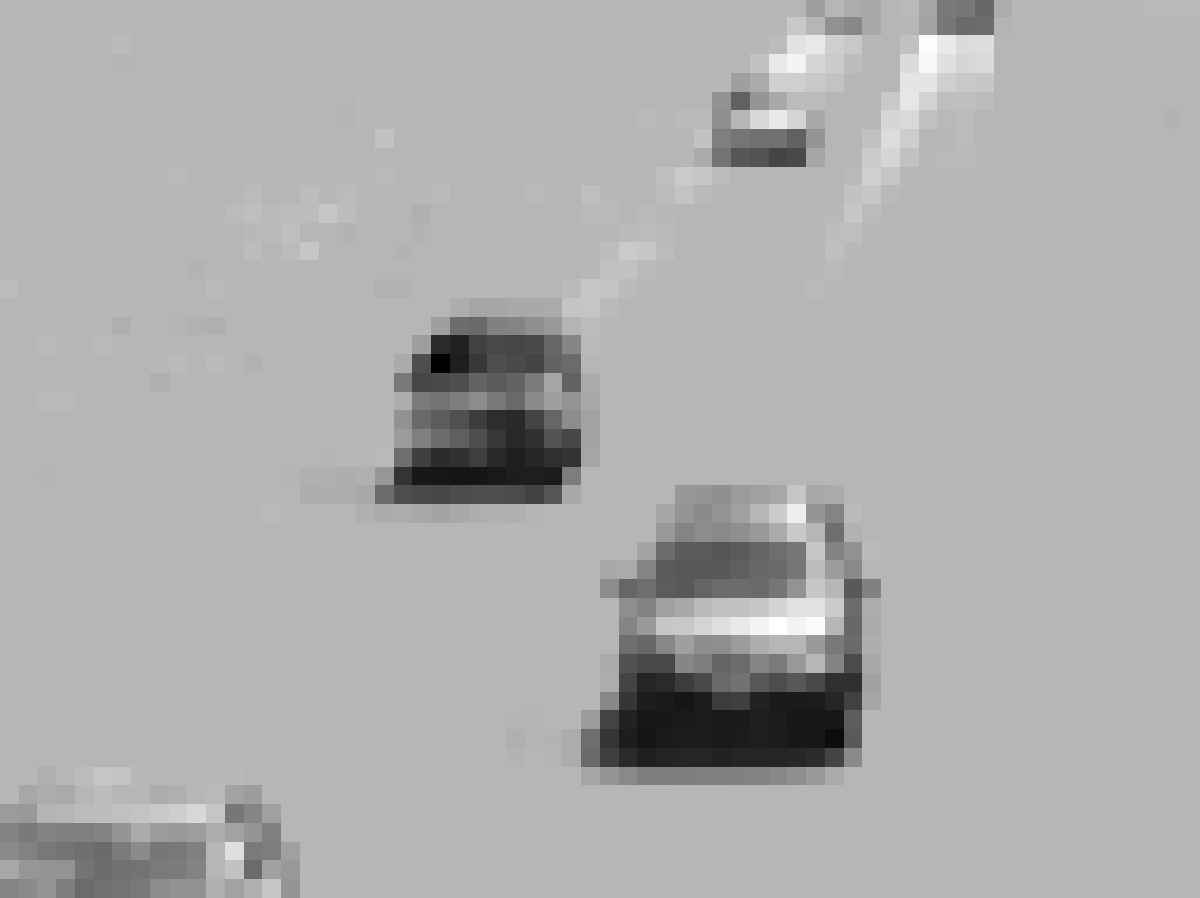} \\ \hline
  \textbf{highway} & rank =$5$ & rank=$3$ & FG = $0.0176$ & FG=$0.01$ \\ \hline
  \includegraphics[scale=0.12]{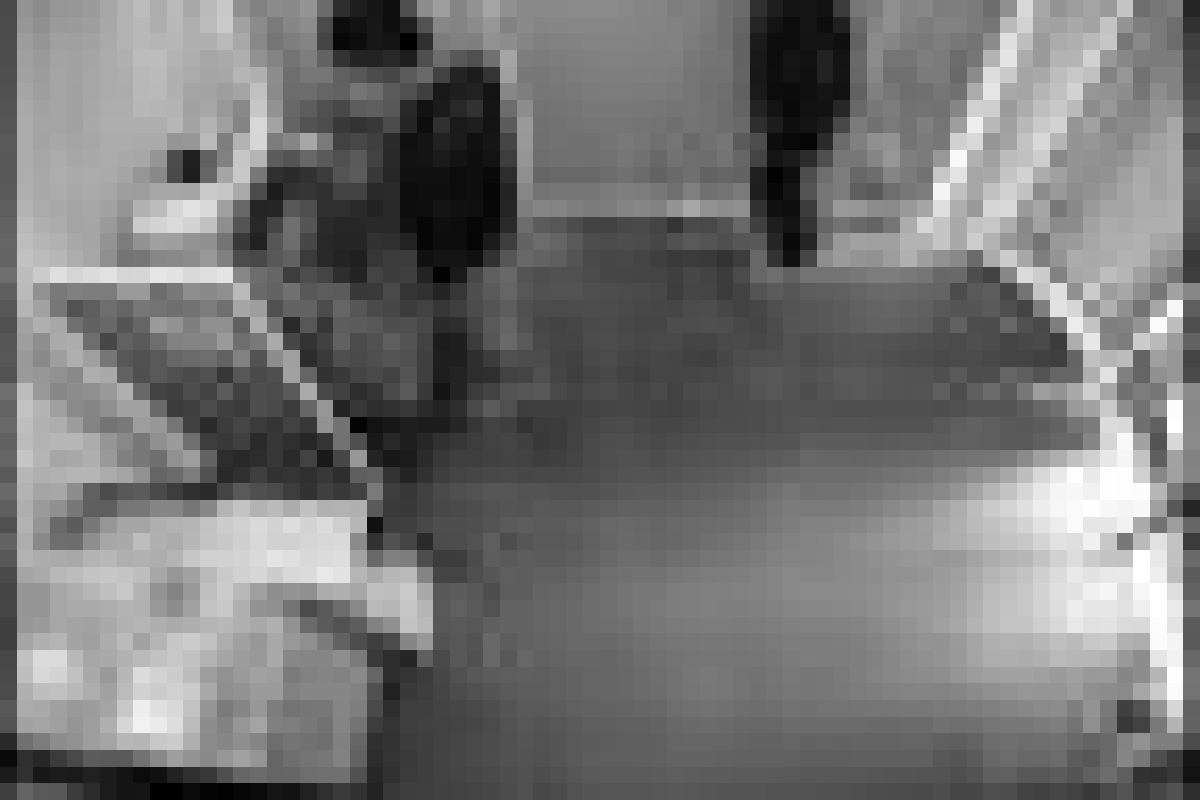} & \includegraphics[scale=0.12]{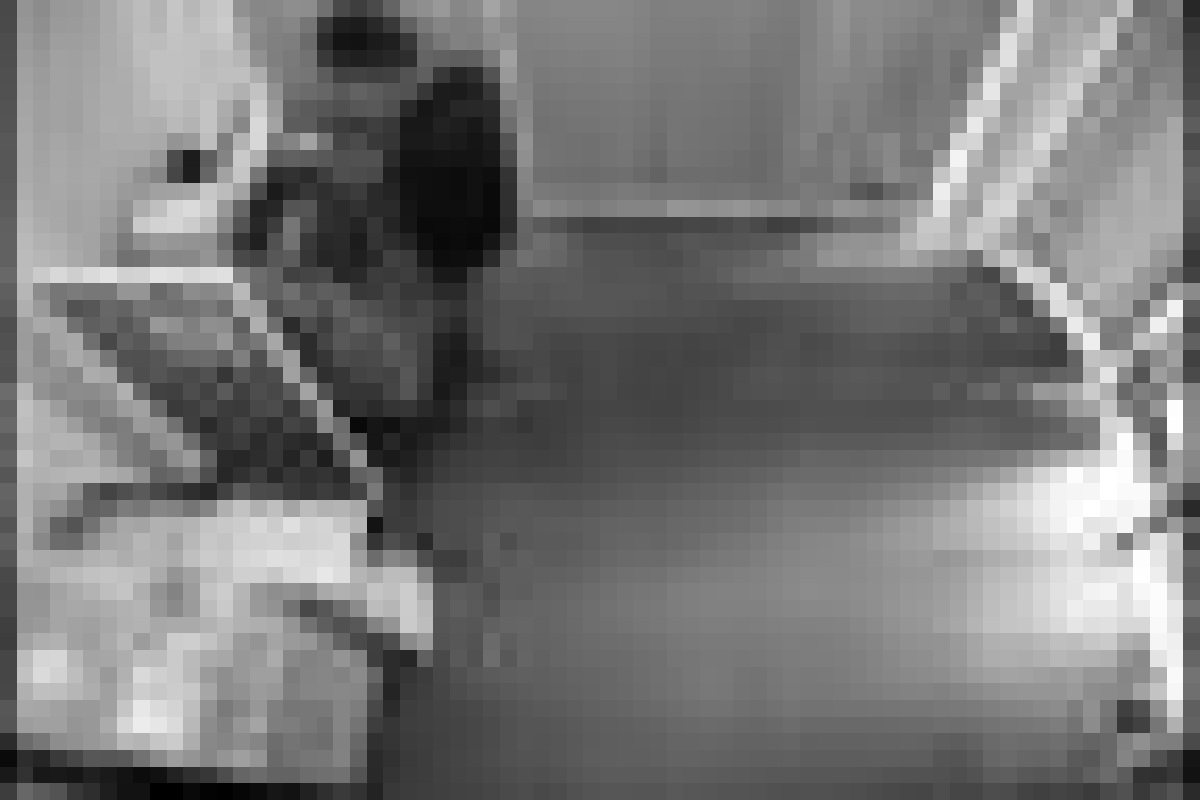} & \includegraphics[scale=0.12]{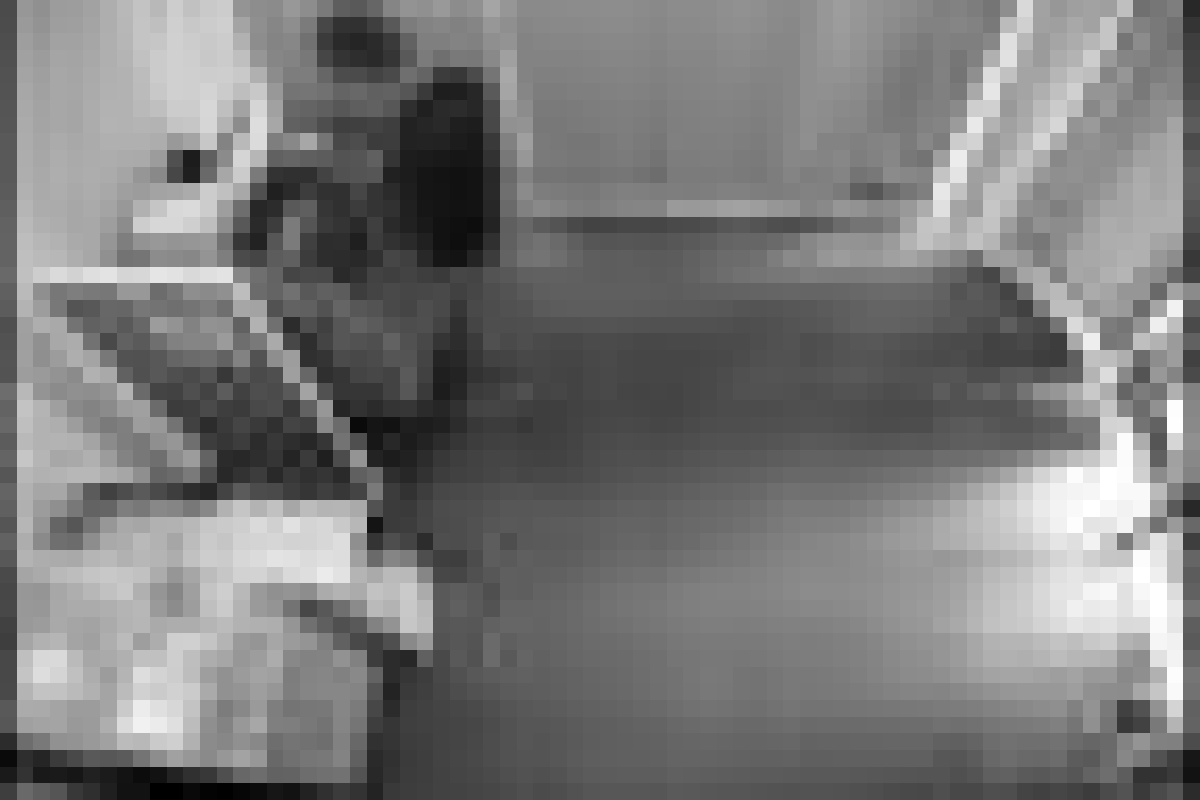} & \includegraphics[scale=0.12]{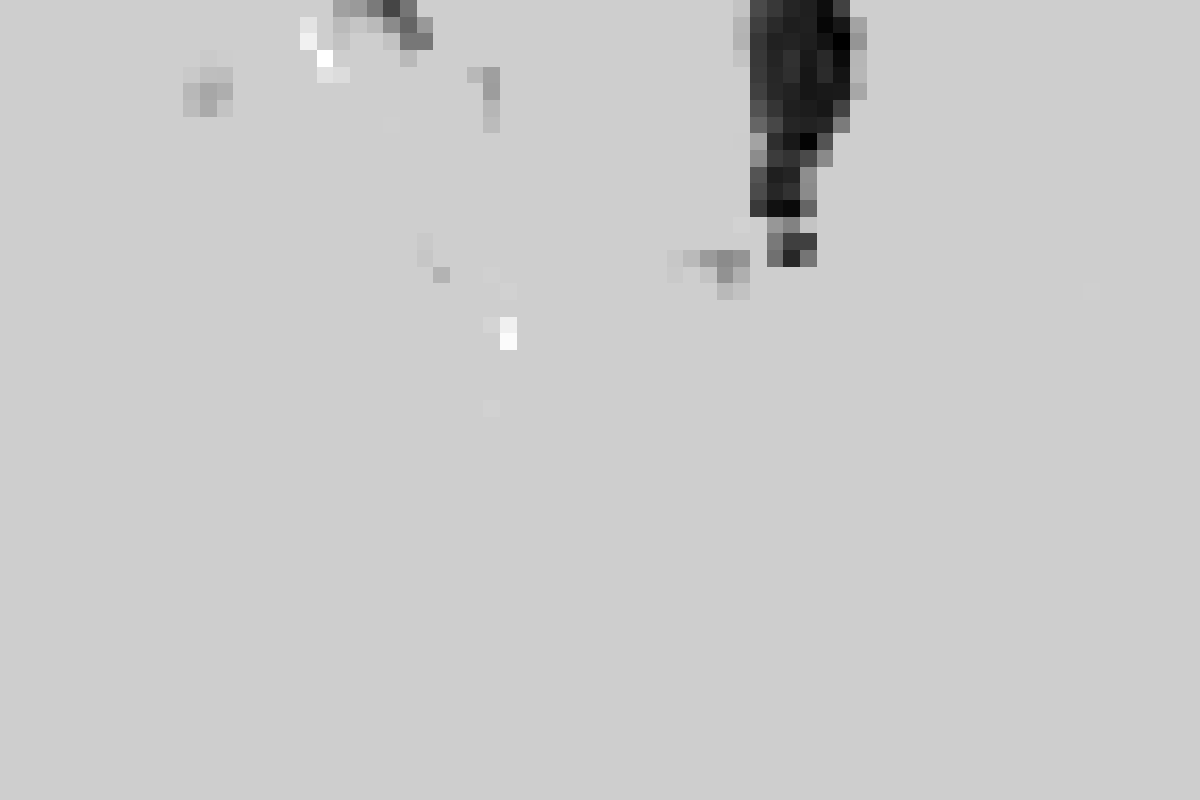} & \includegraphics[scale=0.12]{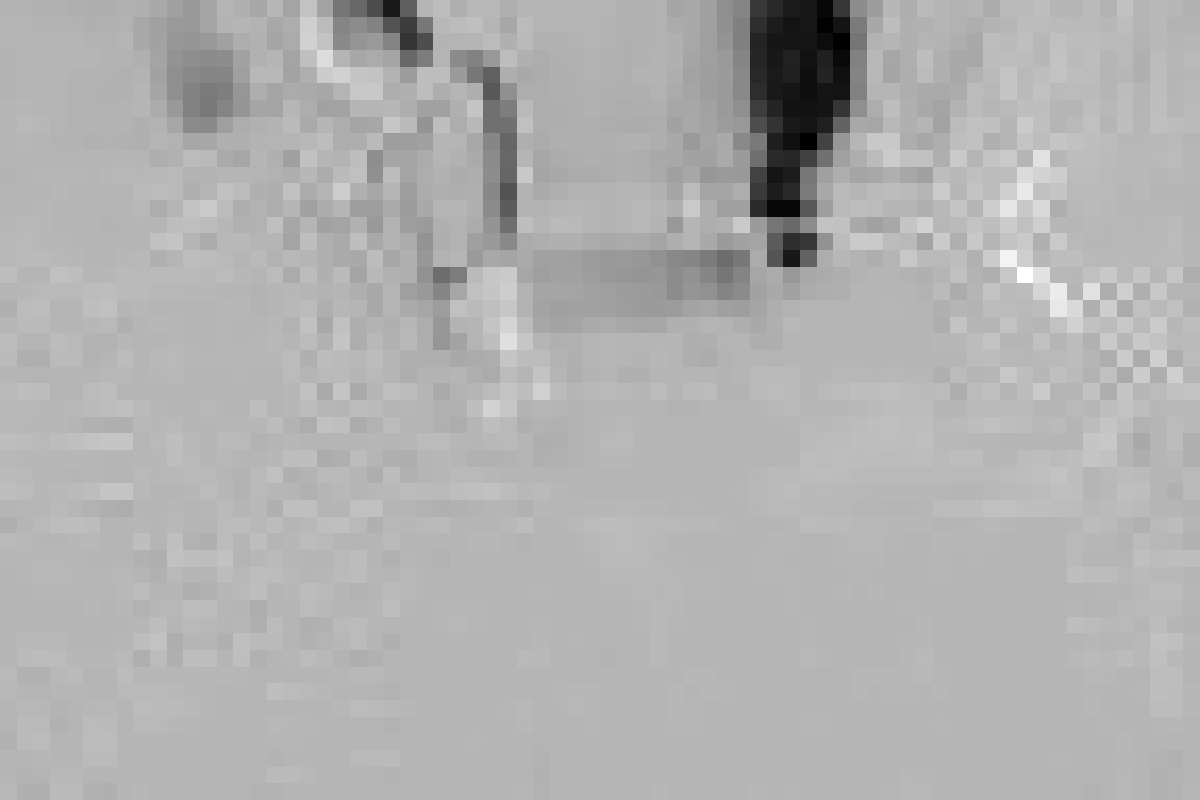} \\ \hline
  \textbf{copy machine} & rank = $7$ & rank=$4$ & FG = $0.0364$ & FG=$0.0031$ \\ \hline
  \includegraphics[scale=0.12]{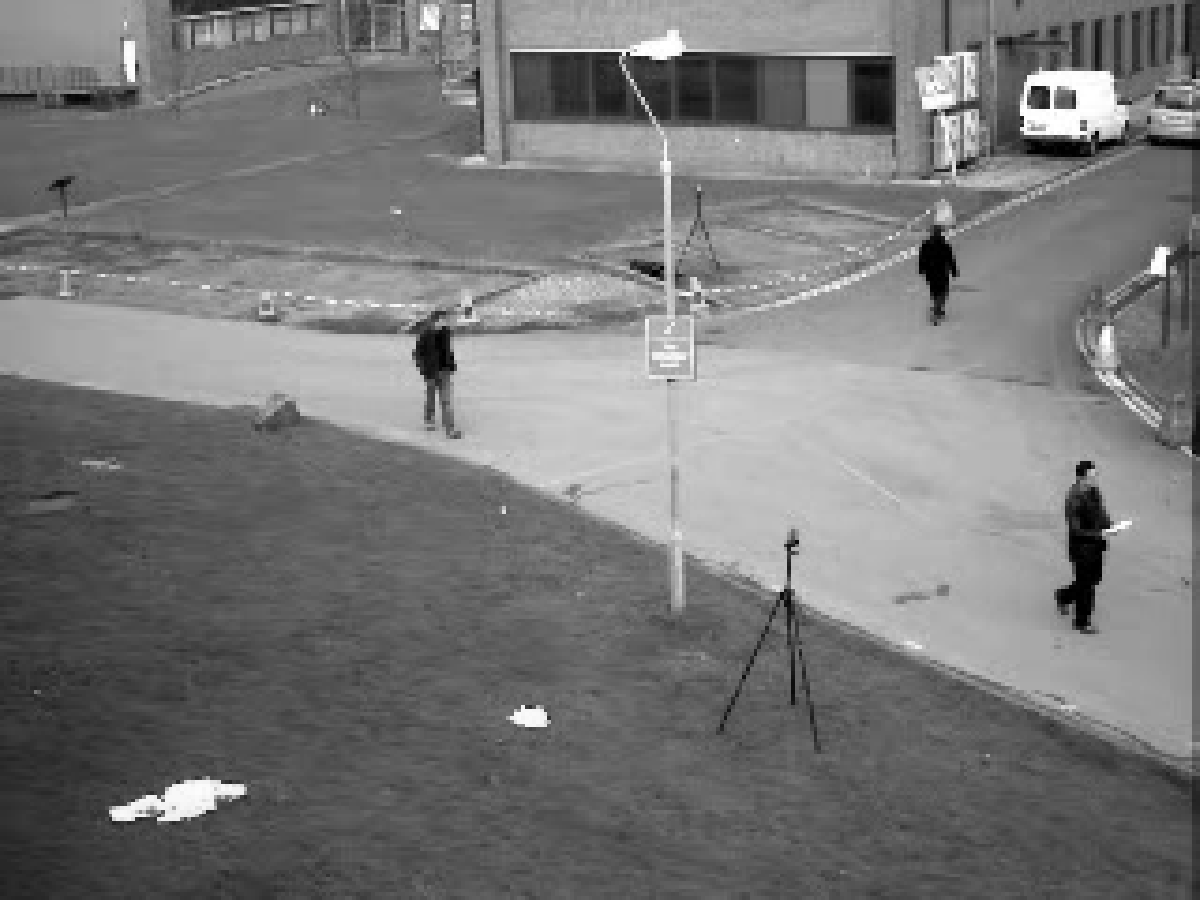} & \includegraphics[scale=0.12]{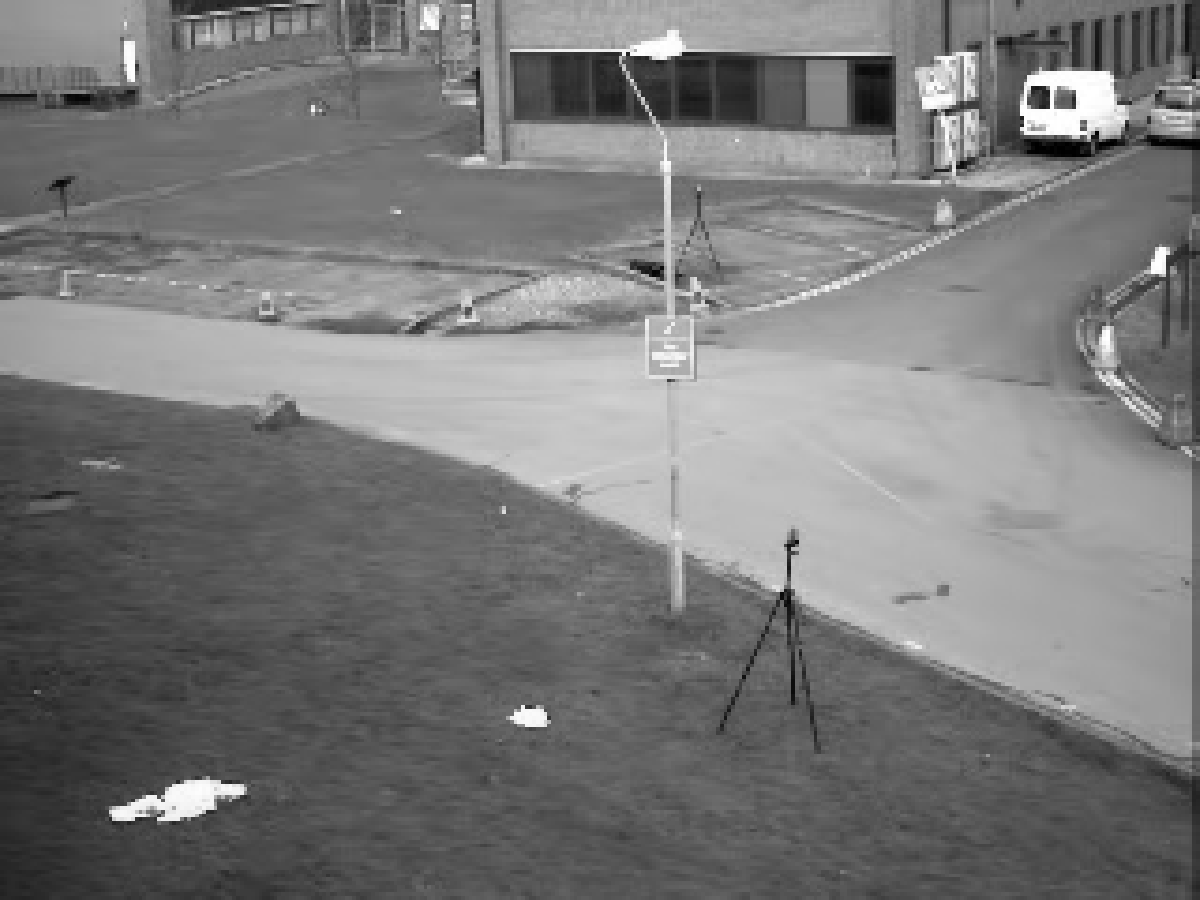} & \includegraphics[scale=0.12]{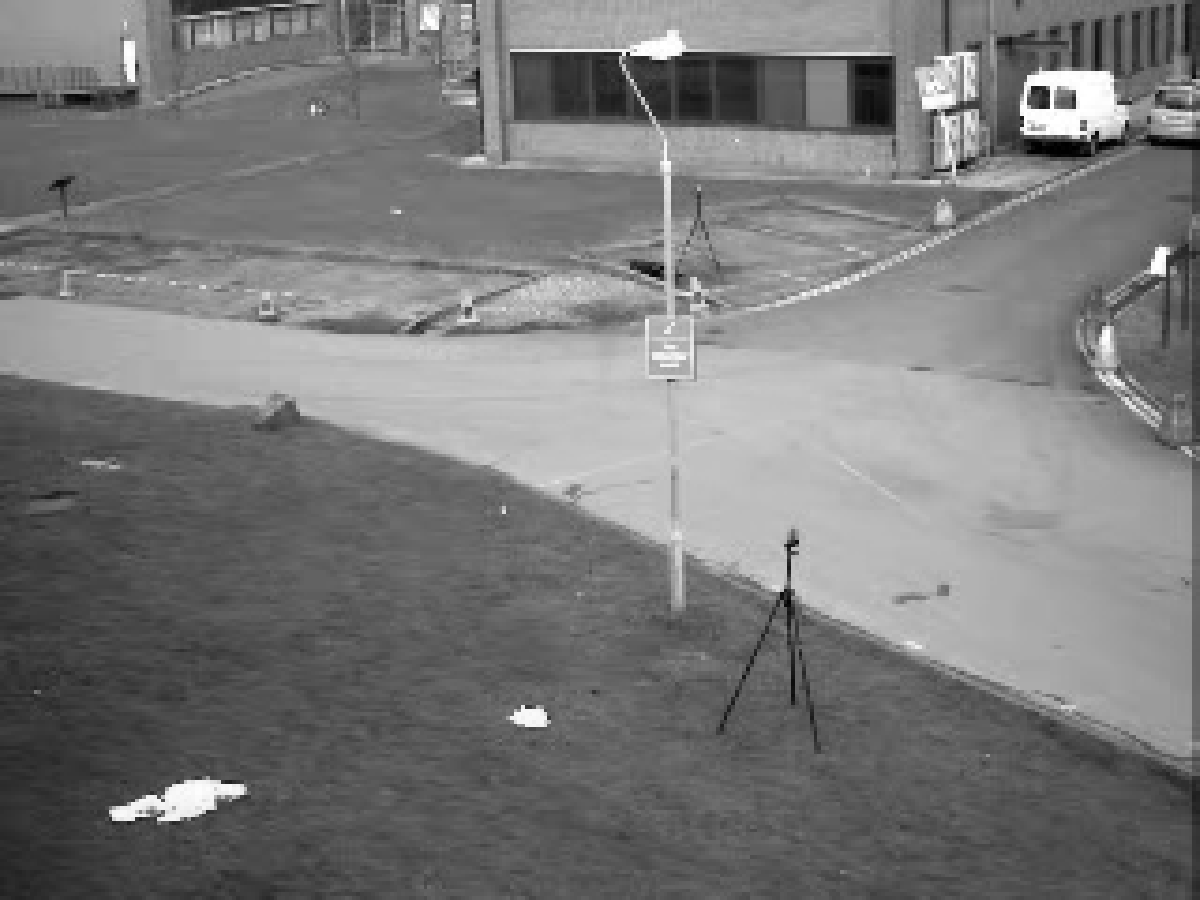} & \includegraphics[scale=0.12]{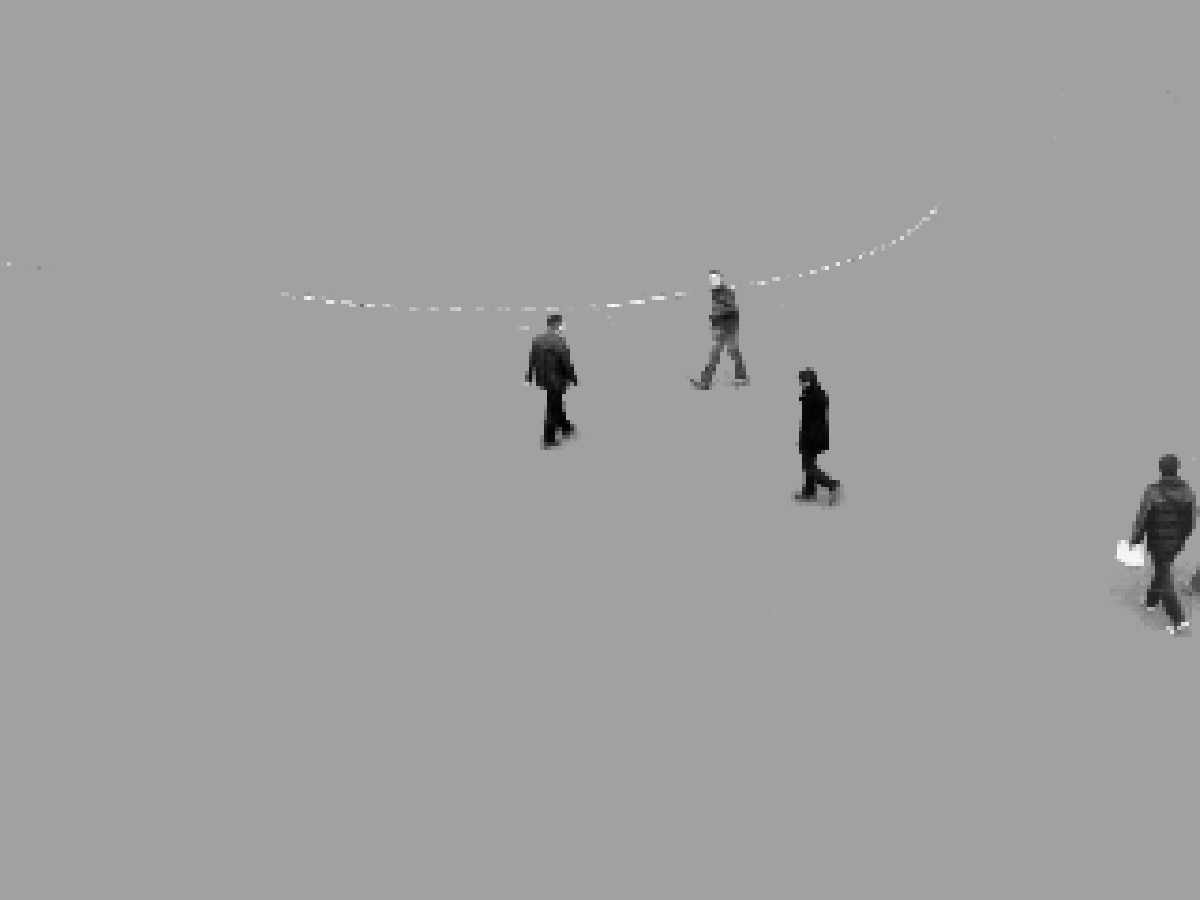} & \includegraphics[scale=0.12]{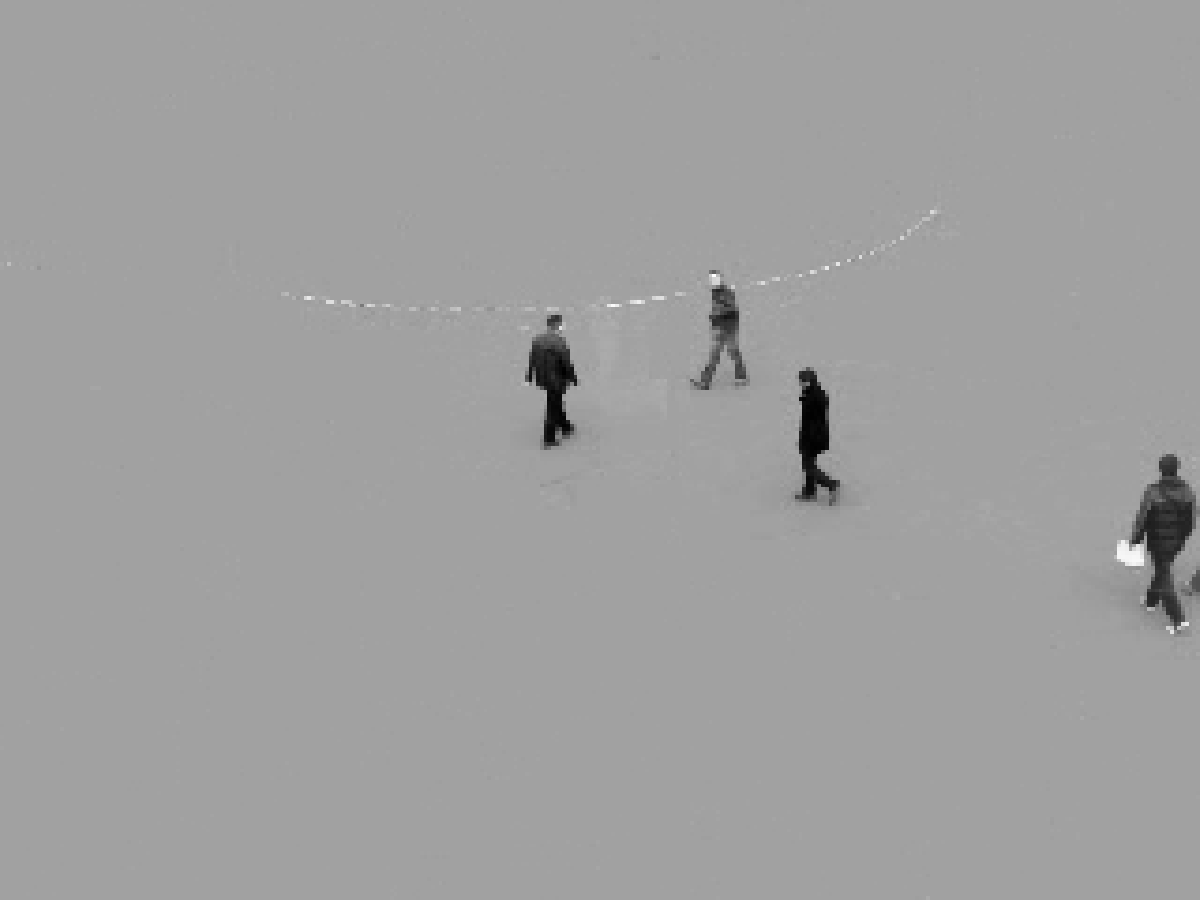} \\ \hline
  \textbf{walk} & rank = $2$ & rank=$1$ & FG = $0.02$ & FG=$0.0231$ \\ \hline
  \includegraphics[scale=0.12]{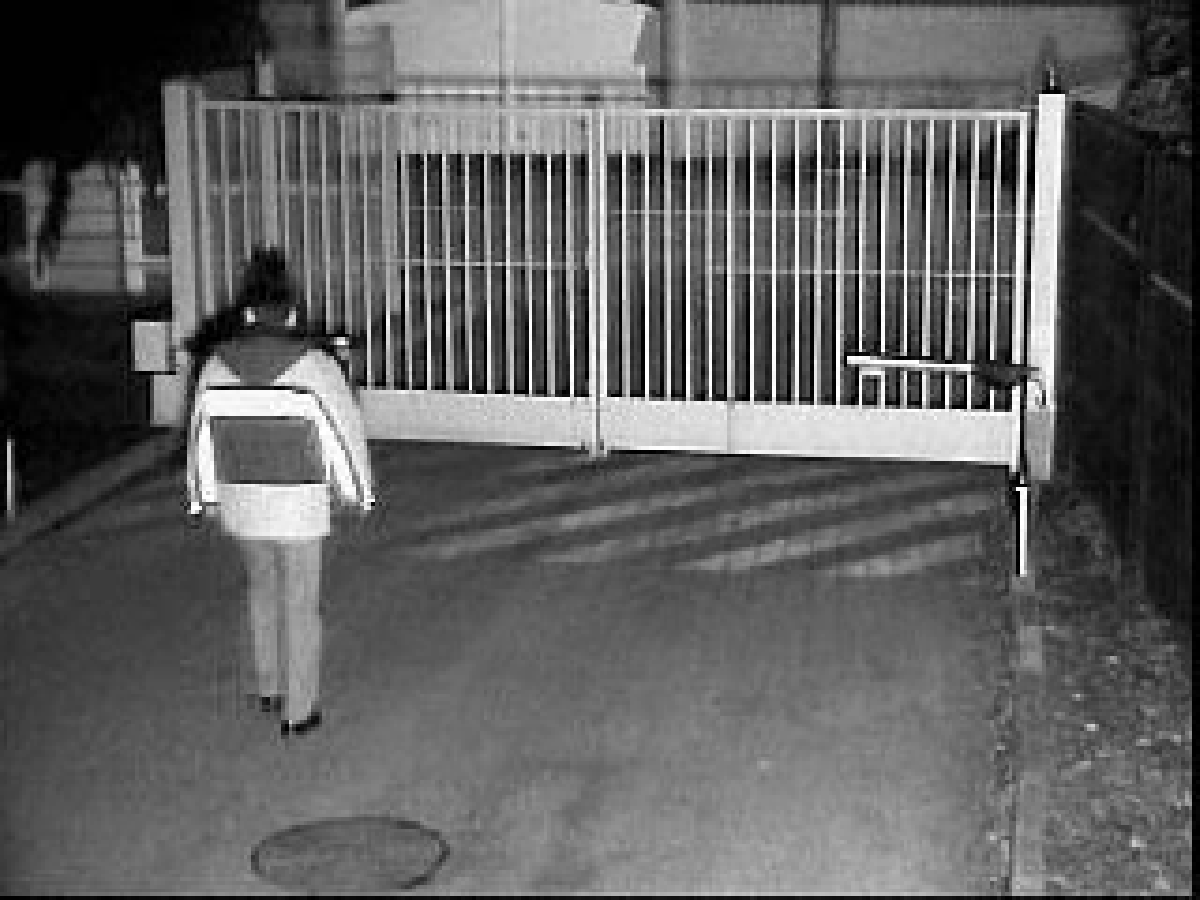} & \includegraphics[scale=0.12]{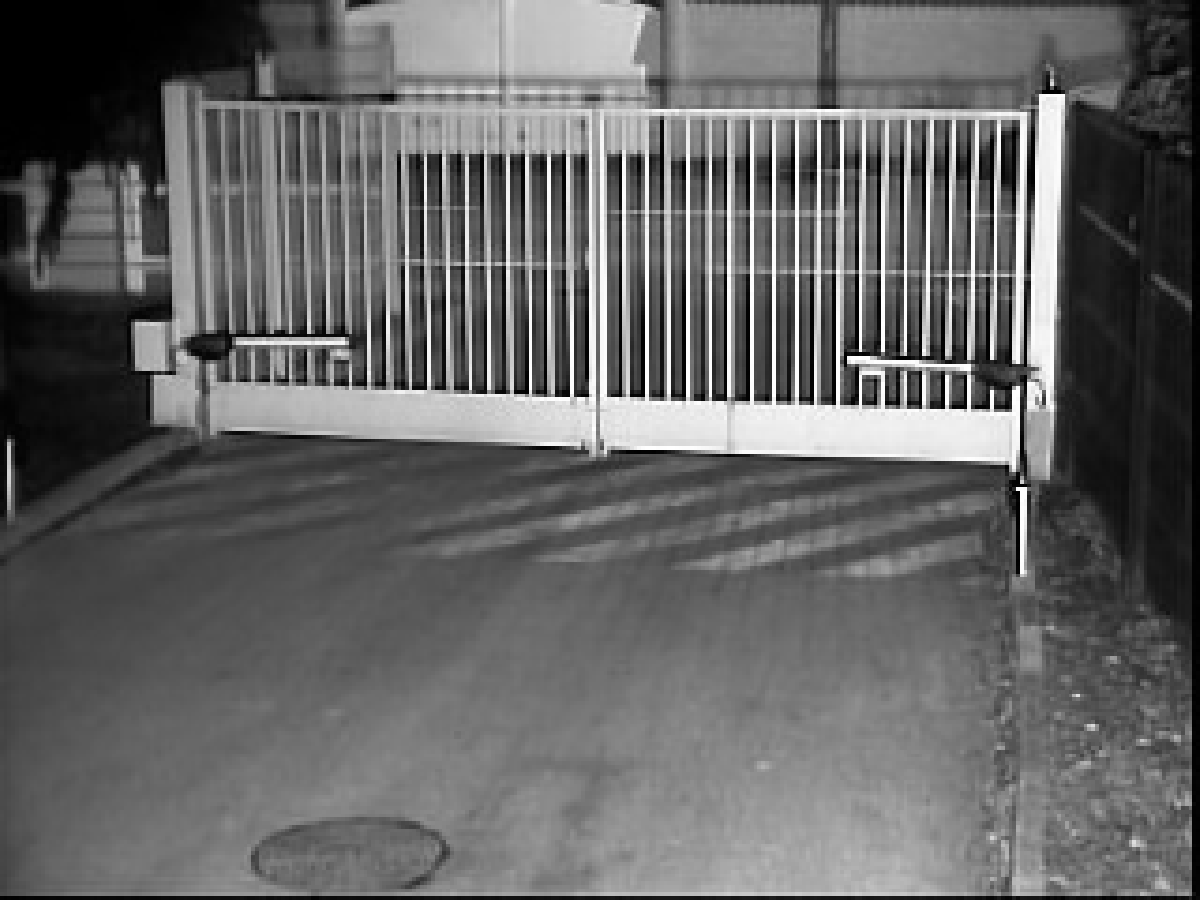} & \includegraphics[scale=0.12]{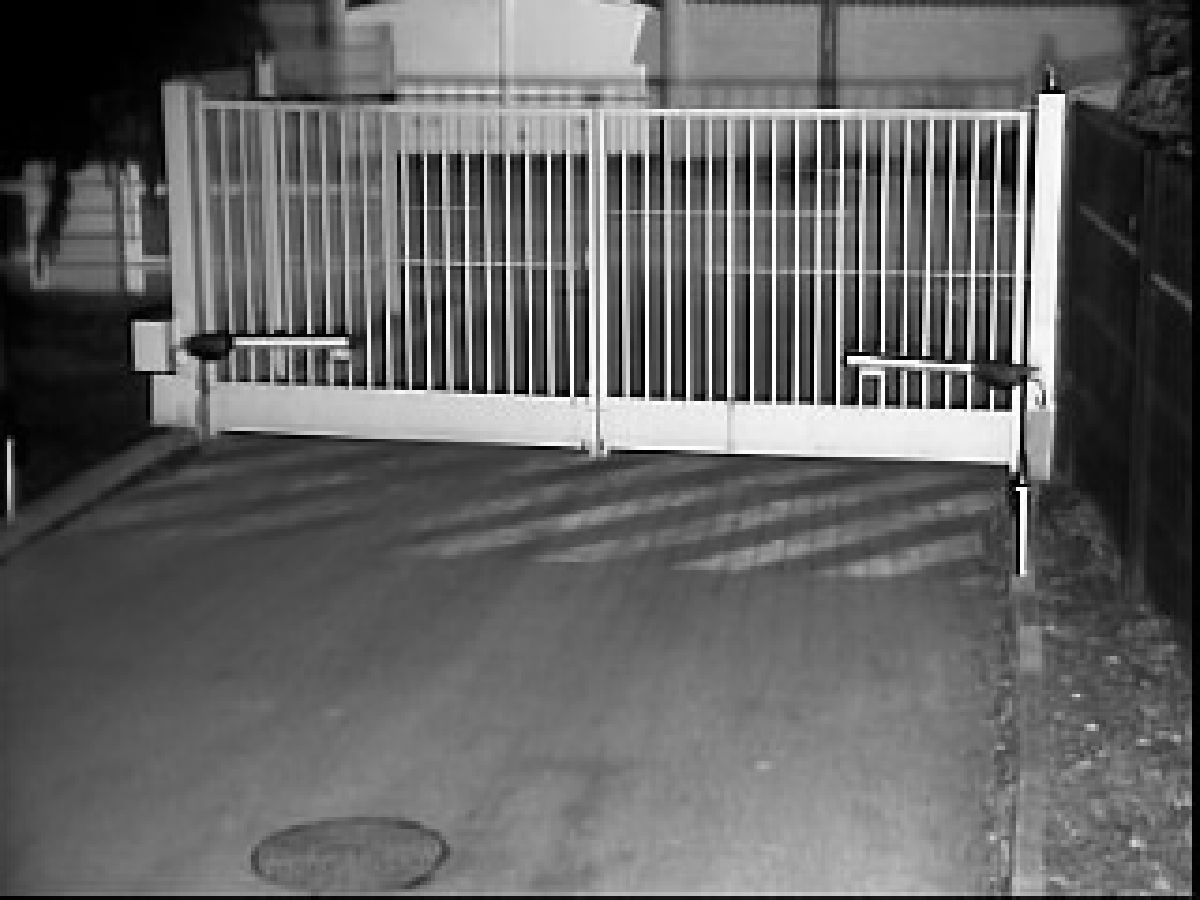} & \includegraphics[scale=0.12]{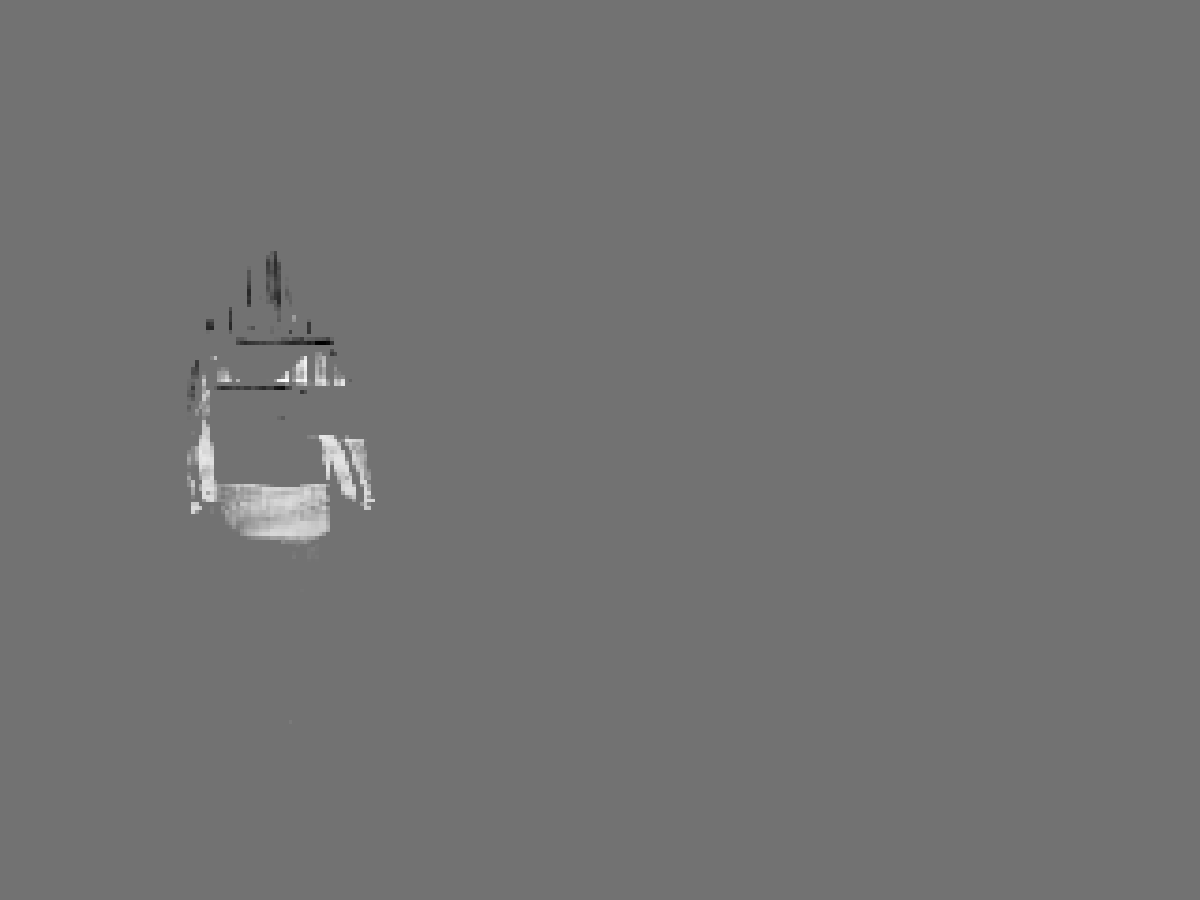} & \includegraphics[scale=0.12]{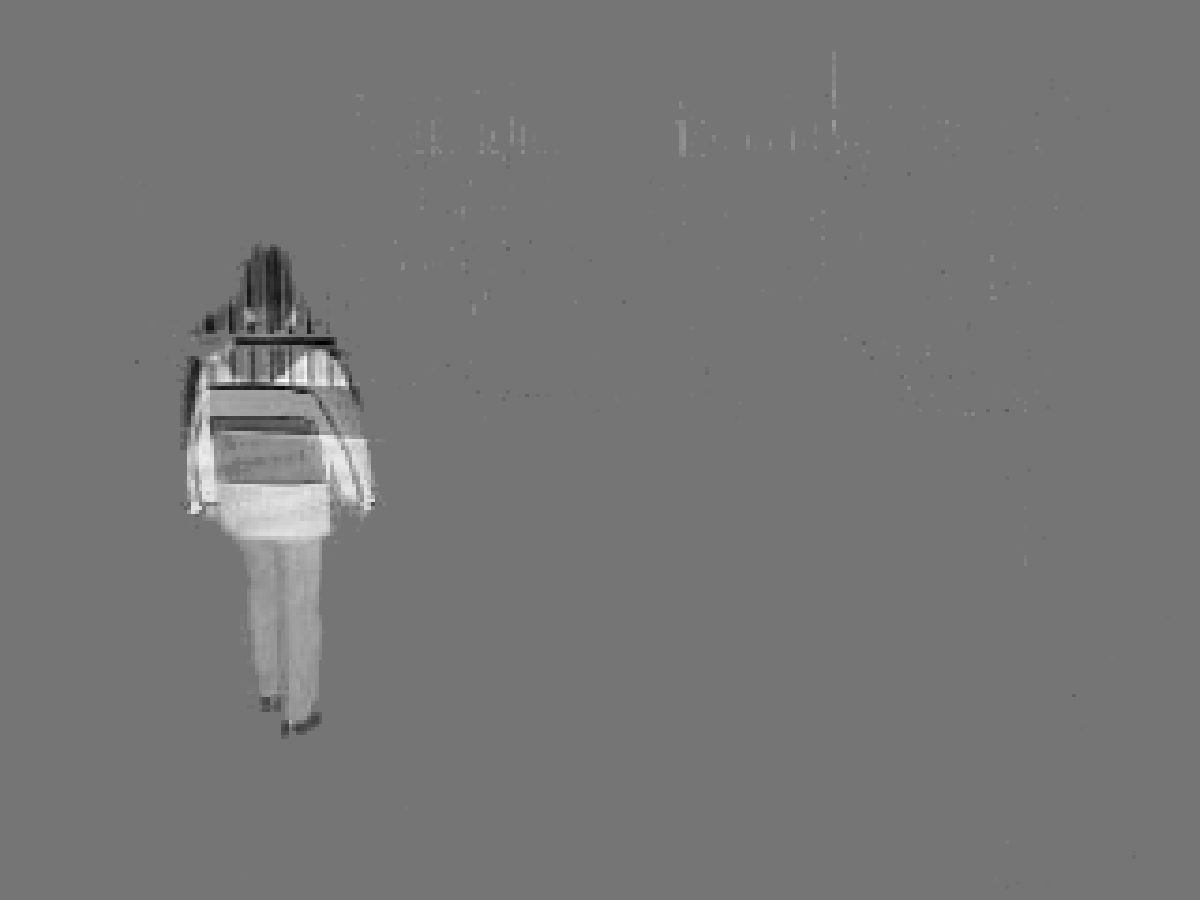} \\ \hline
  \textbf{gates} & rank = $3$ & rank=$3$ & FG = $0.05$ & FG=$0.04$ \\ \hline
  \end{tabular}
\caption{Examples from solving video background extraction problems via IALM and ML-IALM methods. Both IALM and ML-IALM run for a fixed CPU seconds. Each row corresponds respectively to \textbf{highway} ($48 \times 64\times 400$), \textbf{copy machine} ($ 48\times 72\times 3,400$), \textbf{walk} ($ 240\times 320\times 794$) and \textbf{gates} ($ 240\times 320\times 1,895$) videos from top to bottom. With each frame we also report the respective rank of the low rank component and the feasibility gap (FG): $\Vert \mathbf{D}-\mathbf{L}-\mathbf{S}\Vert_F / \Vert\mathbf{D}\Vert_F$.}
\label{fig:exp-video-ialm}
\end{figure*}

In order to further investigate the convergence properties of the ML-IALM algorithm compared to the standard IALM, we measure the relative error of the current iterates compared to the ground truth $(\mathbf{L}_0, \mathbf{S}_0)$ and FGs during the iterations of both algorithms through the same time interval. We report those relative errors against CPU time (seconds) and iteration numbers in Figure \ref{fig:exp-video-plots}. 
The plots suggest that ML-IALM performs only slightly faster than IALM on the smaller \textbf{highway} example, however, as expected it is significantly faster on the larger \textbf{copy machine} problem. As we could anticipate from the theory, at each iteration ML-IALM achieves a very good approximation as measured by the reconstruction error, and since its iterations are significantly cheaper, it performs more iterations during the same time interval than IALM.

\begin{figure*}
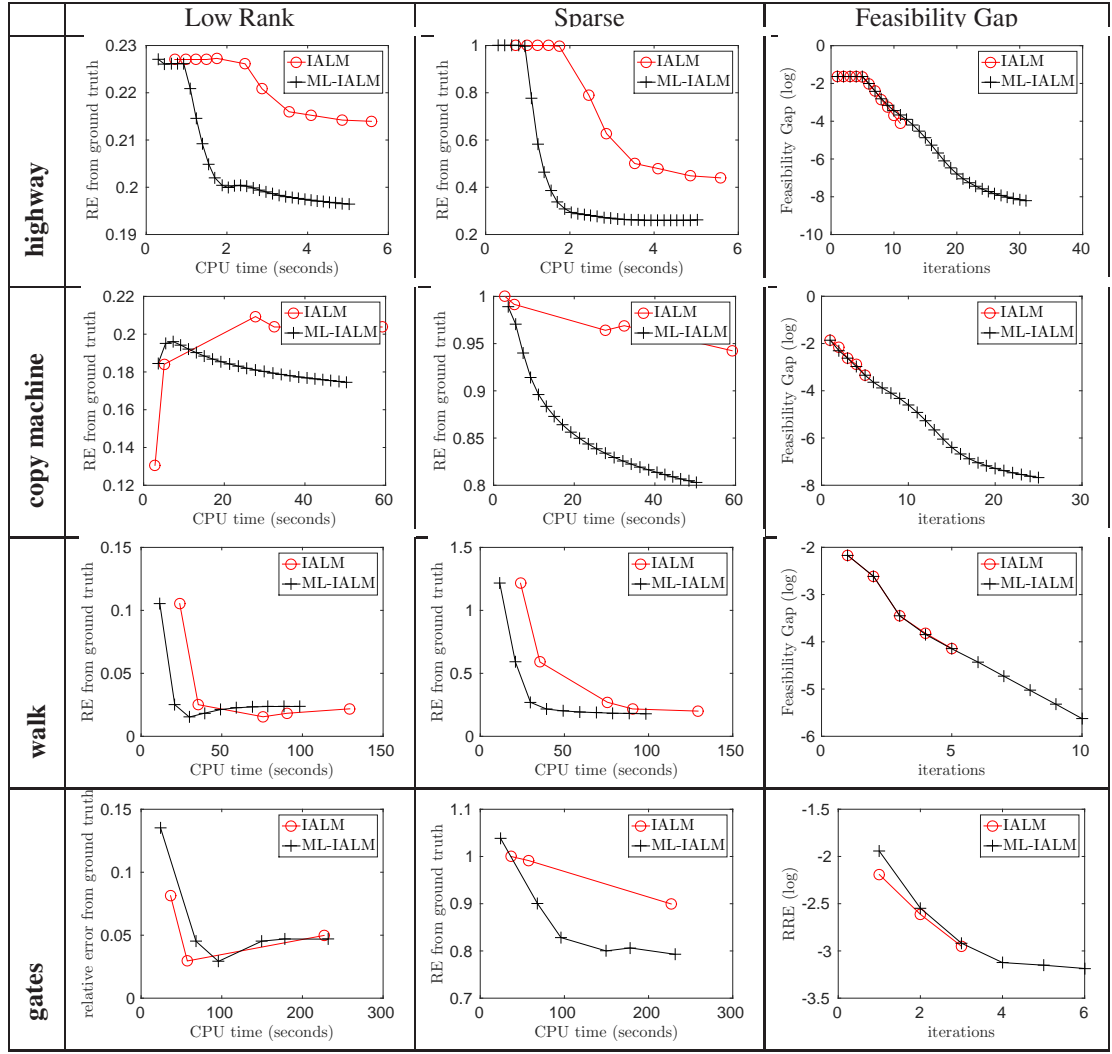

\centering
\begin{tabular}{|c|c|c|c|}\hline
   & Low Rank & Sparse & Feasibility Gap \\ \hline
  \rot{\rlap{\textbf{~~~highway}}} & \includegraphics[scale=0.22]{img/highway_errL.eps} & \includegraphics[scale=0.22]{img/highway_errS.eps} & \includegraphics[scale=0.22]{img/highway_frob_err.eps} \\ \hline
  \rot{\textbf{~~~copy machine}} & \includegraphics[scale=0.22]{img/copymachine_errL.eps} & \includegraphics[scale=0.22]{img/copymachine_errS.eps} & \includegraphics[scale=0.22]{img/copymachine_frob_err.eps} \\ \hline
  \rot{\textbf{~~~walk}} & \includegraphics[scale=0.22]{img/walk_errL.eps} & \includegraphics[scale=0.22]{img/walk_errS.eps} & \includegraphics[scale=0.22]{img/walk_frob_err.eps} \\ \hline
  \rot{\textbf{~~~gates}} & \includegraphics[scale=0.22]{img/gates_errL.eps} & \includegraphics[scale=0.22]{img/gates_errS.eps} & \includegraphics[scale=0.22]{img/gates_frob_err.eps}\\ \hline
\end{tabular}
\caption{Comparing the relative errors during IALM and ML-IALM iterations. The first two columns give relative errors (RE) compared to the ground truth $(\mathbf{L}_0, \mathbf{S}_0)$, and the third column gives feasibility gaps (FG) during iterations. Each row corresponds respectively to \textbf{highway} ($48 \times 64\times 400$), \textbf{copy machine} ($ 48\times 72\times 3,400$), \textbf{walk} ($ 240\times 320\times 794$) and \textbf{gates} ($ 240\times 320\times 1,895$) videos from top to bottom.}
\label{fig:exp-video-plots}
\end{figure*}

In all experiments we used $4$ levels of coarse models for all four problems. In this case as well, the multilevel variant largely outperforms the original algorithm. In fact, the larger the original problem, the bigger relative speed up can be achieved using the multilevel approach, since for larger $n$ we can use deeper levels.

Next we test the performance of our ML-FWT algorithm against the standard FWT. In this case we will also add $75\%$ random noise to the original video. Here we run both algorithms until convergence with $10^{-3}$ accuracy as suggested in \cite{mu2016scalable}. Consequently, here we will report the running times and the achieved results (objective value, rank of the low-rank component and sparsity of the sparse component) of each algorithm in Table \ref{t:fwt}. As the numbers indicate both algorithms achieve very similar objective values and sparsity of the sparse component. However that ranks of the low rank components is better for ML-FWT. In fact for the largest problems FWT returns values with very large ranks and thus fails to solve the problem, while ML-FWT performs equally well on all problems. Furthermore, ML-FWT is much faster, especially on larger problems.

\begin{table}
\center
\begin{tabular}{|c|c|c|c|c|c|c|c|c|}
	\hline 
    \multirow{2}{*}{problem (dimensions)} & \multicolumn{4}{|c|}{FWT} & \multicolumn{4}{|c|}{ML-FWT} \\
	\cline{2-9}
	& sec & $f^\star$ & $\rank(\mathbf{L}^\star)$ & $\sparsity(\mathbf{S}^\star)$ & sec & $f^\star$ & $\rank(\mathbf{L}^\star)$ & $\sparsity(\mathbf{S}^\star)$ \\
	\hline 
	highway ($3072\times 400$) & $14$ & $0.001$ & $39$ & $0.22$ & $10$ & $0.001$ & $36$ & $0.22$ \\
	\hline 
	hall ($25344\times 200$) & $50$ & $0.001$ & $38$ & $0.47$ & $37$ & $0.001$ & $12$ & $0.47$ \\
	\hline 
	copym. ($3456\times 3400$) & $373$ & $0.001$ & $104$ & $0.11$ & $160$ & $0.001$ & $26$ & $0.17$ \\
	\hline 
	mall ($81920\times 300$)& $337$ & $0.001$ & $32$ & $0.42$ & $195$ & $0.001$ & $9$ & $0.43$ \\
	\hline 
	lobby ($20480\times 1000$)& 487 & $0.001$ & $111$ & $0.05$ & 385 & $0.001$ & $31$ & $0.06$ \\
	\hline 
\end{tabular} 
\caption{CPU time (in seconds), achieved objective value, rank and sparsity  after solving the resulting PCP problem for noisy video background extraction up to tolerance $10^{-3}$ using the standard Frank-Wolfe Thresholding (FWT) and its multilevel variant ML-FWT.}
\label{t:fwt}
\end{table}

\subsection{Shadow removal from facial images} \label{sec:shadow}

Here we have a set of facial images from one individual under various illuminations and the task is to remove shadow/light noises from images. This problem can also be modelled as RPCA by stacking the facial images as column vectors and then putting them together to form the data matrix. Then since aligned frontal facial images span a low dimensional subspace, we can represent the clear images as the low-rank component of the data matrix and the shadow will become the sparse component.

We used images of individuals from the Yale B facial extended database \cite{GeBeKr01}. It contains ($ 96\times 84$) dimensional facial images of $39$ subjects taken under various poses and illuminations each, with total $2,414$ images. For this setting as well we ran the IALM and ML-IALM algorithms for a fixed $5$ second and compare the returned results, which are reported in Figure \ref{fig:exp-face-ialm}. Here as well each row represents a particular problem setting (individual). The first column contains sample frames from each corresponding facial database, then each of the following four columns contains correspondingly low rank and sparse components as returned from IALM, and ML-IALM algorithms. With each image we also report the corresponding achieved rank of the low rank component and the feasibility gap.

\begin{figure*}
\center
\setlength\tabcolsep{6pt}
\begin{tabular}{|c|c|c|c|c|}\hline
  \multirow{2}{*}{Original} & \multicolumn{2}{|c|}{Low Rank} & \multicolumn{2}{|c|}{Sparse} \\ \cline{2-5}
   & IALM & ML-IALM & IALM & ML-IALM \\ \hline
  \includegraphics[scale=0.15]{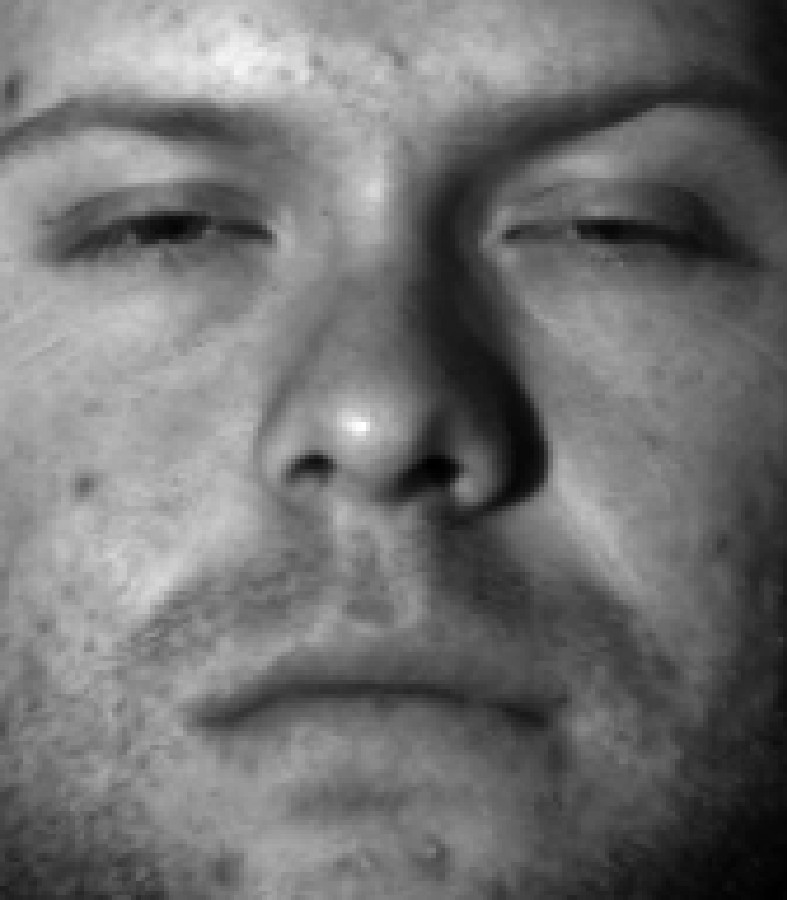} & \includegraphics[scale=0.15]{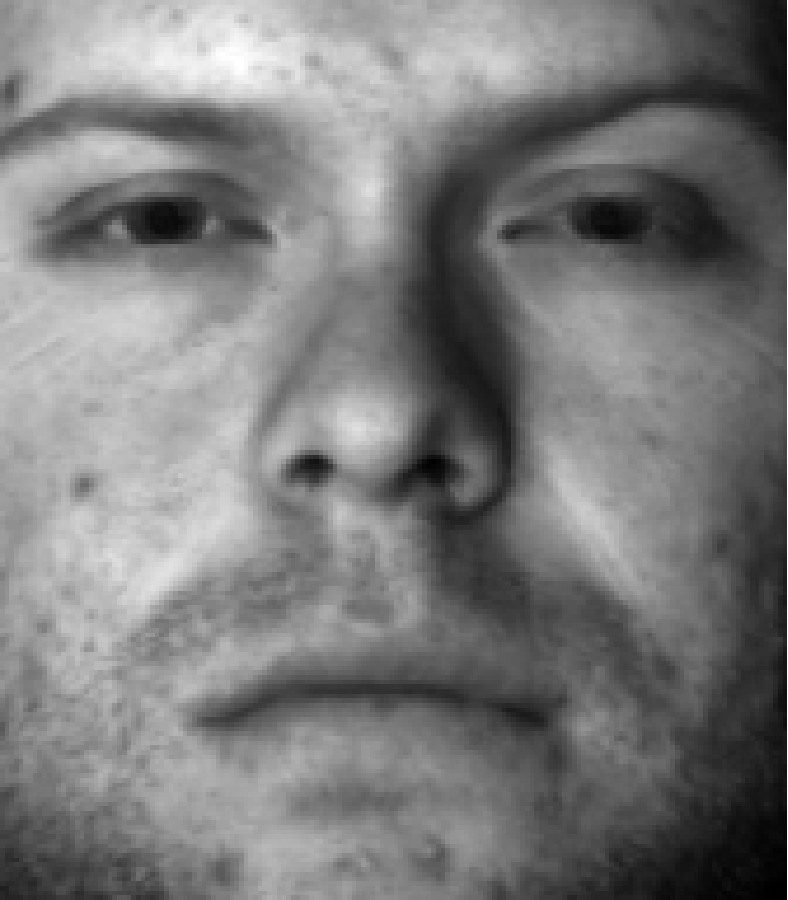} & \includegraphics[scale=0.15]{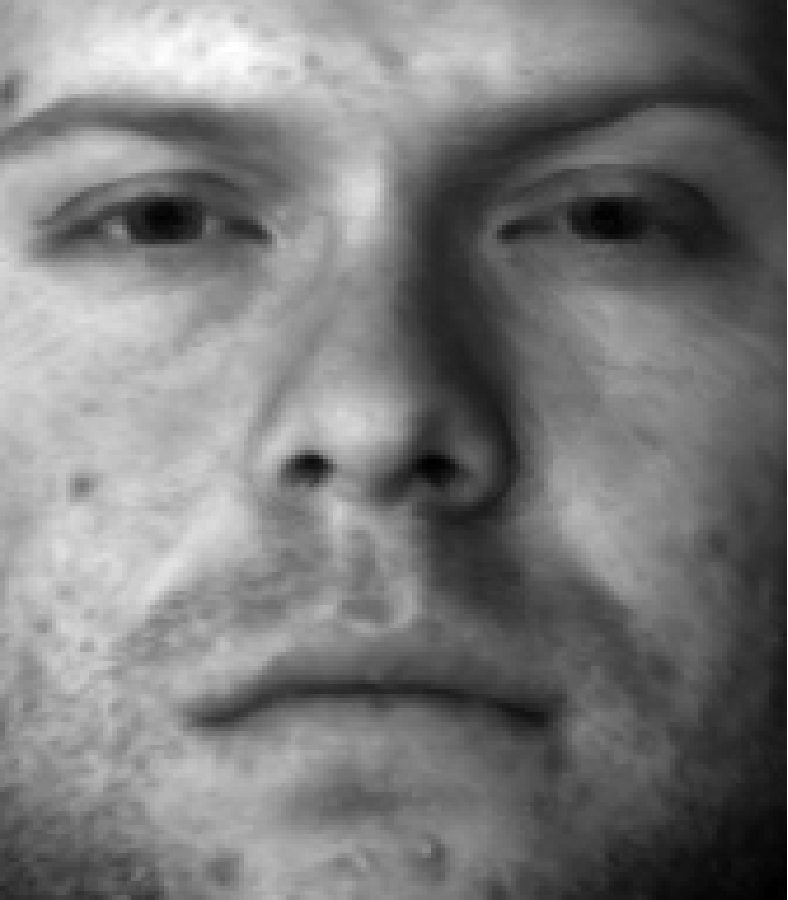} & \includegraphics[scale=0.15]{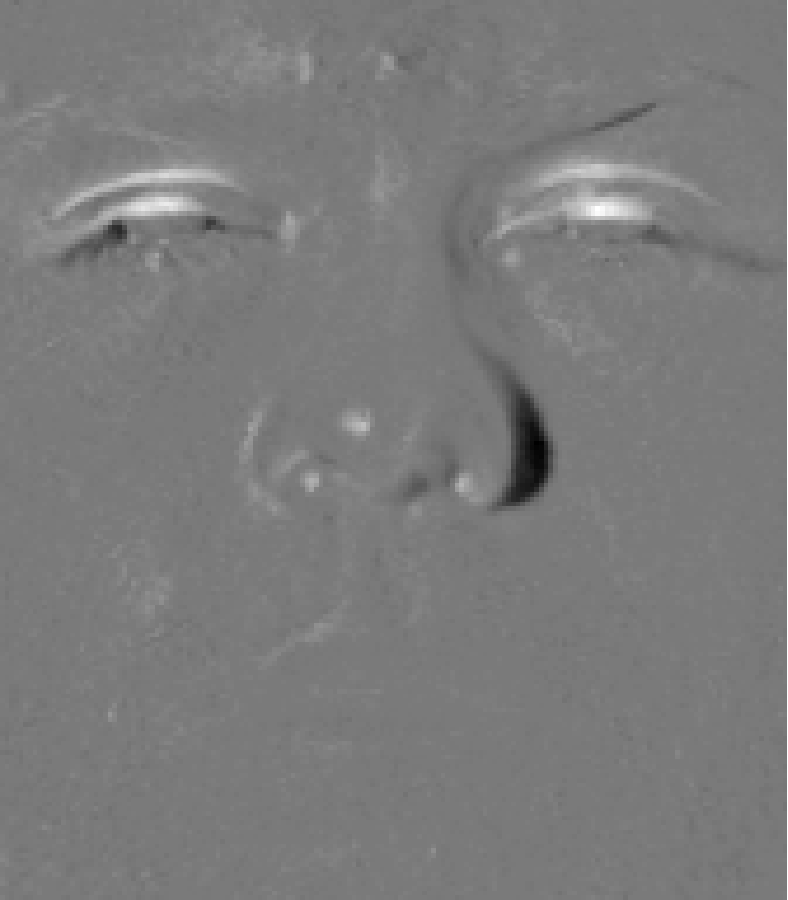} & \includegraphics[scale=0.15]{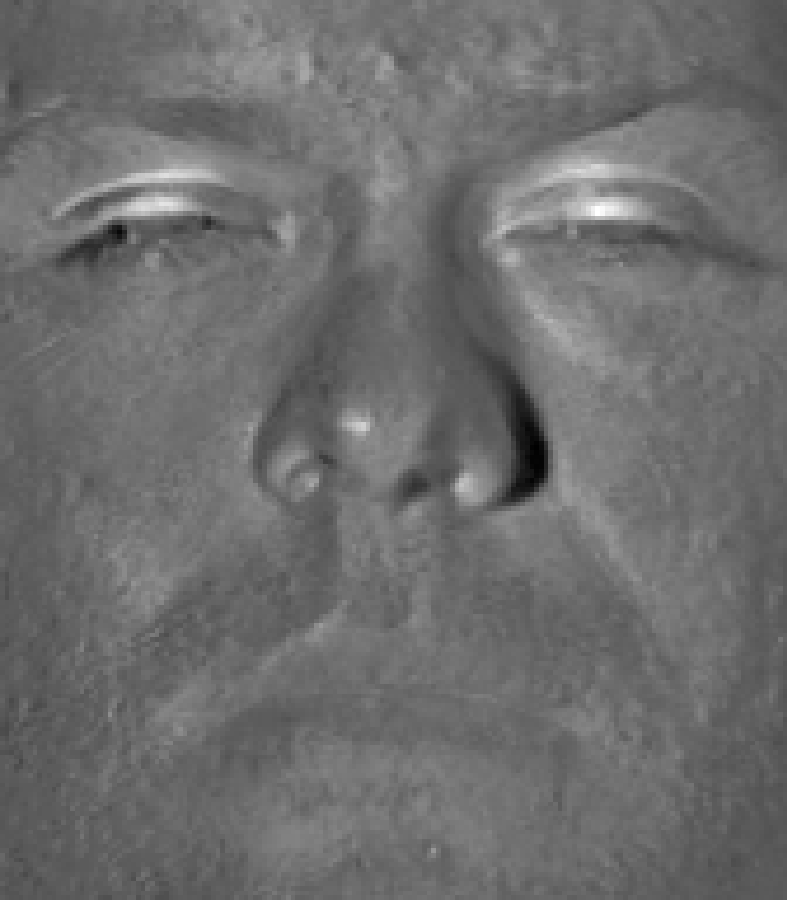} \\ \hline
  \textbf{Yale B01} & rank = $5$ & $10$ & FG = $0.24$ & $0.05$ \\ \hline
  \includegraphics[scale=0.15]{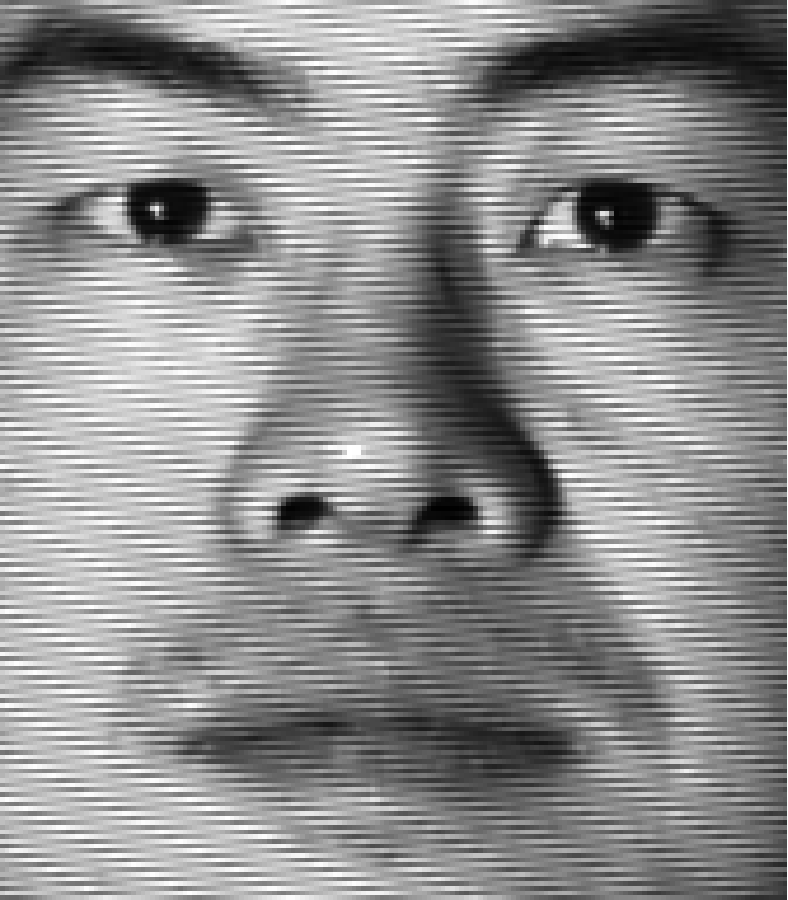} & \includegraphics[scale=0.15]{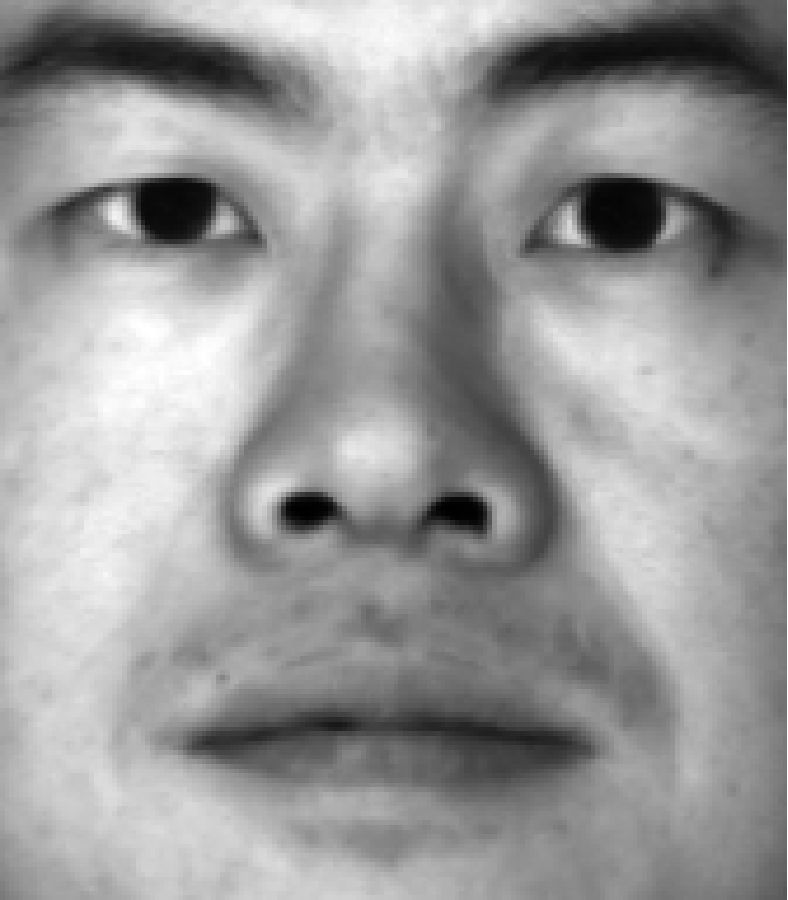} & \includegraphics[scale=0.15]{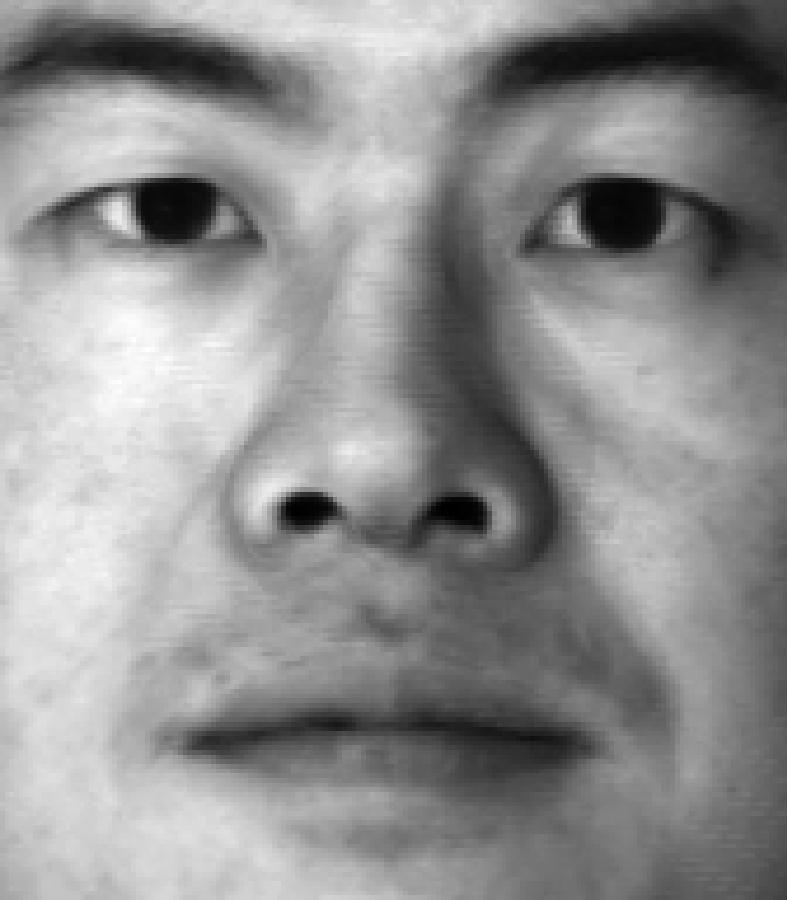} & \includegraphics[scale=0.15]{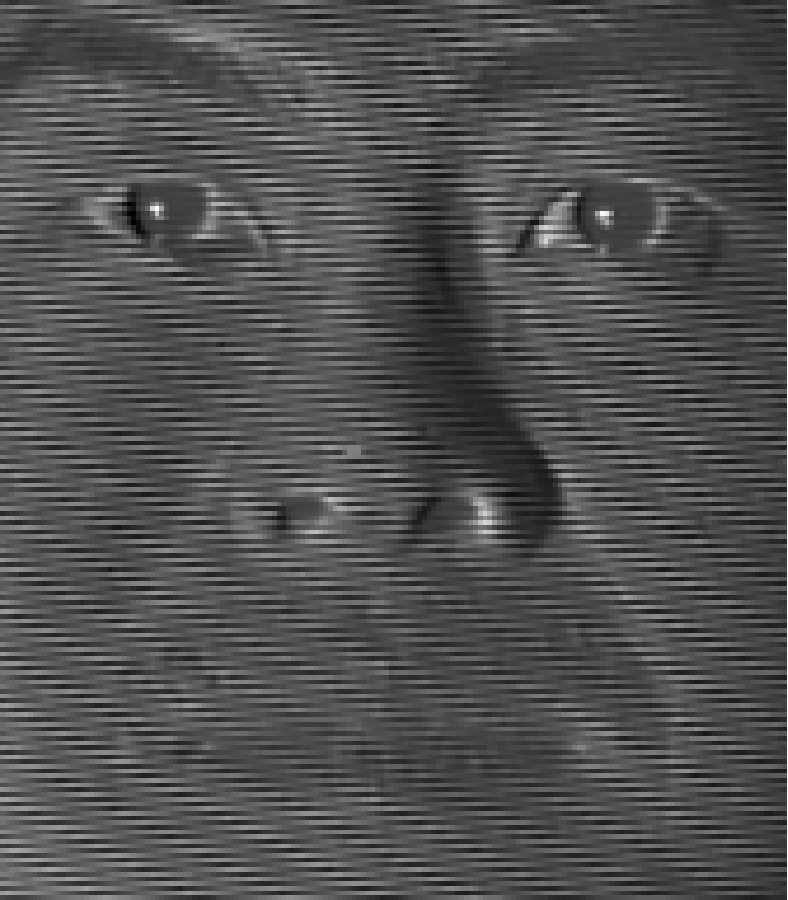} & \includegraphics[scale=0.15]{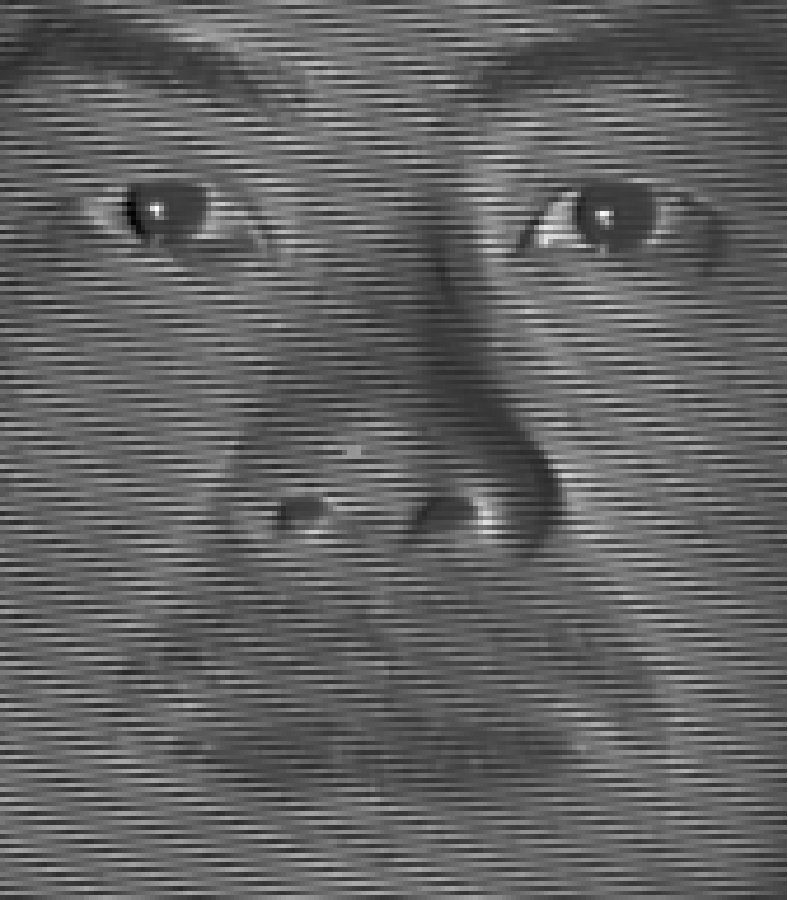} \\ \hline
  \textbf{Yale B02} & rank = $5$ & $10$ & FG = $0.24$ & $0.05$ \\ \hline
  \includegraphics[scale=0.15]{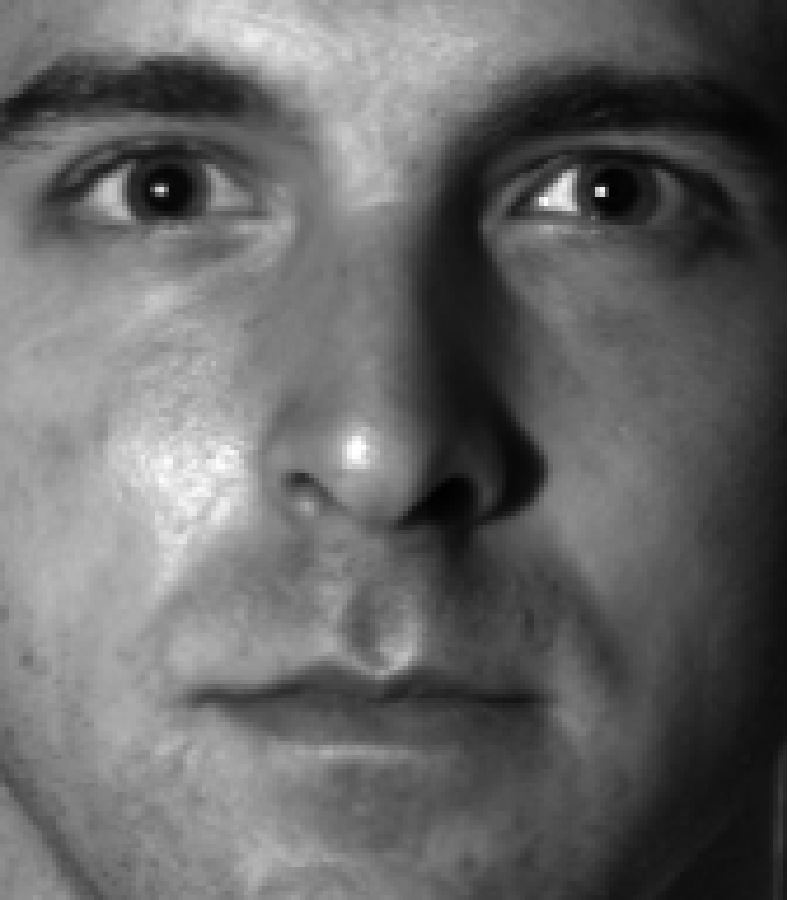} & \includegraphics[scale=0.15]{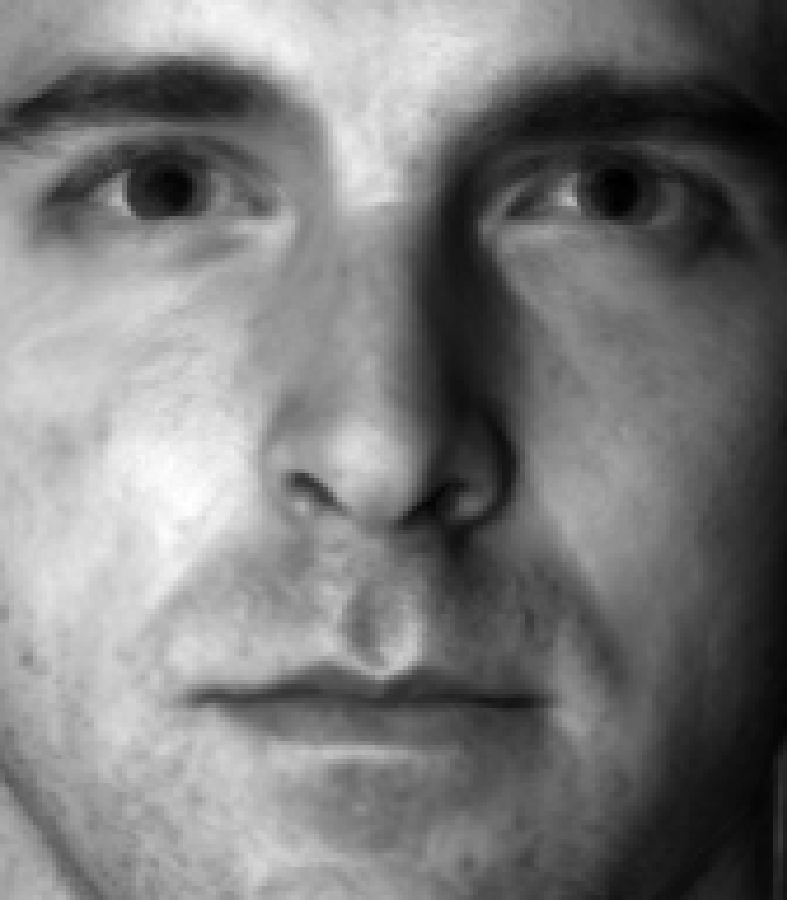} & \includegraphics[scale=0.15]{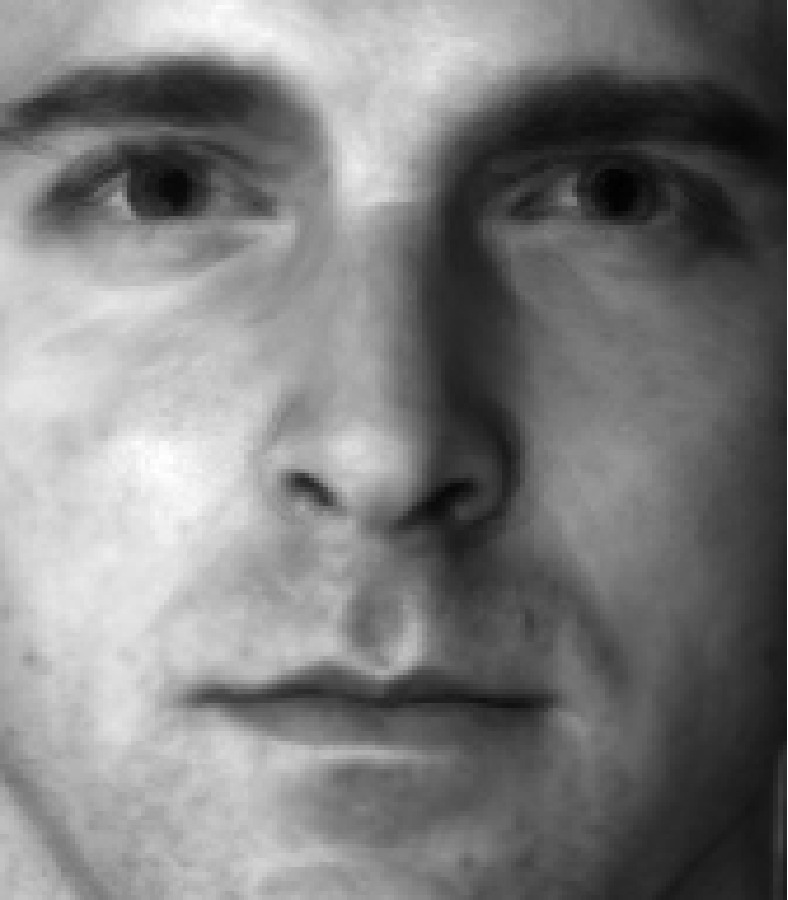} & \includegraphics[scale=0.15]{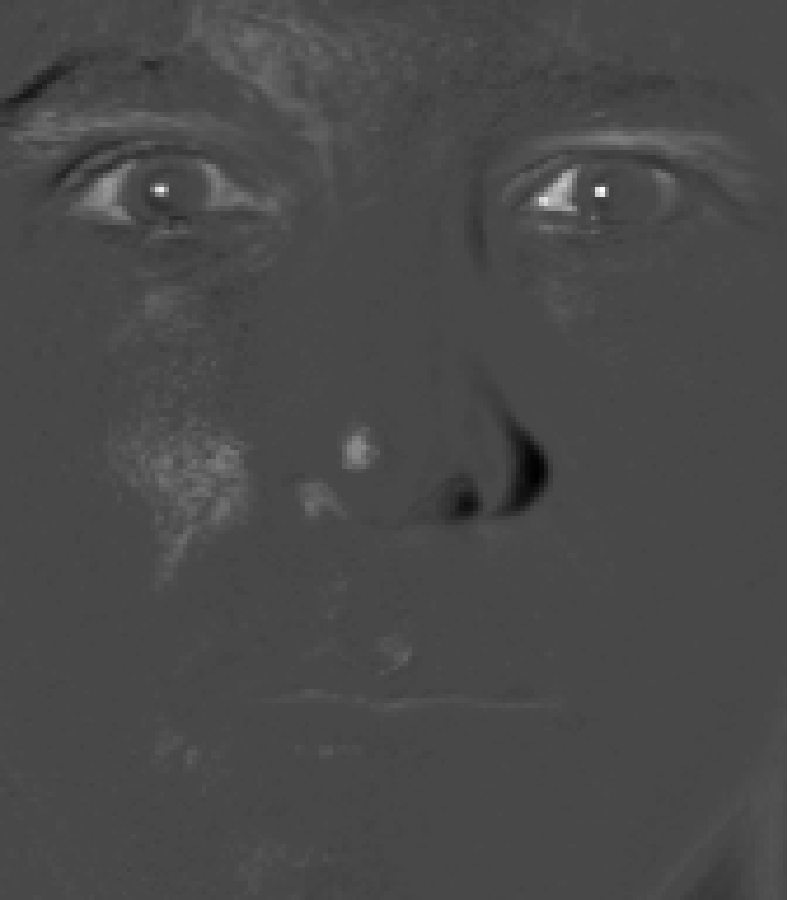} & \includegraphics[scale=0.15]{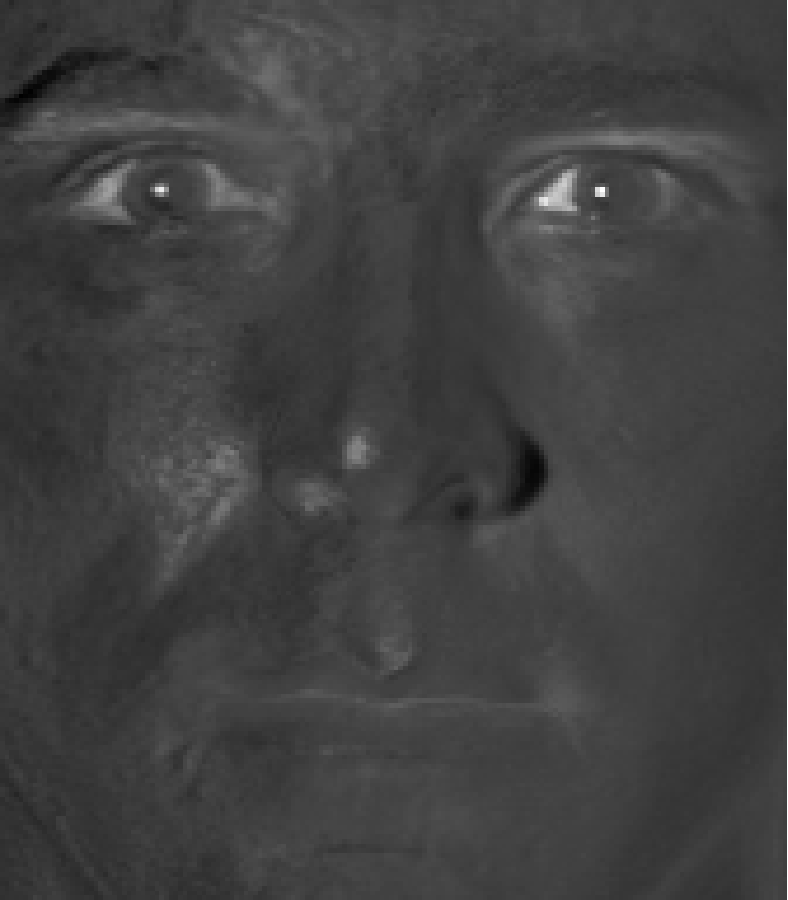} \\ \hline
  \textbf{Yale B10} & rank = $3$ & $10$ & FG = $0.24$ & $0.05$ \\ \hline
  \end{tabular}
\caption{Examples from solving facial shadow removal problems via IALM and  ML-IALM algorithms on cropped \textbf{Yale B} database ($ 96\times 84\times 2414$). We run both IALM and ML-IALM for fixed five seconds. With each image we also report the respective rank of the low rank component and the feasibility gap (FG): $\Vert \mathbf{D}-\mathbf{L}-\mathbf{S}\Vert_F / \Vert\mathbf{D}\Vert_F$.}
\label{fig:exp-face-ialm}
\end{figure*}

A brief examination of the Table \ref{fig:exp-face-ialm} reveals that ML-IALM produces much better separation for each subject, with a more accurate rank $10$ of the low-rank component and a five times smaller feasibility gap.

For the shadow removal problem as well, we test the Frank-Wolfe methods on the noisy data with $75\%$ contaminated entries. Both FWT and ML-FWT run until convergence with $10^{-3}$ tolerance and record CPU times (seconds) and the achieved rank of the low-rank component and sparsity of the sparse component. The results of all $35$ subjects are reported in Table \ref{t:exp-face-fwt}. In all experiments we used up to $4$ levels of coarse models. Here in all experiments both methods achieved similar objective values and sparsity values, as expected. However ML-FWT is not only twice faster, but it also achieves a much better rank of the low-rank component. In fact, FWT fails to solve the problem, since the low-rank component is essentially full rank.

\begin{table}
\centering
\begin{tabular}{|c|c|c|c|c|c|c|c|c|}
\hline
\multirow{2}{*}{problem} & \multicolumn{4}{|c|}{FWT} & \multicolumn{4}{|c|}{ML-FWT} \\
\cline{2-9}
& sec & $f^\star$ & $\rank(\mathbf{L}^\star)$ & $\sparsity(\mathbf{S}^\star)$ & sec & $f^\star$ & $\rank(\mathbf{L}^\star)$ & $\sparsity(\mathbf{S}^\star)$ \\
\hline 	
yaleB01 & 83 & 0.001 & 65 & 0.71 & 41 & 0.001 & 8 & 0.73 \\
\hline
yaleB02 & 71 & 0.00098 & 65 & 0.71 & 39 & 0.00098 & 8 & 0.74 \\
\hline
yaleB03 & 75 & 0.00099 & 65 & 0.71 & 40 & 0.00099 & 8 & 0.73 \\
\hline
yaleB04 & 72 & 0.001 & 65 & 0.72 & 42 & 0.001 & 8 & 0.73 \\
\hline
yaleB05 & 61 & 0.00096 & 65 & 0.7 & 33 & 0.00094 & 8 & 0.73 \\
\hline
yaleB06 & 77 & 0.001 & 65 & 0.72 & 48 & 0.001 & 8 & 0.73 \\
\hline
yaleB07 & 78 & 0.001 & 65 & 0.72 & 42 & 0.00099 & 8 & 0.73 \\
\hline
yaleB08 & 75 & 0.00099 & 65 & 0.71 & 44 & 0.00098 & 8 & 0.73 \\
\hline
yaleB09 & 84 & 0.001 & 65 & 0.71 & 48 & 0.00097 & 8 & 0.73 \\
\hline
yaleB10 & 67 & 0.00099 & 65 & 0.72 & 45 & 0.001 & 8 & 0.73 \\
\hline
yaleB11 & 75 & 0.00098 & 60 & 0.71 & 41 & 0.00097 & 7 & 0.73 \\
\hline
yaleB12 & 79 & 0.00099 & 59 & 0.71 & 41 & 0.00097 & 7 & 0.73 \\
\hline
yaleB13 & 60 & 0.00096 & 60 & 0.71 & 35 & 0.00097 & 7 & 0.74 \\
\hline
yaleB15 & 85 & 0.001 & 63 & 0.71 & 43 & 0.00097 & 8 & 0.73 \\
\hline
yaleB16 & 77 & 0.00098 & 62 & 0.7 & 47 & 0.001 & 7 & 0.73 \\
\hline
yaleB17 & 66 & 0.00098 & 63 & 0.71 & 42 & 0.00099 & 8 & 0.73 \\
\hline
yaleB18 & 85 & 0.001 & 63 & 0.71 & 45 & 0.00099 & 8 & 0.73 \\
\hline
yaleB19 & 77 & 0.001 & 64 & 0.71 & 48 & 0.001 & 8 & 0.73 \\
\hline
yaleB20 & 73 & 0.00099 & 64 & 0.7 & 43 & 0.001 & 8 & 0.73 \\
\hline
yaleB21 & 75 & 0.00098 & 64 & 0.7 & 43 & 0.00099 & 8 & 0.73 \\
\hline
yaleB22 & 88 & 0.00098 & 64 & 0.69 & 44 & 0.00093 & 8 & 0.72 \\
\hline
yaleB23 & 63 & 0.00098 & 64 & 0.71 & 45 & 0.001 & 8 & 0.73 \\
\hline
yaleB24 & 74 & 0.00099 & 64 & 0.71 & 49 & 0.001 & 8 & 0.73 \\
\hline
yaleB25 & 75 & 0.00099 & 64 & 0.7 & 48 & 0.001 & 8 & 0.73 \\
\hline
yaleB26 & 77 & 0.001 & 64 & 0.71 & 45 & 0.00099 & 8 & 0.73 \\
\hline
yaleB27 & 61 & 0.00097 & 64 & 0.69 & 40 & 0.00096 & 8 & 0.72 \\
\hline
yaleB28 & 61 & 0.00096 & 64 & 0.69 & 40 & 0.00097 & 8 & 0.73 \\
\hline
yaleB29 & 71 & 0.00099 & 64 & 0.7 & 44 & 0.00099 & 8 & 0.73 \\
\hline
yaleB30 & 72 & 0.00098 & 64 & 0.71 & 43 & 0.00099 & 8 & 0.73 \\
\hline
yaleB31 & 78 & 0.001 & 64 & 0.71 & 48 & 0.001 & 8 & 0.73 \\
\hline
yaleB32 & 60 & 0.00096 & 64 & 0.7 & 37 & 0.00097 & 8 & 0.72 \\
\hline
yaleB33 & 79 & 0.00099 & 64 & 0.7 & 44 & 0.00099 & 8 & 0.73 \\
\hline
yaleB34 & 78 & 0.00099 & 64 & 0.7 & 37 & 0.00097 & 8 & 0.73 \\
\hline
yaleB35 & 75 & 0.00099 & 64 & 0.7 & 45 & 0.00099 & 8 & 0.73 \\
\hline
yaleB36 & 72 & 0.00099 & 64 & 0.71 & 42 & 0.00098 & 8 & 0.73 \\
\hline
\end{tabular}
\caption{CPU times (in seconds) after solving shadow removal problems up to a fixed tolerance using the standard Frank-Wolfe Thresholding algorithm and its multilevel variant. For all experiments we used 2 levels for the multilevel algorithm.}
\label{t:exp-face-fwt}
\end{table}

\section{Conclusion}
In this paper, we presented two multilevel algorithms for solving problems modelled as robust principle component analysis or compressive principle component pursuit optimisation problems. The first algorithm is a multilevel variant of the well-known inexact augmented Lagrange method (or more generally, ADMM), called ML-IALM. We proved that ML-IALM converges to an approximate solution, with approximation error being small for many computer vision problems, including those studied here. To the best of our knowledge this is the first time when an ADMM with approximate steps was proven to converge.
Our second algorithm is a multilevel variant of the well known Frank-Wolfe method modified to be most efficient for CPCP problems. We showed that this multilevel algorithm also converges to the solution of the CPCP problem with the same rate as its standard counterpart, while having much lower per iteration complexity. We tested both methods on various synthetic and real life problems. The results clearly show that the multilevel algorithms are not only several times faster (especially on larger problems), but also can often solve problems that their standard counterparts cannot.

{
\bibliographystyle{plain}
\bibliography{main}
}

\end{document}